\documentclass[aos, preprint]{imsart}
\usepackage{amsmath, amssymb,amsthm, color, relsize, enumerate, enumitem, mathtools}
\RequirePackage[colorlinks,citecolor=blue,urlcolor=blue]{hyperref}
\usepackage{latexsym}
\usepackage{xparse}
\usepackage{array}
\usepackage{xstring}
\usepackage{indentfirst}
\usepackage{bm,dsfont} 
\usepackage{graphicx,adjustbox}
\usepackage{chngcntr}
\usepackage{import}
\usepackage{comment}
\usepackage{amsfonts}
\usepackage{lipsum}
\usepackage[OT1]{fontenc}
\usepackage[utf8]{inputenc}
\usepackage{mathtools}  
\usepackage{multirow,multicol}
\usepackage{subfig}
\usepackage[english]{babel}
\numberwithin{equation}{section}
\usepackage[toc,page]{appendix}
\usepackage{pgfplots, pdfpages}
\pgfplotsset{compat=newest} 
\usepackage{tikzscale}
\pgfplotsset{plot coordinates/math parser=false}

\renewcommand{\abstractname}{} \addto\captionsenglish{\renewcommand{\abstractname}{}}

\startlocaldefs

\DeclareMathOperator{\tr}{tr}

\theoremstyle{definition}
\newtheorem{defn}{Definition}[section]

\newtheorem{prop}{Proposition}[section]

\newtheorem{lem}{Lemma}[section]
\newtheorem{thm}{Theorem}[section]
\newtheorem{cor}{Corollary}[section]
\newtheorem{rem}{Remark}[section]
\newtheorem{assu}{Assumption}[section]

\newtheorem{clawithinpf}{Claim}
\makeatletter
\@addtoreset{clawithinpf}{proof}
\makeatother

\makeatletter
\newcommand{\thickhline}{%
	\noalign {\ifnum 0=`}\fi \hrule height 1pt
	\futurelet \reserved@a \@xhline
}

\newcommand{\suchthat}{\mathop{\mathrm{s.t.}}}

\newcommand\independent{\protect\mathpalette{\protect\independenT}{\perp}}
\def\independenT#1#2{\mathrel{\rlap{$#1#2$}\mkern2mu{#1#2}}}

\newcommand{\bb}[1]{\mathbb{#1}} 
\newcommand{\bbM}[1]{\mathds{#1}} 
\newcommand{\cc}[1]{\mathcal{#1}} 
 
\newcommand{\zero}{\boldsymbol{0}}
\newcommand{\TimeDepMatrix}[2]{\boldsymbol{#1}^{\paren{#2}}}

\newcommand{\FA}{\mathop{\mathrm{FA}}}
\newcommand{\vect}{\mathop{\mathrm{vec}}}
\newcommand{\MD}{\mathop{\mathrm{MD}}}

\newcommand{\diag}{\mathop{\mathrm{diag}}}

\newcommand{\supp}{\mathop{\mathrm{supp}}}

\newcommand{\cov}{\mathop{\mathrm{cov}}} 
 
\newcommand{\var}{\mathop{\mathrm{var}}}

\newcommand{\brac}[1]{\left[#1\right]}
\newcommand{\set}[1]{\left\{#1\right\}}
\newcommand{\abs}[1]{\left\lvert #1 \right\rvert}
\newcommand{\paren}[1]{\left(#1\right)}
\newcommand{\parbra}[1]{\left(#1\right]}
\newcommand{\brapar}[1]{\left[#1\right)}
\newcommand{\Bigparen}[1]{\Big(#1\Big)}
\newcommand{\Bigbrac}[1]{\Big[#1\Big]}

\newcommand{\InnerProd}[2]{\langle #1,#2 \rangle}

\newcommand{\cp}[1]{\overset{#1}{\rightarrow}}    
\newcommand{\eqd}{\overset{d}{=}}  
\newcommand{\RelNum}[2]{\overset{#1}{#2}}

\newcommand{\LpNorm}[2]{\left\| #1\right\|_{\ell_{#2}}}

\newcommand{\OpNorm}[3]{\left\| #1\right\|_{#2\rightarrow#3}}
\newcommand{\RNum}[1]{\uppercase\expandafter{\romannumeral #1\relax}}

\endlocaldefs

\begin{document}

\begin{frontmatter}
		
\title{Online detection of local abrupt changes in high-dimensional Gaussian graphical models}
\runtitle{Online detection of local abrupt changes in high-dimensional GGMs}
	
\begin{aug}
\author[A]{\fnms{Hossein} \snm{Keshavarz}\ead[label=e1]{hossein.keshavarzshenastaghi@gm.com}}
\and
\author[B]{\fnms{George} \snm{Michailidis}\ead[label=e2]{gmichail@ufl.edu}}

\address[A]{Advanced Analytics Research CoE, General Motors, \printead{e1}}
\address[B]{Department of Statistics \& UF Informatics Institute, University of Florida, \printead{e2}}

\end{aug}

\begin{abstract}
\renewcommand{\abstractname}{}  
The problem of identifying change points in high-dimensional Gaussian graphical models (GGMs) in an \emph{online} fashion is of interest, due to new applications in biology, economics and social sciences. The offline version of the problem, where all the data are a priori available, has led to a number of methods and associated algorithms involving regularized loss functions. However, for the online version, there is currently only a single work in the literature that develops a sequential testing procedure and also studies its asymptotic false alarm probability and power. The latter test is best suited for the detection of change points driven by global changes in the structure of the precision matrix of the GGM, in the sense that many edges are involved. Nevertheless, in many practical settings the change point is driven by local changes, in the sense that only a small number of edges exhibit changes. To that end, we develop a novel test to address this problem that is based on the $\ell_\infty$ norm of the normalized covariance matrix of an appropriately selected portion of incoming data. The study of the asymptotic distribution of the proposed test statistic under the null (no presence of a change point) and the alternative (presence of a change point) hypotheses requires new technical tools that examine maxima of graph-dependent Gaussian random variables, and that of independent interest. It is further shown that these tools lead to the imposition of mild regularity conditions for key model parameters, instead of more stringent ones required by leveraging previously used tools in related problems in the literature. Numerical work on synthetic data illustrates the good performance of the proposed detection procedure both in terms of computational and statistical efficiency across numerous experimental settings.
\end{abstract}



\end{frontmatter}

\section{Introduction}\label{Intro}

Learning the dependence between variables in high-dimensional data represents an important learning task in macro-econometrics \cite{stock2006forecasting, stock2008evolution}, in gaining insights into regulatory mechanisms in biology  \cite{michailidis2013autoregressive}, and optimizing large-scale portfolio allocations in finance \cite{garleanu2013dynamic}. The conditional dependence between components in multivariate observations can be modeled using graphical models \cite{wainwright2008graphical}. However, the presence of more variables than available observations (the so-called high-dimensional scaling regime) led to the study of estimating such models under the \emph{sparsity} assumption, namely that most pairs of variables are conditionally independent given the remaining ones. For multivariate Gaussian observations, this assumption translates to sparsity in the inverse covariance (precision) matrix.  As a result, a rich body of literature including fast, scalable algorithms together with their theoretical guarantees has been developed for estimating sparse precision matrices from independent and identically distributed (i.i.d.) observations (see \cite{wainwright2019high} and references therein).

In applications, where the graphical models are estimated based on time course data, the  stationarity assumption may be too stringent. For example, there is strong evidence for changing conditional dependence patterns amongst brain regions \cite{hutchison2013dynamic, hindriks2016can}, or the stock returns of financial firms \cite{lin2017regularized}. A simple, yet useful in many applied settings, departure is that of piecewise stationarity that in turn implies that the conditional dependence structure of the data remains constant between consecutive break points that define the stationary segments. In this case, the estimation problem gets more involved, since one needs to both identify the break/change points, as well as estimate the parameters of the underlying graphical models.

There are two streams of change point detection problems in the literature: the \emph{offline} one and \emph{sequential (online)} one. In the first stream, the data under consideration are available and the analytical task is to identify change points (detection) and estimate the parameters of the model employed, so that insights of what led to the occurrence of a change point are obtained (diagnosis). In the online framework, data are acquired in a sequential manner and one is interested in detecting a change point with small delay. Both versions of the problem have been extensively studied in the literature for various univariate and multivariate statistical models (see, e.g., \cite{basseville1993detection, horvath2014extensions, keshavarz2017optimal} and references therein). However, the literature on change point detection for high-dimensional graphical models is significantly sparser and related work is fairly recent. Next, we provide a brief review of the literature for both versions of the problem.

\begin{enumerate}[label = (\alph*),leftmargin=*]
\item (Offline methods) A number of methods re-parameterize the piecewise Gaussian Graphical Model (GGM), and introduce new parameters that correspond to differences of the model parameters at every point in time, which are subsequently regularized based on a fused lasso penalty \cite{kolar2010estimating,kolar2012estimating,gibberd2017multiple}. Hence, the non-zero set of parameters corresponds to candidate change points. An analogous strategy is adopted in Safikhani et al. \cite{safikhani2017joint} for high dimensional vector autoregressive models.
Roy et al. \cite{roy2016change} studied detection of a single change point in high dimensional sparse Markov random fields based on an exhaustive search. To reduce the cost of exhaustive search, Atchad\'{e} and Bybee \cite{atchade2017scalable} proposed an approximate majorize-minimize (MM) algorithm for GGMs. Another stream of literature focused on testing for the presence of a change point in GGMs, rather than assuming their presence and estimating them. To that end,  \cite{cai2013two,li2012two} developed two sample tests for detecting differences in covariance matrices (rather than precision matrices), but a number of techniques developed also prove useful for the problem at hand; such tests can be operationalized for detecting the presence of a change point. Avanesove et al. \cite{avanesov2018change} proposed a test statistic, based on the de-sparsified regularized estimator of the precision matrix \cite{jankova2015confidence}, for the same task. They latter paper employed also bootstrap sampling for computing the critical values of the proposed test.
\item (Online methods) The only available work is that of Keshavarz et al. \cite{keshavarz2018sequential} that introduces an online algorithm for detecting abrupt changes in the precision matrix of sparse GGMs. Further, the proposed test statistic is shown to be asymptotically Gaussian, thus providing a closed-form expression for the critical value of the test. 
\end{enumerate}

Note that the test in \cite{keshavarz2018sequential} was designed for settings where many of the entries in the precision matrix change, thus contributing to the occurrence of a change point. However, in many applications the changes in the precision matrix may be few, in which case the former test will lack power. As an example, and using the terminology of graphical models wherein variables correspond to nodes in the underlying graph and edges capture conditional dependence relationships, consider a setting where only a few edges possibly related with a single node change; in the case of studying dependencies amongst stock returns of firms, suppose that a change is confined to those of a particular economic sector. Motivated by such examples, the goal of this paper is to develop a test and the
associated sequential change point detection algorithm for sparse GGMs, suitable for settings driven by changes affecting few edges. 

Hence, the key contributions of the paper are: \\
\noindent
(i) development of a sequential algorithm for quickest detection of changes in the precision matrix of sparse high-dimensional GGM, driven by changes in few edges. The subsequent asymptotic technical analysis calibrates both the false alarm probability of the statistical test at the heart of the detection algorithm, as well as its power. We also discuss how to operationalize the test based on plug-in quantities obtained from the data. \\
\noindent
(ii) The proposed test statistic is based on the $\ell_\infty$ norm of a standardized Wishart matrix and its asymptotic analysis heavily relies on extreme value theory. In contrast to the work in \cite{avanesov2018change}, finding an exact formulation of the critical value of our proposed test does not require bootstrap sampling, thus making the detection algorithm computationally inexpensive and suitable for applications where new data come at high frequency. Instead, our analysis leverages results developed by Galambos \cite{ galambos1988variants, galambos1972distribution} on the exact distribution of the maximum of dependent variables, which in turn provide us with a closed-form formulation of the critical value in terms of the number of nodes in the GGM and false alarm rate. Note that the sparsity assumption for the underlying GGM plays a key role for allowing us to control the correlation between the entries of the Wishart matrix appearing in the proposed test statistic, which in turn allows us to leverage the results in \cite{galambos1988variants, galambos1972distribution}. 

Our rigorously developed novel techniques are of independent interest for studying non-asymptotic properties of $\ell_\infty$ norms of random matrices. Further, note that the strategy used in \cite{avanesov2018change, cai2013two} on the detection delay based on a Gaussian approximation of the maximum of centered empirical processes \cite{chernozhukov2014gaussian} leads to an exceedingly stringent condition, which is likely not to hold in applications. On the other hand, our novel techniques lead to a mild condition that make the detection algorithm operational and suitable for real data.

The remainder of the paper is organized as follows: Section \ref{Methodology} is devoted to formulating online change point detection problem in GGMs, as well as presenting the proposed detection algorithm for both oracle and data-driven scenarios. Section \ref{Section3} is reserved for studying the asymptotic properties of the proposed oracle test under both null (no presence of a change point) and alternative hypotheses. In Section \ref{Section4}, we investigate the asymptotic properties of our algorithm in more realistic data-driven setting. In Section \ref{Synthetic}, we numerically gauge the performance of our proposed algorithm. Section \ref{Discussion} serves as the conclusion. We prove the main results of the paper in Section \ref{Proofs}. Lastly, Appendices \ref{AppendixA} and \ref{AppendixB} contain auxiliary technicalities which are essential for the results in Section \ref{Proofs}.

\subsection{Notation}

Boldface symbols denote vectors and matrices. $\bbM{1}\paren{\cdot}$, $\wedge$ and $\vee$ denote the indicator function, minimum and maximum operators, respectively. $\bb{R}^m_+$ is a compact way of representing $\brapar{0,\infty}^m$. We use $\boldsymbol{I}_m$, $\zero_m$ and $\bbM{1}_m$ to denote the $m\times m$ identity matrix, all zeros column vector of length $m$, and all ones column vector of $m$ entries, respectively. $S^{p\times p}_{++}$ denotes the space of strictly positive definite $p\times p$ matrices. For $i,j\in\set{1,\ldots,p}$, $\boldsymbol{M}_{i,:}$, $\boldsymbol{M}_{:,j}$ and $M_{ij}$ represent the $i-\mbox{th}$ row, $j-\mbox{th}$ column and $\paren{i,j}$-entry of $\boldsymbol{M}$. $\diag\paren{\boldsymbol{M}}$ refers to the main diagonal entries of $\boldsymbol{M}$ and  $\boldsymbol{M}_{\cc{S}} \coloneqq \brac{M_{ij}\bbM{1}\paren{\paren{i,j}\in\cc{S}}}_{i,j}$ for any set $\cc{S}$. For matrices of the same size $\boldsymbol{M}$ and $\boldsymbol{M'}$, $\InnerProd{\boldsymbol{M}}{\boldsymbol{M'}}{}\coloneqq \sum_{i,j} M_{ij}M'_{ij}$ denotes their usual inner product. We also use $\boldsymbol{M}\circ \boldsymbol{M'}$ for denoting the Hadamard product of $\boldsymbol{M}$ and $\boldsymbol{M'}$, defined by $\paren{\boldsymbol{M}\circ \boldsymbol{M'}}_{ij} = M_{ij}M'_{ij}$. We use the following norms on  matrix $\boldsymbol{M}$. For any $1\leq p\leq\infty$, $\LpNorm{\boldsymbol{M}}{p}$ stands for element-wise $\ell_p$-norm defined by $\LpNorm{\boldsymbol{M}}{p}^p \coloneqq \sum_{i,j} \abs{M_{ij}}^p$. $\OpNorm{\boldsymbol{M}}{p}{q}$ refers to $\ell^p \mapsto \ell^q$ operator norm given by $\OpNorm{\boldsymbol{M}}{p}{q} = \max \brac{\LpNorm{\boldsymbol{M}v}{q} s.t. \LpNorm{v}{p}=1}$. We write $\eqd$ for denoting the equality in distribution. For non-negative sequences $\set{a_m}$ and $\set{b_m}$, we write $a_m \lesssim b_m$, if there exists a bounded positive scalar $C_{\max}$ (depending on model parameters) such that $\limsup_{m\rightarrow\infty} a_m/b_m \leq C_{\max}$. Further, $a_m \asymp b_m$ refers to the case that $a_m \lesssim b_m$ and $a_m \gtrsim b_m$. For a non-negative deterministic $\set{a_m}$ and random sequence $\set{b_m}$, we write $b_m = \cc{O}_{\bb{P}}\paren{a_m}$, if $\bb{P}\paren{b_m\leq C_{\max} a_m} \rightarrow 1$, as $m\rightarrow\infty$, for a bounded scalar $C_{\max}$ (which may depend on model parameters). Lastly for a binary test statistic $\Xi$, the false alarm and mis-detection probabilities are respectively defined by
\begin{equation*}
\bb{P}_{\FA}\paren{\Xi}\coloneqq \bb{P}\paren{\Xi = 1 \mid \bb{H}_0},\quad\mbox{and}\quad \bb{P}_{\MD}\paren{\Xi}\coloneqq \bb{P}\paren{\Xi = 0 \mid \bb{H}_1}.
\end{equation*} 

\section{Problem Formulation}\label{Methodology}

We focus on a time-varying GGM with vertex (variable) set $\cc{V}=\set{1,\ldots, p}$. We observe $\boldsymbol{X}_t\in\bb{R}^p$ as a realization of a zero-mean GGM $\cc{G}_t = \paren{\cc{V}, \cc{E}_t}$ at time $t\in\bb{N}$. Specifically, $\boldsymbol{X}_t$ is a centered Gaussian vector whose density function is given by
\begin{equation*}
g\paren{\boldsymbol{x}, \TimeDepMatrix{\Omega}{t}} = \paren{2\pi}^{-p/2}\sqrt{\det \TimeDepMatrix{\Omega}{t}} \exp\paren{-\frac{\boldsymbol{x}^\top \TimeDepMatrix{\Omega}{t} \boldsymbol{x}}{2}},\quad\forall\;\boldsymbol{x}\in\bb{R}^{p},
\end{equation*}
where $\TimeDepMatrix{\Omega}{t}\in S^{p\times p}_{++}$ denotes the precision matrix of $\boldsymbol{X}_t$. Note that $\cc{G}_t$ can be equivalently represented by   $\TimeDepMatrix{\Omega}{t}$, as $\cc{E}_t = \set{\paren{r,s}\in\cc{V}\times\cc{V}: r\ne s\,\;\mbox{and}\; \TimeDepMatrix{\Omega}{t}_{rs} \ne 0}$.

A change point exists at time $t^\star$, if $\TimeDepMatrix{\Omega}{t^\star}$ switches to a new configuration at $t^\star+1$; i.e., $\TimeDepMatrix{\Omega}{t^\star} \ne \TimeDepMatrix{\Omega}{t^\star+1}$. Throughout this manuscript, $t^\star$ and $t^\star+1$ are respectively referred to as \emph{pre-change} and \emph{post-change} regimes. We use $\cc{C}^\star = \set{t^\star_0,t^\star_1,\ldots}$, sorted in ascending
order and with $t^\star_0 = 0$, to denote the location of all abrupt changes in $\set{\cc{G}_t}_{t\in\bb{N}}$. So, for consecutive change points $t^\star_j$ and $ t^\star_{j+1}$, $\boldsymbol{X}_t:\; t^\star_j < t \leq t^\star_{j+1}$ are i.i.d. centered Gaussian random vectors.

To determine whether a structural change occurs at time $t$, we use the following hypothesis testing problem.
\begin{equation}\label{HypothesisTestingFormulation}
\bb{H}_{0,t}:\;t\notin\cc{C}^\star\quad\mbox{versus}\quad \bb{H}_{1,t}:\;t\in\cc{C}^\star
\end{equation} 
Gathering adequate information about the post-change framework is necessary for distinguishing between $\bb{H}_{0,t}$ and $\bb{H}_{1,t}$, especially for high-dimensional objects such as GGMs. Strictly speaking, any sequential decision function $\Xi_t\in\set{0,1}$ flags an abrupt change at $t$ (rejecting $\bb{H}_{0,t}$ in Eq. \eqref{HypothesisTestingFormulation}) after observing $w$ samples from the potential new regime, i.e. $\boldsymbol{X}_{t+1},\ldots,\boldsymbol{X}_{t+w}$ for some $w\in\bb{N}$. Hence, $\Xi_t$ is a function of $\boldsymbol{X}_{t^\star},\ldots,\boldsymbol{X}_{t+1},\ldots,\boldsymbol{X}_{t+w}$, where $w$ and $t^\star$ denote the detection delay and the location of the last detected abrupt change, respectively. Our objective is to design a detection procedure for hypothesis testing problem \eqref{HypothesisTestingFormulation}, whose false alarm rate is controlled below some pre-specified rate $\pi_0\in\paren{0,1}$, that in addition exhibits a small mis-detection rate and short delay.

\subsection{Detection algorithm: oracle setting}\label{Section2.1}

Next, we introduce a novel online procedure for solving hypothesis testing problem \eqref{HypothesisTestingFormulation} with delay $w$. We begin by focusing on the oracle case in which the pre-change precision matrix $\TimeDepMatrix{\Omega}{t}$ is fully known. Throughout this paper, we assume that $\boldsymbol{X}_{t+1},\ldots, \boldsymbol{X}_{t+w}$ are observed prior to deciding whether $t\in\cc{C}^\star$. For ease of presentation, we also assume that no change point occurs between $\paren{t+1}$ and $\paren{t+w-1}$. Thus, $\boldsymbol{X}_{t+1},\ldots, \boldsymbol{X}_{t+w}$ are independent draws from the multivariate Gaussian distribution $\cc{N}\paren{\zero_p, \TimeDepMatrix{\Omega}{t+1}}$. This assumption will be relaxed in the next section.

Consider the transformed vectors $\boldsymbol{Y}_{t+r} = \TimeDepMatrix{\Omega}{t}\boldsymbol{X}_{t+r}$ for any $r=1,\ldots, w$. If $t\notin\cc{C}^\star$, the random variables $\set{\boldsymbol{Y}_{t+r}:\;r=1,\ldots, w}$ are i.i.d. centered Gaussian vectors with covariance matrix $\TimeDepMatrix{\Omega}{t}$. Note that an abrupt change in the structure of $\cc{G}_t$ can be translated to an abrupt change in the covariance matrix of $\set{\boldsymbol{Y}_{t+r}:\;r=1,\ldots, w}$. The simplicity of working with the covariance matrix (instead of its inverse) is the major benefit of using the transformed samples. Define $\boldsymbol{S}_{t, w}$ by
\begin{equation*}
\boldsymbol{S}_{t, w} \coloneqq \frac{1}{w}\sum_{r=1}^{w} \boldsymbol{Y}_{t+r}\boldsymbol{Y}_{t+r}^\top.
\end{equation*}
Under $\bb{H}_{0,t}$, $\boldsymbol{S}_{t, w}$ is an unbiased estimate of $\TimeDepMatrix{\Omega}{t}$. In contrast when $t\in\cc{C}^\star$, i.e. $\TimeDepMatrix{\Omega}{t} \ne \TimeDepMatrix{\Omega}{t+1}$, then the expected value of $\boldsymbol{S}_{t, w}$ is given by
\begin{equation*}
\bb{E}\paren{\boldsymbol{S}_{t, w} \mid \bb{H}_{1,t}} = \TimeDepMatrix{\Omega}{t}\paren{\TimeDepMatrix{\Omega}{t+1}}^{-1}\TimeDepMatrix{\Omega}{t}= \TimeDepMatrix{\Omega}{t} +  \TimeDepMatrix{\Omega}{t}\brac{\paren{\TimeDepMatrix{\Omega}{t+1}}^{-1}-\paren{\TimeDepMatrix{\Omega}{t}}^{-1}} \TimeDepMatrix{\Omega}{t}.
\end{equation*}

This identity suggests that studying the behavior of a suitable norm of a standardized version of $\boldsymbol{S}_{t, w}$ can be helpful for distinguishing between the null and alternative hypotheses for decision problem \eqref{HypothesisTestingFormulation}. Let $\boldsymbol{E}_{t, w}$ denote the standardized version of $\boldsymbol{S}_{t, w}$, i.e.,
\begin{equation*}
\boldsymbol{E}_{t, w} = \paren{\boldsymbol{S}_{t, w}-\TimeDepMatrix{\Omega}{t}} \circ \brac{1/\sqrt{\var\brac{\paren{\boldsymbol{S}_{t, w}}_{uv}}}}^p_{u,v=1}.
\end{equation*}
Using \emph{Isserlis’ Theorem} \cite{janson1997gaussian}, we obtain the following closed form expression for $\var\brac{\paren{\boldsymbol{S}_{t, w}}_{uv}}$. 
\begin{equation*}
\var\brac{\paren{\boldsymbol{S}_{t, w}}_{uv}} = \frac{1}{w}\brac{\TimeDepMatrix{\Omega}{t}_{uu}\TimeDepMatrix{\Omega}{t}_{vv} + \paren{\TimeDepMatrix{\Omega}{t}_{uv}}^2}.
\end{equation*}
Thus, $\boldsymbol{E}_{t, w}$ can be equivalently written as
\begin{equation}\label{Sbartw}
\boldsymbol{E}_{t, w} = \sum_{r=1}^{w}\frac{ \paren{\boldsymbol{Y}_{t+r}\boldsymbol{Y}_{t+r}^\top-\TimeDepMatrix{\Omega}{t}}}{\sqrt{w}} \circ \brac{ \paren{\TimeDepMatrix{\Omega}{t}_{uu}\TimeDepMatrix{\Omega}{t}_{vv} + \paren{\TimeDepMatrix{\Omega}{t}_{uv}}^2}^{-1/2} }^p_{u,v=1}.
\end{equation}

\begin{rem}
The idea of using the transformed random vectors $\boldsymbol{Y}_{t+1},\ldots, \boldsymbol{Y}_{t+w}$ has previously appeared in the problem of decision making on multivariate Gaussian observations. For instance, Cai et al. \cite{cai2014two} proposed a similar transformation for designing an optimal two-sample testing procedure for a sparse difference in means problem, with correlated observations. The same transformation has also showed its utility in the context of online detection of abrupt changes in the inverse covariance matrix of high dimensional GGMs \cite{keshavarz2018sequential}.
\end{rem}

Next, we introduce the proposed sequential detection algorithm. Let $\pi_0$ be a pre-specified false alarm rate. Consider the following binary decision function
\begin{equation}\label{ProposedTest}
T_t = \bbM{1}\paren{\LpNorm{\boldsymbol{E}_{t, w}}{\infty} \geq \zeta_{\pi_0,p,w}},
\end{equation}
for distinguishing between $\bb{H}_{0,t}$ (no abrupt change at $t$) and $\bb{H}_{1,t}$ in Eq. \eqref{HypothesisTestingFormulation}, with $\zeta_{\pi_0,p,w}$ being a critical value depending on $\pi_0, p,$ and $w$. Simply put, we reject $\bb{H}_{0,t}$ ($T_t=1$) only if $\LpNorm{\boldsymbol{E}_{t, w}}{\infty}$ is greater than $\zeta_{\pi_0,p,w}$ . We choose $\zeta_{\pi_0,p,w}$ so that the probability of falsely rejecting $\bb{H}_{0,t}$ is around $\pi_0$, if there is no change point between $\paren{t+1}$ and $\paren{t+w-1}$.

\begin{rem}
As mentioned in the introductory Section,
Keshavarz et al. \cite{keshavarz2018sequential} examined a similar problem, and designed an online algorithm for detecting abrupt changes in the topology of GGMs involving \emph{many edges simultaneously}. Their method is based on aggregating a convex function of the diagonal entries of $\boldsymbol{E}_{t, w}$, and proves effective for spotting change points distributing across the entire GGM (e.g., $\TimeDepMatrix{\Omega}{t+1}-\beta\TimeDepMatrix{\Omega}{t}$ is positive definite for some $\beta>0$). However, the emphasis in this manuscript is on addressing the online detection problem wherein the change is driven by changes in a few edges. This different objective motivates the choice of the $\ell_\infty$ norm in the formulation of $T_t$.
\end{rem}

\subsection{Data driven case}\label{Section2.2}

The oracle setting is not realistic in most real-world applications. We relax this condition by plugging an estimate of $\TimeDepMatrix{\Omega}{t}$ into $T_t$. Specifically, let $\TimeDepMatrix{\hat{\Omega}}{t}$ denote an estimate of $\TimeDepMatrix{\Omega}{t}$. We approximate the oracle decision function \eqref{ProposedTest} by
\begin{equation}\label{PluginProposedTest}
\hat{T}_t = \bbM{1}\paren{\LpNorm{\boldsymbol{\hat{E}}_{t, w}}{\infty} \geq \zeta_{\pi_0,p,w}},
\end{equation}
where $\boldsymbol{\hat{E}}_{t, w}$ represents the plug-in estimate of $\boldsymbol{E}_{t, w}$ (see Eq. \eqref{Sbartw}), given by
\begin{equation}\label{PlugnStatistic}
\boldsymbol{\hat{E}}_{t, w} = \sum_{r=1}^{w}\frac{ \paren{\TimeDepMatrix{\hat{\Omega}}{t}\boldsymbol{X}_{t+r}\paren{\TimeDepMatrix{\hat{\Omega}}{t}\boldsymbol{X}_{t+r}}^\top-\TimeDepMatrix{\hat{\Omega}}{t}}}{\sqrt{w}} \circ \brac{ \paren{\TimeDepMatrix{\hat{\Omega}}{t}_{uu}\TimeDepMatrix{\hat{\Omega}}{t}_{vv} + \paren{\TimeDepMatrix{\hat{\Omega}}{t}_{uv}}^2}^{-1/2} }^p_{u,v=1}.
\end{equation}
Given enough temporal separation between consecutive changes in a time-varying sparse GGM $\set{\cc{G}_t: t\in\bb{N}}$, the pre-change precision matrix can be estimated using effective procedures in the literature, such as the \emph{CLIME} \cite{cai2011constrained} algorithm. Controlling the false alarm and mis-detection rate hinges upon the availability of a good estimate of $\TimeDepMatrix{\Omega}{t}$, and an adequate separation in time between two consecutive change points. In order to formalize this notion, recall that $t^\star_j$ denotes the location of the $j$-th change point and we also set $t^\star_0 = 0$. Specifically, we suppose that there is $N\in\bb{N}$, depending on $p, w$ and the sparsity pattern of $\cc{G}_t$ between $t^\star_j$ and $t^\star_{j+1}$, such that 
\begin{equation}\label{BurnInCond}
\abs{t^\star_{j+1}-t^\star_j} > N,\;\forall\; j\geq 1.
\end{equation}
We refer to the first $N$ samples after $t^\star_j$ as the \emph{burn-in period}. Conditions similar to Eq. \eqref{BurnInCond} have appeared in both offline and online change point detection literature (see e.g. \cite{keshavarz2018sequential,roy2016change}). For the time being, selection of $N$ is postponed to a later Section.

Detecting each change point is broken into two segments. For brevity, we only focus on locating $t^\star_1$. Notice that for each $t\leq t^\star_1$, $\TimeDepMatrix{\Omega}{t} = \TimeDepMatrix{\Omega}{1}$.
\begin{enumerate}[label = (\alph*),leftmargin=*]
\item Given $\boldsymbol{X}_1,\ldots,\boldsymbol{X}_{N}$, we estimate $\TimeDepMatrix{\Omega}{1}$ using the CLIME algorithm (which is denoted by $\TimeDepMatrix{\hat{\Omega}}{1}$).
\item For any $t>N$, we compute $\boldsymbol{\hat{E}}_{t, w}$ and $\hat{T}_t$ using Eq. \eqref{PluginProposedTest}-\eqref{PlugnStatistic}. When $\hat{T}_t = 0$, we update $\TimeDepMatrix{\hat{\Omega}}{1}$ after observing a new data-point. In contrast if $\hat{T}_t$ equals one (an abrupt change at $t$), we wait for $\boldsymbol{X}_{t+1},\ldots,\boldsymbol{X}_{t+N}$ for estimating the post-change inverse covariance matrix.
\end{enumerate}

We use batch procedure for updating the pre-change precision matrix, wherein we first get $B$ (a pre-specified batch size) new samples and subsequently a new estimate at time $t = N+kB$ ($k\in\bb{N}$) by employing $\boldsymbol{X}_1,\ldots,\boldsymbol{X}_{N+kB}$; the parameter $k$ tracks the number of size-$B$ batches before the first abrupt change. Throughout this paper, the CLIME algorithm~\cite{cai2011constrained} is used for estimating the oracle test statistic $\boldsymbol{E}_{t, w}$, due to its desirable theoretical and numerical properties. The detailed pseudocode of the detection procedure is presented in \emph{Algorithm $1$}.

\begin{center}\label{Algo1}
\begin{tabular}{ >{\arraybackslash}m{6.1in} }
\thickhline
\vspace{1mm}
\begingroup
\fontsize{12pt}{12pt}\selectfont
\textbf{Algorithm 1} Sequential detection with batch update of pre-change precision matrix 
\endgroup\\		
\hline
\multicolumn{1}{l}{\small \textbf{Input:} $N, w, B,$and $\zeta_{\pi_0, p, w}$}\\
\multicolumn{1}{l}{\small \textbf{Initialization} Set $\hat{\cc{D}} = \emptyset$ and $k = 0$. Given $\boldsymbol{X}_1,\ldots,\boldsymbol{X}_{N}$, compute $\TimeDepMatrix{\hat{\Omega}}{1}$ by the CLIME algorithm.}\\
\multicolumn{1}{l}{\small Also set $\hat{t}_{last} = 0$, where $\hat{t}_{last}$ denotes the estimated location of the last change point.}\\
\multicolumn{1}{l}{\small\textbf{Iterate} For $t > N$} \\
\vspace{-2mm}
\begin{itemize}
\item[] \small Set $\hat{T}_t = \bbM{1}\paren{\LpNorm{\boldsymbol{\hat{E}}_{t, w}}{\infty} \geq \zeta_{\pi_0,p,w}}$.
\item[] \small If $\hat{T}_t = 0$ (no change point)
\vspace{1mm}
\begin{itemize}
\item[] \small $b \leftarrow b+1$ and $t \leftarrow t+1$.
\item[] \small If $b = B$ (Update pre-change precision matrix after observing $B$ new samples)
\vspace{1mm}
\begin{itemize}
\item[] \small Update $\TimeDepMatrix{\hat{\Omega}}{t}$ using the CLIME algorithm with data points $\boldsymbol{X}_{1+\hat{t}_{last}},\ldots, \boldsymbol{X}_{t-1},\boldsymbol{X}_{t}$.
\item[] \small $b\leftarrow 0$.
\end{itemize}
\item[] \small Else
\vspace{1mm}
\begin{itemize}
\item[] \small $\TimeDepMatrix{\hat{\Omega}}{t} = \TimeDepMatrix{\hat{\Omega}}{t-1}$
\end{itemize}
\end{itemize}
\item[] \small Else
\begin{itemize}
\item[] \small $\hat{t}_{last} = t$ and $\hat{\cc{D}} \leftarrow \hat{\cc{D}} \cup \set{{\hat{t}_{last}}}$.
\item[] \small Given $\boldsymbol{X}_{t},\ldots, \boldsymbol{X}_{t+N-1}$, estimate post-change precision matrix using CLIME method.
\item[] \small $t \leftarrow t+N$ and $b\leftarrow 0$.
\end{itemize}
\vspace{-3mm}		
\end{itemize}\\
\multicolumn{1}{l}{\small \textbf{Output: $\hat{\cc{D}}$}}\\
\hline
\end{tabular}
\end{center}

\section{Large-sample analysis of $T_t$}\label{Section3}

This section is devoted to the large-sample properties of the oracle decision function $T_t$ introduced in Eq. \eqref{ProposedTest} under both the null and alternative hypotheses. Recall that calculating $T_t$ requires full knowledge of $\TimeDepMatrix{\Omega}{t}$. The results in this section form the necessary backbone of the analysis involving data. In particular, we address the following issues.
\begin{enumerate}[leftmargin=*]
\item How to select the critical value $\zeta_{\pi_0,p,w}$ and $w$, to ensure that the false alarm probability converges to $\pi_0$, in the asymptotic scenario of growing graph size $p$ and delay $w$? 
\item Establishing an upper bound for the mis-detection probability of $T_t$.
\end{enumerate}
For a more clear presentation of the main results, we start by introducing some simplifying notation.

\begin{defn}\label{InnerProdGauss}
Let $\boldsymbol{X}, \boldsymbol{Y}\in\bb{R}^w$ be two independent standard Gaussian random vectors. Define their standardized inner product by
\begin{equation*}
\vartheta_w \coloneqq \frac{\InnerProd{\boldsymbol{X}}{\boldsymbol{Y}}}{\sqrt{w}}.
\end{equation*}
\end{defn}

\begin{defn}\label{WellPosedMatrixSet}
For two scalars $d_{\max}\in\bb{N}$ and $\alpha_{\min} > 0$, define $\cc{C}^{p\times p}_{++}\paren{\alpha_{\min}, d_{\max}}$ by
\begin{equation*}
\cc{C}^{p\times p}_{++}\paren{\alpha_{\min}, d_{\max}} \coloneqq \set{\boldsymbol{A}\in S^{p\times p}_{++}:\; \max_{1\leq i\leq p} \LpNorm{\boldsymbol{A}\boldsymbol{e}_i}{0} \leq d_{\max},\; \lambda_{\min}\paren{\boldsymbol{A}}\geq \alpha_{\min} }.
\end{equation*}
\end{defn}

To obtain the limiting null distribution of $T_t$, in addition to requiring sparsity of $\TimeDepMatrix{\Omega}{t}$, we also assume that its eigenvalues are bounded from above and below. Specifically, we consider the following setting.

\begin{assu}\label{AssuCLT}
$\TimeDepMatrix{\Omega}{t}\in \cc{C}^{p\times p}_{++}\paren{\alpha_{\min}, d_{\max}}$ for some fixed, bounded and strictly positive scalars $d_{\max}$ and $\alpha_{\min}$. Further, there exists a scalar $r_{\max}\in\paren{0,1}$ such that 
\begin{equation*}
\max_{1\leq i < j\leq p} \abs{\frac{\TimeDepMatrix{\Omega}{t}_{ij}}{\sqrt{\TimeDepMatrix{\Omega}{t}_{ii}\TimeDepMatrix{\Omega}{t}_{jj}}}} \leq r_{\max}.
\end{equation*}
\end{assu}

A slightly weaker version of Assumption \ref{AssuCLT} has appeared in the context of two-sample testing for high-dimensional and sparse means (see e.g. \cite{cai2014two}). Assumption \ref{AssuCLT} restricts $\TimeDepMatrix{\Omega}{t}$ to have sparse rows. Namely, the maximum degree of $\cc{G}^{\paren{t}}$ is supposed to remain below some fixed $d_{\max}$ for all $t$'s. We postulate this assumption (instead of softer versions controlling $\OpNorm{\TimeDepMatrix{\Omega}{t}}{1}{1}$ from above) only for simplifying the theoretical derivations, without being distracted by cumbersome algebraic details. We believe that Assumption \ref{AssuCLT} can be relaxed by making appropriate adjustments in the proof.

As the first result, we present sufficient conditions on $w$, $p$, and the topology of the time-varying GGM, under which the asymptotic false alarm probability of $T_t$, in Eq. \eqref{ProposedTest}, is guaranteed to remain below a pre-specified level $\pi_0\in\paren{0,\frac{1}{2}}$. For studying the null distribution of $T_t$, we assume that no change point occurs between $t$ and $t+w$, i.e., $\TimeDepMatrix{\Omega}{t} = \TimeDepMatrix{\Omega}{t+1} = \ldots = \TimeDepMatrix{\Omega}{t+w}$. Such a restriction provides both intuitive insights to the theoretical novelty of the results and eases comprehension of the proof strategies by the reader.

\begin{thm}\label{thm0} (A First Result based on a Stringent Condition) \\
Consider the asymptotic scenario $p, w\rightarrow\infty$ with the following conditions:
\begin{enumerate}[label = (\alph*),leftmargin=*]
\item $\TimeDepMatrix{\Omega}{t}$ satisfies Assumption \ref{AssuCLT}.
\item $w^{-1}\log^8 p \rightarrow 0$.
\end{enumerate}
Further, let $\pi_0\in\paren{0,\frac{1}{2}}$ and choose $\zeta_{\pi_0, p, w}$ by
\begin{equation}
\zeta^2_{\pi_0, p, w} = 2\log{p+1 \choose 2}-\log\log{p+1 \choose 2} -2\log\brac{2\sqrt{\pi}\log\paren{\frac{1}{1-\pi_0/2}}}
\end{equation}
Then, 
\begin{equation*}
\limsup_{w, p\rightarrow\infty} \;\bb{P}_{\FA}\paren{T_t} = \limsup_{w, p\rightarrow\infty} \;\bb{P}\paren{\LpNorm{\boldsymbol{E}_{t, w}}{\infty}\geq \zeta_{\pi_0, p, w} } \leq \pi_0.
\end{equation*}	
\end{thm}

\begin{rem}\label{rem3.1}
For gaining insights, we outline a brief sketch of the proof of Theorem \ref{thm0}; full details are provided in Section \ref{Proofs}. Define the set $\cc{K}_p \coloneqq \set{\paren{r,s}:\;1\leq r\leq s\leq p}$. The goal is to find a critical value $\zeta_{\pi_0, p, w}$ such that
\begin{equation*}
\bb{P}_{\FA}\paren{T_t} = \bb{P}\paren{\max_{\paren{r,s}\in\cc{K}_p} \abs{\paren{\boldsymbol{E}_{t, w}}_{rs}}\geq \zeta_{\pi_0, p, w}} \leq \pi_0\brac{1+o\paren{1}},\quad\mbox{as\;\;} p, w\rightarrow\infty.
\end{equation*}
For ease of presentation, we drop the dependence on $t$ and $w$ in $\boldsymbol{E}_{t, w}$. An application of the union bound yields
\begin{equation*}
\bb{P}_{\FA}\paren{T_t}\leq \bb{P}\paren{\max_{\paren{r,s}\in\cc{K}_p} \boldsymbol{E}_{rs}\geq \zeta_{\pi_0, p, w}} + \bb{P}\paren{\max_{\paren{r,s}\in\cc{K}_p} -\boldsymbol{E}_{rs}\geq \zeta_{\pi_0, p, w}}.
\end{equation*}
The goal is to demonstrate that $\bb{P}\paren{\max_{\paren{r,s}\in\cc{K}_p} \boldsymbol{E}_{rs}\geq \zeta_{\pi_0, p, w}} \leq \frac{\pi_0}{2}\brac{1+o\paren{1}}$. As each entry of $\boldsymbol{E}$ is a summation of $w$ independent standardized sub-exponential random variables, well-known results for the Gaussian approximation of the maximum of zero-mean empirical processes -see Theorem $4.1$ in Chernozhukov et al. \cite{chernozhukov2014gaussian}- prove that under Assumption \ref{AssuCLT},
\begin{equation}\label{GaussApprox}
\max_{\paren{r,s}\in\cc{K}_p} \boldsymbol{E}_{rs}-\max_{\paren{r,s}\in\cc{K}_p} \boldsymbol{G}_{rs} \cp{\bb{P}} 0,
\end{equation}
where $\boldsymbol{G}$ is a centered Gaussian process (indexed by $\cc{K}_p$) with the same correlation structure as $\boldsymbol{E}$. Specifically,
\begin{equation*}
\cov\paren{\boldsymbol{G}_{rs}, \boldsymbol{G}_{r's'}} = \cov\paren{\boldsymbol{E}_{rs}, \boldsymbol{E}_{r's'}},\quad \paren{r,s},\paren{r',s'}\in\cc{K}_p.
\end{equation*}
Next, we employ Lemma $6$ of \cite{cai2014two}, on the extreme value distribution of unstructured sequence of Gaussian random variables with sparse covariance matrix, to obtain $\zeta_{\pi_0, p, w}$ satisfying
\begin{equation*}
\bb{P}\paren{\max_{\paren{r,s}\in\cc{K}_p} \boldsymbol{G}_{rs}\geq \zeta_{\pi_0, p, w}} \rightarrow \frac{\pi_0}{2}.
\end{equation*}
Combining these two pieces of information concludes the proof.
\end{rem}

\begin{rem}
We  presented a rather simple version of Theorem \ref{thm0} in this section, due to the difficulties of tracking cumbersome algebraic derivations. For instance, the reader may demand to see the convergence rate in Eq. \eqref{GaussApprox} as a function of $w$ and $p$. A meticulous review of algebraic steps in the proof of Theorem \ref{thm0} (see the proof of Claim \ref{Claim2Thm0} in pages $22-24$) reveals that for any $\xi\in\brapar{0, 8}$
\begin{equation*}
\bb{P}\paren{\abs{\max_{\paren{r,s}\in\cc{K}_p} \boldsymbol{E}_{rs}-\max_{\paren{r,s}\in\cc{K}_p} \boldsymbol{G}_{rs}} \geq C_\xi\sqrt[8]{\frac{\log^{8+\xi} p}{w}} } = \cc{O}\paren{\sqrt[4]{\frac{\log^{8-\xi} p}{w}}},
\end{equation*}
where $C_\xi$ is a bounded scalar depending only on $\xi$. Note that the second condition in Theorem \ref{thm0} focuses on the simplest case of $\xi = 0$.
\end{rem}

\begin{rem}
The union bound in the proof of Theorem \ref{thm0} provides a proper setting for using existing Gaussian approximation results in the literature. Despite an unsuccessful attempt, we guess that a modified technique can be used for proving the following (two-sided) variant of Eq. \eqref{GaussApprox}.
\begin{equation*}
\max_{\paren{r,s}\in\cc{K}_p} \abs{\boldsymbol{E}_{rs}}-\max_{\paren{r,s}\in\cc{K}_p} \abs{\boldsymbol{G}_{rs}} \cp{\bb{P}} 0,
\end{equation*}
Given the validity of our conjecture, one can show (using the same proposed method) that $\bb{P}_{\FA}\paren{T_t}\rightarrow \pi_0$, if the critical value $\zeta_{\pi_0, p, w}$ is chosen by
\begin{equation}\label{zeta2}
\zeta^2_{\pi_0, p, w} = 2\log{p+1 \choose 2}-\log\log{p+1 \choose 2} -2\log\brac{\sqrt{\pi}\log\paren{\frac{1}{1-\pi_0}}}+o\paren{1}.
\end{equation}
\end{rem}

Note that Theorem \ref{thm0} requires $\log^8 p = o\paren{w}$ that is excessively stringent for most real-world settings. For example, for a GGM with $p=100$ nodes/variables, $w$ should be of the order $10^6$. The reason is that leveraging a generic infinite-dimensional Gaussian approximation result in \cite{chernozhukov2014gaussian} is an unnecessarily powerful tool for obtaining the null distribution of $\max_{\paren{r,s}\in\cc{K}_p} \abs{\boldsymbol{E}_{rs}}$ (maxima of a finite-dimensional, yet asymptotically growing stochastic process). Hence, the resulting proof strategy that employs Theorem $4.1$ in \cite{chernozhukov2014gaussian}  (note that $\abs{\boldsymbol{E}_{rs}}$ comprises of heavy-tail quadratic terms) requires a stringent condition on $w$. 

In the sequel, we derive results for the false alarm rate of $T_t$ based on a novel theoretical technique that relaxes considerably Condition (b). Specifically, we establish:

\begin{thm}\label{Thm1} [A Refined Result based on a Relaxed Condition] \\
We consider the asymptotic scenario $p, w\rightarrow\infty$ with the following conditions:
\begin{enumerate}[label = (\alph*),leftmargin=*]
\item $\TimeDepMatrix{\Omega}{t}$ satisfies Assumption \ref{AssuCLT}.
\item $w^{-1}\log^3 p \rightarrow 0$.
\end{enumerate}
Further, let $\pi_0\in\paren{0,1}$ and choose $\zeta_{\pi_0, p, w}$ so that
\begin{equation}\label{zeta}
\bb{P}\paren{ \abs{\vartheta_w} \geq \zeta_{\pi_0, p, w}} = \frac{2}{p\paren{p+1}}\log\paren{\frac{1}{1-\pi_0}}.
\end{equation}
Then as $w, p\rightarrow\infty$, we have
\begin{equation*}
\bb{P}_{\FA}\paren{T_t} = \bb{P}\paren{\LpNorm{\boldsymbol{E}_{t, w}}{\infty}\geq \zeta_{\pi_0, p, w} } \rightarrow \pi_0.
\end{equation*} 
\end{thm}  

\begin{rem}
Theorem \ref{Thm1} replaces the restriction on $w$ in Theorem \ref{thm0} with the requirement that $\log^3 p = o\paren{w}$, which is suitable for many real-world settings. For example, $w$ is now of the order $100$ for a GGM comprising of $p=100$ nodes. Note that the sparsity of $\TimeDepMatrix{\Omega}{t}$ implies that 
\begin{itemize}[leftmargin=*]
\item $\boldsymbol{E}_{rs} \eqd \vartheta_w$ for the majority of $\paren{r,s}\in\cc{K}_p$.
\item Asymptotically, a portion of the order $\cc{O}\paren{p^{-1}d_{\max}}$ of distinct pairs of edges $\paren{\boldsymbol{E}_{rs}, \boldsymbol{E}_{r's'}}$ are dependent.
\end{itemize}
Thus, under some regularity conditions, the distribution of $\LpNorm{\boldsymbol{E}_{t, w}}{\infty}$ is close to that of the maximum of $\abs{\cc{K}_p} = {p+1 \choose 2}$ i.i.d. random variables distributed as $\abs{\vartheta_w}$. This fact qualitatively justifies the formulation of $\zeta_{\pi_0, p, w}$ in Eq. \eqref{zeta}.
\end{rem}

\begin{rem}
In contrast to the proof of Theorem \ref{thm0}, we do not approximate the elements in $\set{\abs{\boldsymbol{E}_{rs}}:\; \paren{r,s}\in\cc{K}_p}$ by centered Gaussian random variables, with the same dependence structure, in Theorem \ref{Thm1}. Instead, we characterizes the limiting distribution of the extreme value of dependent random variables $\set{\abs{\boldsymbol{E}_{rs}}:\; \paren{r,s}\in\cc{K}_p}$ in a \emph{direct} fashion. The technical challenge is that unlike time series, the elements of $\cc{K}_p$ do not exhibit \emph{any natural ordering}. The upshot is that we can not directly utilize classical results on the extreme values of time series (see e.g., \cite{leadbetter1988extremal}). 

Another strategy could be similar to the one adopted in Cai et al. \cite{cai2014two}, wherein the \emph{Bonferroni inequality} was used (together with similar conditions as in Assumption \ref{AssuCLT}) to simultaneously control the asymptotic distribution of the maximum of \emph{unordered} dependent Gaussian random variables from above and below (see the proof of Lemma $6$ in \cite{cai2013two} for further details). From a theoretical standpoint, the framework in Theorem \ref{Thm1} is much more challenging that univariate correlated Gaussian random variables. First, the difficulty of working with a large number of joint probability terms, which is a key disadvantage of the Bonferroni inequality, is exacerbated in our case, since $\abs{\cc{K}_p} = \cc{O}\paren{p^2}$ and its elements move in two directions (rows and columns of $\TimeDepMatrix{\Omega}{t}$). Further, the marginal and joint distributions of elements in $\boldsymbol{E}_{t,w}$ are significantly more complicated than the Gaussian case in Lemma $6$ in \cite{cai2013two}. 

The upshot of the previous discussion is that a new proof strategy is needed to establish the results in Theorem \ref{Thm1}. To that end, we leverage \emph{Galambos' Theorem} (see e.g. \cite{galambos1988variants, galambos1972distribution}, as stated for completeness in Theorem \ref{GalambosThm}). The latter is specifically designed for finding the asymptotic distribution of the maximum of \emph{graph-dependent} random variables. The proof of Galambos' result is based on a generalized and more flexible version of the inclusion-exclusion principle by Renyi \cite{renyi1976general}.
\end{rem}

\begin{rem}
A careful reading of the proof of Theorem \ref{Thm1} reveals that it can be extended to the case that $d_{\max} = \cc{O}\paren{p^{\kappa}}$ for some $\kappa\in\paren{0,1}$. However, the assumption of fixed $d_{\max}$ improves the readability of our technical contribution without focusing on unnecessary cumbersome technicalities in the asymptotics.  
\end{rem}

\begin{rem} (Asymptotic behaviour of $\zeta_{\pi_0, p, w}$) Lemma \ref{ChernoffBound} guarantees the existence of a bounded scalar $C$ for which
\begin{equation*}
\bb{P}\paren{ \abs{\vartheta_w} \geq \zeta_{\pi_0, p, w}} \leq Ce^{-\frac{\zeta^2_{\pi_0, p, w}}{2}}.
\end{equation*}
The definition of $\zeta_{\pi_0, p, w}$ in Eq. \eqref{zeta} also implies that $\bb{P}\paren{ \abs{\vartheta_w} \geq \zeta_{\pi_0, p, w}} = \cc{O}\paren{p^{-2}}$. Combining these two facts yields that $\zeta_{\pi_0, p, w} = \cc{O}\paren{\sqrt{\log p}}$, or equivalently $\zeta^6_{\pi_0, p, w} = \cc{O}\paren{\log^3 p}$, as $p, w\rightarrow\infty$. According to Corollary \ref{GaussInnerProdUppBnd}, if $\zeta^6_{\pi_0, p, w} \asymp \log^3p = o\paren{w}$ when $p,w\rightarrow\infty$, then 
\begin{equation*}
\frac{2}{p\paren{p+1}}\log\paren{\frac{1}{1-\pi_0}} = \bb{P}\paren{\abs{\vartheta_w}\geq \zeta_{\pi_0,p,w}} \sim \int_{\abs{x} \geq \zeta_{\pi_0,p,w}} \frac{e^{-x^2/2}}{\sqrt{2\pi}} dx \sim \frac{\exp\paren{-\frac{\zeta^2_{\pi_0,p,w}}{2}}}{\zeta_{\pi_0,p,w}}\sqrt{\frac{2}{\pi}}.
\end{equation*}
Rearranging the terms in the both sides, yields
\begin{equation*}
\zeta^2_{\pi_0, p, w} = 2\log{p+1 \choose 2}-\log\log{p+1 \choose 2} -2\log\brac{\sqrt{\pi}\log\paren{\frac{1}{1-\pi_0}}} +o\paren{1},
\end{equation*}
which is exactly the same as the posited expression for $\zeta_{\pi_0, p, w}$ in Eq. \eqref{zeta2}. So, our novel proof technique successfully establishes the desired result, under a weaker condition on $w$. 
\end{rem}

\begin{figure}
\includegraphics[scale=0.6]{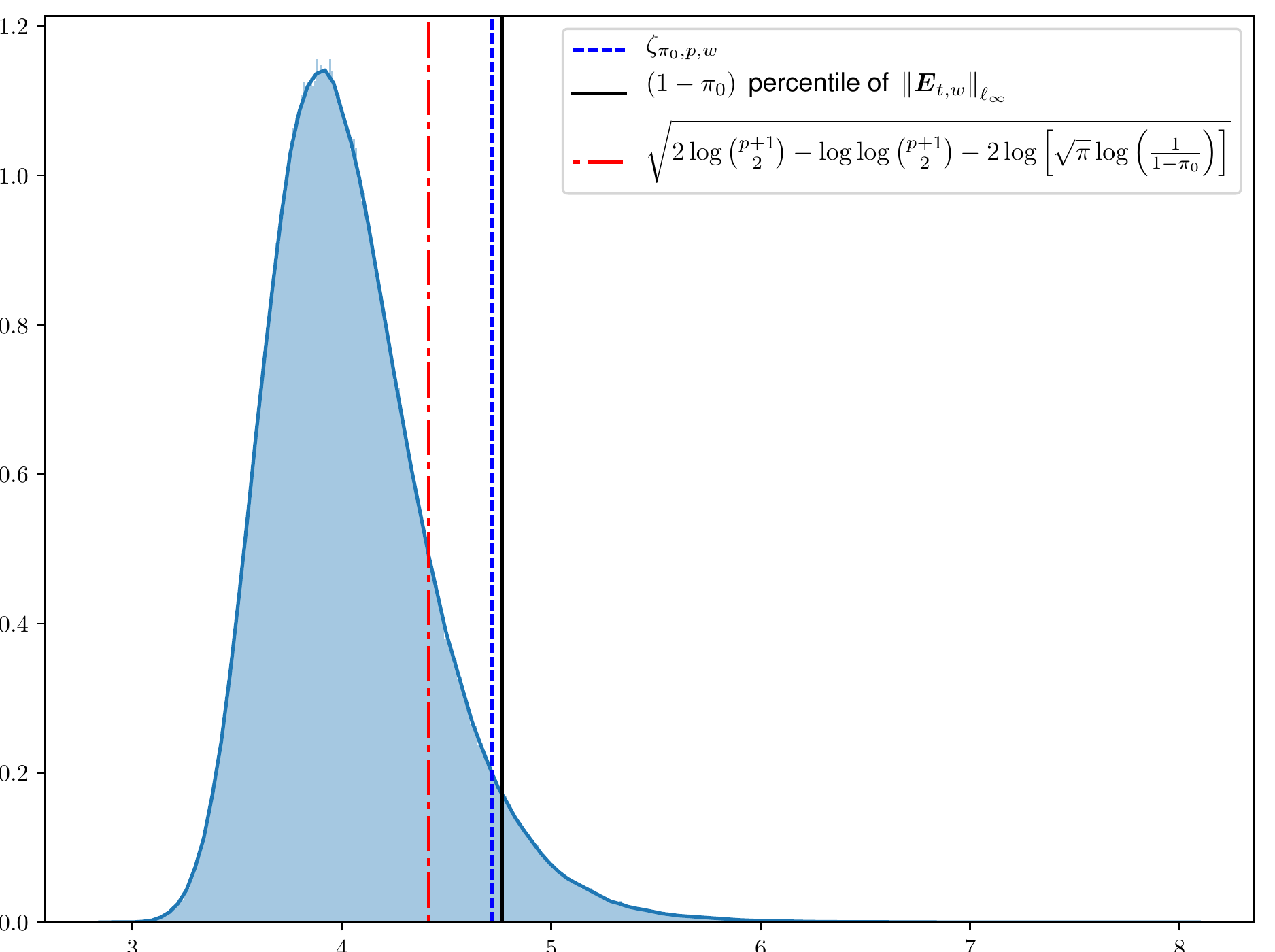}
\caption{Histogram of $\LpNorm{\boldsymbol{E}_{t, w}}{\infty}$ over $10^6$ independent experiments. The solid black, blue dashed and red dashed lines respectively show $\paren{1-\pi_0}$ quantile of the histogram, the threshold calculated by Theorems \ref{Thm1} and \ref{thm0}.}
\label{fig:imagea}
\end{figure}

Although the results in both Theorems \ref{thm0} and \ref{Thm1} yield asymptotically the same critical value, the analysis does not reveal all nuances for moderate-sized GGMs. To that end, we conclude this section by presenting a numerical experiment. We consider a stationary GGM with $p=100$ vertices (corresponding to $\bb{H}_{0,t}$), whose precision matrix is given by
\begin{equation*}
\TimeDepMatrix{\Omega}{t} = \begin{bmatrix} 
1 & \rho_0 & 0 & \ldots & 0 \\
\rho_0 & 1 & \rho_0 & \ldots & 0\\
0 & \ldots & \ddots & 1 & \rho_0 \\
0 & \ldots & 0 & \rho_0 & 1
\end{bmatrix},\quad\forall\; t.
\end{equation*}
Roughly around $2\%$ of the nodes are connected together. We choose $\rho_0 = 0.5$, $w=50$, and $\pi_0 = 0.05$. Figure \ref{fig:imagea} shows the histogram of $\LpNorm{\boldsymbol{E}_{t, w}}{\infty}$ over $10^6$ independent replicates The solid black line indicates the $\paren{1-\pi_0}$-quantile of the histogram, whereas the blue and red dashed lines show the corresponding critical values calculated by Theorems \ref{Thm1} and \ref{thm0}, respectively. As it can be seen from Figure \ref{fig:imagea}, $\zeta_{\pi_0, p, w}$ obtained by the direct analysis in Theorem \ref{Thm1} is markedly closer to the actual critical value than the Gaussian approximation approach. This experiment illustrates that our proof technique not only needs weaker conditions, it can also reduce the false alarm rate for moderate-sized GGMs.

\subsection{Distribution of $T_t$ under $\bb{H}_{1,t}$}\label{Section3.2}

This section focuses on the behavior of the proposed test for detecting abrupt changes (corresponding to the alternative hypothesis $\bb{H}_{1,t}$). Our objective is to introduce sufficient conditions under which the mis-detection rate is guaranteed to diminish asymptotically. We also assess the efficiency of our proposed change point detection algorithm by comparing the obtained asymptotic results with existing approaches. Throughout this section, we refer to the asymmetric matrix $\TimeDepMatrix{\Delta}{t}$ by
\begin{equation}\label{Delta}
\TimeDepMatrix{\Delta}{t} \coloneqq \paren{\TimeDepMatrix{\Omega}{t}\TimeDepMatrix{\Sigma}{t+1}\TimeDepMatrix{\Omega}{t}-\TimeDepMatrix{\Omega}{t}}\circ \brac{ \paren{\TimeDepMatrix{\Omega}{t}_{uu}\TimeDepMatrix{\Omega}{t}_{vv} + \paren{\TimeDepMatrix{\Omega}{t}_{uv}}^2}^{-1/2} }^p_{u,v=1}.
\end{equation}
Recall $\boldsymbol{E}_{t,w}$ and $T_t$ from Eq. \eqref{Sbartw} and \eqref{ProposedTest}. It is easy to see that, $\TimeDepMatrix{\Delta}{t}$ denotes the expected value of $\boldsymbol{E}_{t,w}$ under $\bb{H}_{1,t}$. Intuitively, $\TimeDepMatrix{\Delta}{t}$ captures the change point signal at time $t$. It is worth mentioning that as $\TimeDepMatrix{\Delta}{t} = \zero_{p\times p}$, when there is no sudden change at time $t$, the asymptotic detection guarantees will be encoded in terms of some norm of $\TimeDepMatrix{\Delta}{t}$. This qualitative claim is formalized in the following result. 

\begin{thm}\label{Thm2}
Consider the asymptotic scenario $p, w\rightarrow\infty$ with the following conditions:
\begin{enumerate}[label = (\alph*),leftmargin=*]
\item $\TimeDepMatrix{\Omega}{t}$ satisfies Assumption \ref{AssuCLT}.
\item $w^{-1}\log p \rightarrow 0$.
\end{enumerate}
For any strictly positive $\xi$, there is a bounded scalar $C_{\xi}$ such that 
\begin{equation*}
\bb{P}_{\MD}\paren{T_t} \leq p^{-\xi},
\end{equation*}
whenever
\begin{equation}\label{SuffDetectCond}
\LpNorm{\TimeDepMatrix{\Delta}{t}}{\infty} \geq \sqrt{\frac{\zeta^2_{\pi_0, p, w}}{w}} +C_{\xi}\frac{d^2_{\max}}{\alpha_{\min}}\sqrt{\frac{\log p}{w}}.
\end{equation}	
\end{thm}

Theorem \ref{Thm2} introduces sufficient conditions on the detection delay $w$, and $\LpNorm{\TimeDepMatrix{\Delta}{t}}{\infty}$ for controlling the mis-detection rate from above. Particularly, we formulate an asymptotic setting under which $\bb{P}_{\MD}\paren{T_t}$ converges to zero at a polynomial rate in $p$. Comparing the conditions in Theorems \ref{Thm1} and \ref{Thm2} reveals that studying $\bb{P}_{\MD}\paren{T_t}$ requires a less restrictive asymptotic framework than the false alarm rate. Unlike Theorem \ref{Thm1}, the intent of Theorem \ref{Thm2} is not to find the exact asymptotic distribution of $T_t$ under $\bb{H}_{1,t}$. Indeed, it solely focuses on obtaining a sharp sufficient condition for controlling $\bb{P}\paren{T_t = 0 \mid \bb{H}_{1,t}}$ from above. 

\begin{rem}
Although the sufficient condition on $\TimeDepMatrix{\Delta}{t}$ in Theorem \ref{Thm2} looks somewhat involved, Assumption \ref{AssuCLT} helps us write a simpler, more intuitive detection criterion. Without loss of generality, we further assume that all the diagonal entries of $\TimeDepMatrix{\Omega}{t}$ are equal to one, as $T_t$ is standardized in the no-change setting. In this case, $\TimeDepMatrix{\Delta}{t}$ can be rewritten in the following form.
\begin{equation*}
\TimeDepMatrix{\Delta}{t} = -\TimeDepMatrix{\Omega}{t}\TimeDepMatrix{\Sigma}{t+1} \paren{\TimeDepMatrix{\Omega}{t+1}-\TimeDepMatrix{\Omega}{t}}\circ \brac{ \sqrt{\frac{1}{1 + \paren{\TimeDepMatrix{\Omega}{t}_{uv}}^2}} }^p_{u,v=1}
\end{equation*}
Assumption \ref{AssuCLT} also ensures the boundedness of $\OpNorm{\TimeDepMatrix{\Omega}{t}\TimeDepMatrix{\Sigma}{t+1}}{\infty}{\infty}$ as $p\rightarrow\infty$. Thus, under the same conditions as in Theorem \ref{Thm1}, $\bb{P}_{\MD}\paren{T_t} \leq p^{-\xi}$, if
\begin{equation*}
\max_{1\leq u\leq v\leq p} \frac{\abs{\TimeDepMatrix{\Omega}{t+1}_{uv}-\TimeDepMatrix{\Omega}{t}_{uv}}}{\sqrt{{1 + \paren{\TimeDepMatrix{\Omega}{t}_{uv}}^2}}} \geq C'\sqrt{\frac{\log p + \zeta^2_{\pi_0, p, w}}{w}} \asymp \sqrt{\frac{\log p}{w}},
\end{equation*}	
where $C'$ is a bounded scalar depending on $\xi, d_{\max},$ and $\alpha_{\min}$. The new sufficient condition is based on $\ell_\infty$ norm of the difference between the pre- and post-change precision matrices.
\end{rem}

Next, we explore the sufficient condition \eqref{SuffDetectCond} in selected scenarios and compare it to the detection condition used in the test by Keshavarz et al. \cite{keshavarz2018sequential}. Note that the sequential algorithm in \cite{keshavarz2018sequential} is designed to detect changes affecting many edges of the GGM.

\begin{enumerate}[label = (\alph*),leftmargin=*]
\item Uniform change in $\TimeDepMatrix{\Omega}{t}$: Suppose that there exists $\beta > -1$ such that $\TimeDepMatrix{\Omega}{t+1} = \TimeDepMatrix{\Omega}{t}/\paren{1+\beta}$. Simply put, all the edges are affected the same way by the sudden change. In this case, 
\begin{equation*}
\LpNorm{\TimeDepMatrix{\Delta}{t}}{\infty} = \abs{\beta}\max_{1\leq u\leq v\leq p}\frac{\abs{\TimeDepMatrix{R}{t}_{uv}} }{\sqrt{1 + \paren{\TimeDepMatrix{R}{t}_{uv}}^2}} = \frac{\abs{\beta}}{\sqrt{2}}.
\end{equation*}
Thus, $T_t$ can detect any $\beta$ satisfying $\abs{\beta}\geq C\sqrt{\frac{\log p}{w}}$ (for some large enough scalar $C$) with high probability. In contrast, the procedure in \cite{keshavarz2018sequential}, which is obtained by applying a convex barrier function on the diagonal entries of $\TimeDepMatrix{\Delta}{t}$, detects a change point, whenever
\begin{equation*}
\beta - \log\paren{\beta+1}\geq C'\sqrt{\frac{\log p}{pw^2}},
\end{equation*}
for a bounded scalar $C'$. One can verify that the proposed algorithm in \cite{keshavarz2018sequential} outperforms $T_t$ and the gap between these two approaches increases as $p$ grows. For example, setting $w = \cc{O}\paren{\log p}$ yields a $\beta = \cc{O}\paren{\paren{p\log p}^{-1/4}}$ that is detectable by that algorithm, which is not the case with $T_t$. The main reason is that the test in \cite{keshavarz2018sequential} is designed for a global (albeit weak) change in the GGM, and this combines/aggregates the (possibly weak) signal across all edges, thus making it more suitable for detecting such uniform changes.
\item Change in a small sub-graph: In this case, the change point only affects edges related to a subset of nodes $\cc{S}\subset\set{1,\ldots, p}$. Particularly, we assume $\boldsymbol{\Theta} = \TimeDepMatrix{\Omega}{t+1}-\TimeDepMatrix{\Omega}{t}$ satisfies the following conditions:
\begin{align}\label{eq1}
& \supp\paren{\boldsymbol{\Theta}}\subset \cc{S}\times \cc{S},\nonumber\\
&\exists\;\xi\in\paren{0,1}\;\;\suchthat\;\; \OpNorm{\paren{\TimeDepMatrix{\Omega}{t}}^{-1/2}\boldsymbol{\Theta}\paren{\TimeDepMatrix{\Omega}{t}}^{-1/2}}{2}{2} \leq \xi < 1.
\end{align}
The second condition in Eq. \eqref{eq1} roughly indicates the presence of a weak change signal, compared to the background ($\TimeDepMatrix{\Omega}{t}$). Such a restriction on 
$\boldsymbol{\Theta}$ is realistic, due to its support constraint. Without loss of generality, we also assume that $\TimeDepMatrix{\Omega}{t}$ has unit diagonal entries. The algorithm in \cite{keshavarz2018sequential} detects a change point under this setting, if 
\begin{equation}\label{eq2}
-\frac{\sum_{s\in\cc{S}} \boldsymbol{\Theta}_{ss}}{\paren{1+\xi}p}-\log\paren{1-\frac{\sum_{s\in\cc{S}} \boldsymbol{\Theta}_{ss}}{\paren{1+\xi}p}} \geq C\sqrt{\frac{\log p}{pw^2}}.
\end{equation}
for a $C < \infty$. Note that asymptotically, wherein both $p$ and $\abs{\cc{S}}$ grow with $\abs{\cc{S}} = o\paren{p^\alpha}$ for some $\alpha\in\paren{0,\frac{3}{4}}$, the detection condition \eqref{eq2} does not hold if
\begin{equation*}
p^{\alpha-1} \LpNorm{\boldsymbol{\Theta}}{\infty} = o\paren{\sqrt[4]{\frac{\log p}{pw^2}}} \;\; \Longleftrightarrow\;\; \lambda_{w,p,\boldsymbol{\Theta}} \coloneqq \LpNorm{\boldsymbol{\Theta}}{\infty}\sqrt{\frac{w}{\log p}} = o\paren{\sqrt[4]{\frac{p^{3-4\alpha}}{\log p}}}.
\end{equation*}
In contrast, $T_t$ can detect an abrupt change satisfying  $\lambda_{w,p,\boldsymbol{\Theta}}=\cc{O}\paren{1}$, which is a considerably weaker restriction on $\lambda_{w,p,\boldsymbol{\Theta}}$. Namely, our proposed algorithm is more suitable for detecting \emph{localized} changes confined to small sub-graphs.

\end{enumerate}

\section{Asymptotic analysis of $\hat{T}_t$}\label{Section4}

Next, we study asymptotic properties of $\hat{T}_t$, introduced in Eq. \eqref{PluginProposedTest}. We demonstrate that $\hat{T}_t$, which is based on the plug-in statistic $\boldsymbol{\hat{E}}_{t,w}$ \eqref{PlugnStatistic}, (asymptotically) performs as well as the oracle test $T_t$ under mild regularity conditions. The analysis of $\hat{T}_t$ relies on certain large-sample properties of the error matrix $\TimeDepMatrix{\Pi}{t} = \TimeDepMatrix{\hat{\Omega}}{t}-\TimeDepMatrix{\Omega}{t}$. In particular, sharp bounds on $\OpNorm{\cdot}{\infty}{\infty}$ and $\LpNorm{\cdot}{\infty}$ norms of  $\TimeDepMatrix{\Pi}{t}$ in terms $p, d_{\max}$ are required, together with the number of samples since the last change point $N$. Such theoretical results are available for most computationally and statistically efficient sparse precision matrix estimation methods, such as the CLIME algorithm \cite{cai2011constrained}, or the QUIC approach \cite{hsieh2014quic}. 

For brevity, $\TimeDepMatrix{\hat{\Omega}}{t}$ is assumed throughout this section to be a symmetric matrix estimated by the CLIME procedure, and projected into the set of positive definite matrices. This projection is carried through by ignoring the components with negative eigenvalues in the eigen-decomposition of $\TimeDepMatrix{\hat{\Omega}}{t}$. Although the formulation of $\boldsymbol{\hat{E}}_{t,w}$ does not strictly require $\TimeDepMatrix{\hat{\Omega}}{t}$ to be positive definite, the positive definiteness is guaranteed (as well as having a bounded condition number as $p$ grows) with high probability, if $\TimeDepMatrix{\Omega}{t}$ satisfies Assumption \ref{AssuCLT}.

We begin by studying the null distribution of $\hat{T}_t$, when no abrupt change occurs between $t-N$ and $t+w-1$, i.e. $\TimeDepMatrix{\Omega}{t-N} = \ldots = \TimeDepMatrix{\Omega}{t} = \ldots = \TimeDepMatrix{\Omega}{t+w-1}$. The observed samples before $t$ ($\boldsymbol{X}_{t-i},\;i=1,\ldots, N$) are used for obtaining an estimate of the unknown ``background" precision matrix, and $w$ samples after $t-1$ for computing $\hat{T}_t$ (detection phase).

\begin{thm}\label{Thm3}
Let $\pi_0\in\paren{0,1}$. Assume that there is no change point between $t-N$ and $t+w-1$. Further, suppose that the following conditions hold as $w, p,$ and $N\rightarrow\infty$.
\begin{enumerate}[label = (\alph*),leftmargin=*]
\item $\TimeDepMatrix{\Omega}{t}$ satisfies Assumption \ref{AssuCLT}.
\item $w^{-1}\log^3 p \rightarrow 0$.
\item $N^{-1}w\log p\rightarrow 0$.
\end{enumerate}
Then,
\begin{equation*}
\bb{P}_{\FA}\paren{\hat{T}_t} = \bb{P}\paren{\LpNorm{\boldsymbol{\hat{E}}_{t, w}}{\infty}\geq \zeta_{\pi_0, p, w} } \rightarrow \pi_0.
\end{equation*} 
\end{thm}

Before proceeding further, we briefly outline the proof strategy for Theorem \ref{Thm3}. Recall that we set $\boldsymbol{Y}_{t+r} = \TimeDepMatrix{\Omega}{t}\boldsymbol{X}_{t+r}$ and $\boldsymbol{\hat{Y}}_{t+r} = \TimeDepMatrix{\hat{\Omega}}{t}\boldsymbol{X}_{t+r}$ in Section \ref{Methodology}. For ease of presentation, we also define
\begin{equation*}
\TimeDepMatrix{\Psi}{t} \coloneqq \brac{ \paren{\TimeDepMatrix{\Omega}{t}_{uu}\TimeDepMatrix{\Omega}{t}_{vv} + \paren{\TimeDepMatrix{\Omega}{t}_{uv}}^2}^{-1/2} }^p_{u,v=1},\quad\mbox{and}\quad \TimeDepMatrix{\hat{\Psi}}{t}\coloneqq \brac{ \paren{\TimeDepMatrix{\hat{\Omega}}{t}_{uu}\TimeDepMatrix{\hat{\Omega}}{t}_{vv} + \paren{\TimeDepMatrix{\hat{\Omega}}{t}_{uv}}^2}^{-1/2} }^p_{u,v=1}.
\end{equation*} 
Note that $\TimeDepMatrix{\hat{\Psi}}{t}$ is a random matrix depending on $\boldsymbol{X}_{t-i},\;i=1,\ldots, N$. Since both oracle and plug-in tests have the same critical value $\zeta_{\pi_0, p, w}$, we only need to show that 
\begin{equation*}
\LpNorm{\hat{\boldsymbol{E}}_{t,w}-\boldsymbol{E}_{t,w}}{\infty} = o_{\bb{P}}\paren{1}.
\end{equation*}
By applying the triangle inequality, we decompose the desired quantity into three terms. 
\begin{eqnarray}\label{Eq9}
\LpNorm{\hat{\boldsymbol{E}}_{t,w}-\boldsymbol{E}_{t,w}}{\infty} &\leq&  \LpNorm{\TimeDepMatrix{\hat{\Psi}}{t}-\TimeDepMatrix{\Psi}{t}}{\infty}\LpNorm{\sum_{r=1}^{w} \frac{\paren{\boldsymbol{Y}_{t+r}\boldsymbol{Y}_{t+r}^\top-\TimeDepMatrix{\Omega}{t}}}{\sqrt{w}}}{\infty}+w^{1/2}\LpNorm{\TimeDepMatrix{\hat{\Omega}}{t}\TimeDepMatrix{\Sigma}{t}\TimeDepMatrix{\hat{\Omega}}{t}-\TimeDepMatrix{\Omega}{t}}{\infty}\nonumber\\
&+& w^{-1/2}\LpNorm{\sum_{r=1}^{w} \paren{\boldsymbol{\hat{Y}}_{t+r}\boldsymbol{\hat{Y}}_{t+r}^\top-\boldsymbol{Y}_{t+r}\boldsymbol{Y}_{t+r}^\top-\TimeDepMatrix{\hat{\Omega}}{t}\TimeDepMatrix{\Sigma}{t}\TimeDepMatrix{\hat{\Omega}}{t}+\TimeDepMatrix{\Omega}{t}}}{\infty}.
\end{eqnarray}
Let $\clubsuit_1, \clubsuit_2$ and $\clubsuit_3$ denote the terms on the right hand side of Eq. \eqref{Eq9}, respectively. We show that
\begin{equation*}
\LpNorm{\TimeDepMatrix{\hat{\Omega}}{t}\TimeDepMatrix{\Sigma}{t}\TimeDepMatrix{\hat{\Omega}}{t}-\TimeDepMatrix{\Omega}{t}}{\infty} = \cc{O}_{\bb{P}}\paren{\sqrt{\frac{\log p}{N}}} \;\;\Longrightarrow\;\; \clubsuit_2 = \cc{O}_{\bb{P}}\paren{\sqrt{\frac{w\log p}{N}}} = o_{\bb{P}}\paren{1}.
\end{equation*}
Note that $\clubsuit_1$ is the product of two independent terms. For controlling  $\clubsuit_1$ from above, establish
\begin{equation*}
\LpNorm{\TimeDepMatrix{\hat{\Psi}}{t}-\TimeDepMatrix{\Psi}{t}}{\infty} = \cc{O}_{\bb{P}}\paren{\sqrt{\frac{\log p}{N}}},\quad\mbox{and}\quad \LpNorm{\sum_{r=1}^{w} \frac{\paren{\boldsymbol{Y}_{t+r}\boldsymbol{Y}_{t+r}^\top-\TimeDepMatrix{\Omega}{t}}}{\sqrt{w}}}{\infty} = \cc{O}_{\bb{P}}\paren{\sqrt{\log p}}.
\end{equation*}
Hence, $\clubsuit_1=o_{\bb{P}}\paren{1}$, if $N$ grows faster than $\log^2 p$. Finally, we establish $\clubsuit_3 = \cc{O}_{\bb{P}}\paren{\frac{\log p}{\sqrt{N}}}$, which implies that $\clubsuit_3$ tends to zero in probability.

\begin{rem}
The condition on $N$ in Theorem \ref{Thm3} ($N$ grows faster than $w$) seems counter-intuitive at first glance. Controlling the bias of estimating the oracle statistic $\boldsymbol{E}_{t,w}$ by $\boldsymbol{\hat{E}}_{t,w}$ is the major reason behind this observation. In particular, it is easy to verify that
\begin{equation}\label{Eq10}
\bb{E}\paren{\boldsymbol{\hat{E}}_{t,w}-\boldsymbol{E}_{t,w} \lvert \boldsymbol{X}_{t-1},\ldots, \boldsymbol{X}_{t-N}} = \sqrt{w}\paren{\TimeDepMatrix{\hat{\Omega}}{t}\TimeDepMatrix{\Sigma}{t}\TimeDepMatrix{\hat{\Omega}}{t}-\TimeDepMatrix{\Omega}{t}}\circ \TimeDepMatrix{\Psi}{t}.
\end{equation}
According to Eq. \eqref{Eq10}, the conditional bias of $\boldsymbol{\hat{E}}_{t,w}$ is proportional to $\sqrt{w}$. So for large $w$, even a slight bias introduced by the CLIME
estimator can change the null distribution of $\hat{T}_t$, which is counterbalanced by increasing $N$.

One the other hand, Theorem \ref{Thm2} suggests that increasing $w$ improves the detection power of our proposed algorithm. Therefore, the proper choice of $w$ is determined by the trade-off between the false alarm rate and the power of $\hat{T}_t$. Hence,
choosing $w = \cc{O}\paren{\log^{3+\delta} p}$ and $N = \cc{O}\paren{\log^{4+2\delta} p}$, for a small $\delta > 0$, represents a good choice in practical settings.
\end{rem}

\begin{rem}
The fact that $d_{\max}$ remains bounded in Theorem \ref{Thm3}, despite growing $p$, helps us highlight the main contribution without any distractions from technical over-complications.
Indeed, the proof of Theorem \ref{Thm3} can be extended to the case of $d_{\max}\rightarrow\infty$ (with $p$), if the second and third conditions in Theorem \ref{Thm3} are replaced by $w^{-1}d^{c_1}_{\max}\log^3 p \rightarrow 0$ and $N^{-1}wd^{c_2}_{\max}\log p\rightarrow 0$, for two appropriately chosen positive scalars $c_1$ and $c_2$.
\end{rem}

\begin{rem}
Note that for the detection procedure in \cite{keshavarz2018sequential}, designed for settings where many edges are impacted by the presence of a change point, 
$N$ grows at a faster rate than $pw\log^2 p$. Further, in offline settings, the detection algorithms for sparse precision matrices in \cite{atchade2017scalable,ross2014introduction} also require $\cc{O}\paren{s\log p}$ (or order $p\log p$ for bounded-degree networks) samples for estimating the location of the change point with an order $\log p$ error. Such restrictions on $N$ are significantly stronger than the conditions in Theorem \ref{Thm3}. Indeed, the aforementioned algorithms rely on the \emph{Frobenius} norm consistency in estimating $\TimeDepMatrix{\Omega}{t}$, as opposed to $\OpNorm{\cdot}{\infty}{\infty}$ norm consistency required for our proposed localized detection procedure, with the former leading to a more stringent condition on the growth of $N$.
\end{rem}

\section{Performance Evaluation}\label{Synthetic}

The first numerical experiment gauges the sensitivity of the critical value $\zeta_{\pi_0, p, w}$ to the estimation error of the pre-change precision matrix. It provides insights on the behavior of $\hat{T}_t$ under a no-change scenario to the oracle test statistic $T_t$. We consider a GGM with $p=80$ nodes (its precision matrix is denoted by $\boldsymbol{\Omega}$). The entries of $\boldsymbol{\Omega}$ are generated according to (using \texttt{R fastclime} package): (i) each row has 6\% of non-zero entries (including the diagonal ones), and (ii) the precision has four hub nodes, each of them connected to $20$ other nodes at random. Furthermore, all nodes are normalized to have unit variance. To ensure positiveness, after the initial generation of $\boldsymbol{\Omega}$, its diagonal entries are inflated by $0.1$.  $N=300$ samples are used for estimating $\boldsymbol{\Omega}$ in both scenarios. The normalized estimation error, defined as 
\begin{equation}\label{NEE}
\hat{e}_N \coloneqq \LpNorm{\boldsymbol{\Omega}}{2}^{-1}\LpNorm{\hat{\boldsymbol{\Omega}}-\boldsymbol{\Omega}}{2},
\end{equation}
is equal to $0.30$ for mechanism (i) and $0.41$ for mechanism (ii). The detection delay parameter is set to $w=50$. Figure \ref{Fig:Fig2} depicts the distribution of $\LpNorm{\boldsymbol{E}_{t, w}}{\infty}$, based on $10^5$ independent replicates, under the no-change scenario for the two data generation mechanisms. Note that in each experiment, $\boldsymbol{E}_{t, w}$ is formed by $w$ i.i.d. samples drawn from a zero-mean Gaussian vector with precision matrix $\boldsymbol{\Omega}$. 
\begin{center}\label{Fig:Fig2}
\begin{figure}
\includegraphics[scale=0.5]{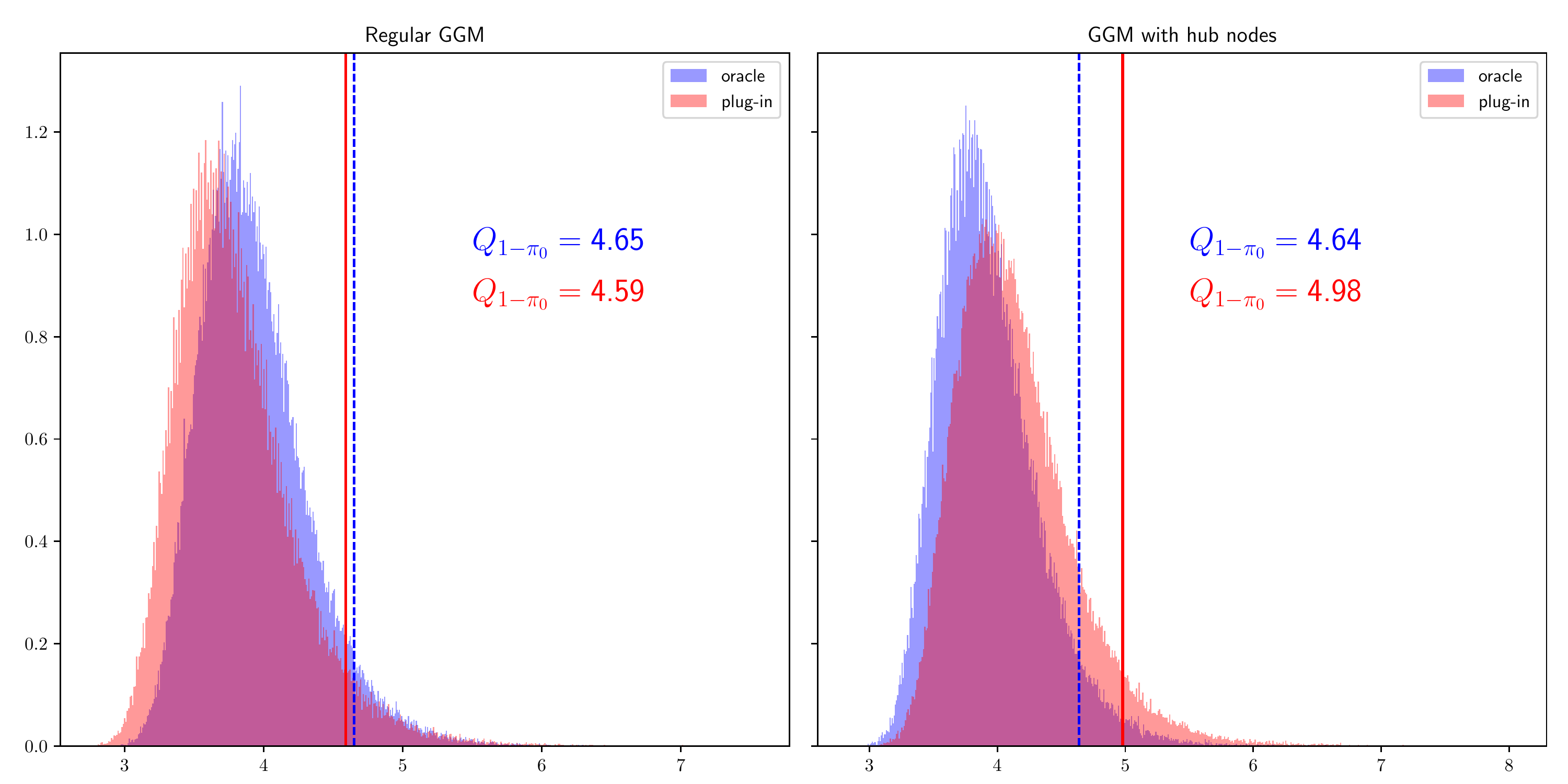}
\caption{Histogram of $\LpNorm{\boldsymbol{E}_{t, w}}{\infty}$ over $10^5$ independent replicates. The blue dashed and solid red lines correspond to the $\paren{1-\pi_0}$-quantile of the oracle and the plug-in test statistic under the null hypothesis, respectively.}
\label{fig:imageb}
\end{figure}
\end{center}
The plots in Figure \ref{Fig:Fig2} show that:
\begin{itemize}
\item Despite having a non-negligible error in estimating the pre-change precision matrix, the oracle and plug-in distributions (and hence their corresponding critical values) of the proposed test statistic are very close for mechanism (i), thus demonstrating the efficacy of the proposed test statistic in such settings.
\item The gap between the oracle and plug-in critical values increases, as the estimate of the pre-change precision matrix becomes less accurate, which is an inevitable consequence of dealing with a more challenging estimation problem.
\end{itemize}

Next, we assess the impact of the burn-in period ($N$) on the gap between the oracle and plug-in critical values. For this task, we use mechanism (i) for generating the precision matrix of size $p=80$ and set $w=40$. The oracle critical value $\zeta_{\pi_0, p, w}$ is based on $10^5$ independent replicates. For obtaining the distribution of the plug-in test statistic $\LpNorm{\hat{\boldsymbol{E}}_{t, w}}{\infty}$, we estimate the pre-change precision matrix based on $N=200, 300, \cdots, 700$ samples. Note that if both the plug-in and oracle critical values are the same, then
\begin{equation*}
p_N\coloneqq \bb{P}\paren{\LpNorm{\hat{\boldsymbol{E}}_{t, w}}{\infty}\leq \zeta_{\pi_0, p, w}} = 1-\pi_0 = 0.95.
\end{equation*}

We numerically calculate $p_N$ based on $10^5$ independent replicates, as well. Table \ref{Tab:Table1} shows $p_N$ and $\hat{e}_N$ for different values of $N$. A careful look at the results in Table \ref{Tab:Table1} indicates that a larger $N$ leads to a smaller estimation error $\hat{e}_N$ and a value for $p_N$ closer to $0.95$. Hence, this numerical experiment shows that the gap between the oracle and the plug-in test statistic becomes smaller, when the number of samples during the burn-in period becomes larger.
\vspace{4mm}
\begin{center} \label{Tab:Table1}
\begin{adjustbox}{width=0.9\columnwidth}
\begin{tabular}{c|c|c|c|c|c|c|c|}
\cline{2-8} 
& $N = 200$ & $N=300$ & $N=400$ & $N=500$ & $N=600$ & $N=700$ & Oracle\\ 
\cline{1-8}
\multicolumn{1}{ |c| }{$\hat{e}_N$} & $45.07\%$ & $32.41\%$ & $32.01\%$ & $27.42\%$ & $25.57\%$ & $21.48\%$ & $\textcolor{red}{\textendash}$\\
\cline{1-8}
\multicolumn{1}{ |c| }{$p_N$} & $77.79\%$ & $90.39\%$ & $92.81\%$ & $95.09\%$ & $95.03\%$ & $94.77\%$ & $\textcolor{red}{95\%}$\\
\cline{1-8}
\end{tabular}
\end{adjustbox}
\captionof{table}{$\hat{e}_N$ and $p_N$ of the plug-in test in the no-change regime for different values of $N$}
\label{Table1}
\end{center}

In the remainder of the section, we focus on gaining insights on the performance of the proposed detection test for locating a change point; i.e. under $\bb{H}_{1,t}$ in Eq. \eqref{HypothesisTestingFormulation}. We generate a zero-mean time-varying GGM $\set{\boldsymbol{X}_i}^{T+N}_{i=1}$ with $p$ nodes as follows.
\begin{equation}\label{SimulSet1}
\cov\paren{\boldsymbol{X}_i} = \boldsymbol{\Omega}_{pre} \bbM{1}_{\set{i\leq N+t_0}} + \boldsymbol{\Omega}_{post} \bbM{1}_{\set{i\geq N+t_0}},\quad i\in\set{1,\ldots, T+N}.
\end{equation}
$N$ samples from the pre-change precision matrix correspond to the burn-in period. We also generate $T$ samples for the detection procedure. We assume that a sudden change occurs in the precision matrix at time $t_0$ after the burn-in period. Both pre- and post- change point precision matrices, which are denoted by $\boldsymbol{\Omega}_{pre}$ and $\boldsymbol{\Omega}_{post}$ are generated according to mechanism (i) with $0.04 p$ non-zero entries per row (excluding the diagonal one). Similar to the first simulation study, we add $0.1$ to the diagonal entries for controlling the condition number and we standardize the diagonal entries of the covariance matrix (unit variance). The remaining model parameters are selected as follows:
\begin{enumerate}
\item $p=100$, resulting in $\boldsymbol{\Omega}_{pre}$ having $302$ non-zero entries on average. There are $N=300$ samples for estimating $\boldsymbol{\Omega}_{pre}$ during the burn-in period. The normalized estimation error, defined in Eq. \eqref{NEE}, is equal to $0.29$. We also set $t_0 = 100$, $w=75$ and $T=t_0+w$.
\item $p=150$, resulting in $\boldsymbol{\Omega}_{pre}$ having  $598$ unknown non-zero entries on average, and set $w=100$ and $\pi_0 = 0.05$. We choose $N=600$ samples for estimating $\boldsymbol{\Omega}_{pre}$ in the burn-in period. The normalized estimation error, defined in Eq. \eqref{NEE}, is equal to $0.28$. We also set $t_0 = 150$, $w=100$ and $T=t_0+w$.
\end{enumerate}

\begin{center}
\begin{figure}
\includegraphics[scale=0.45]{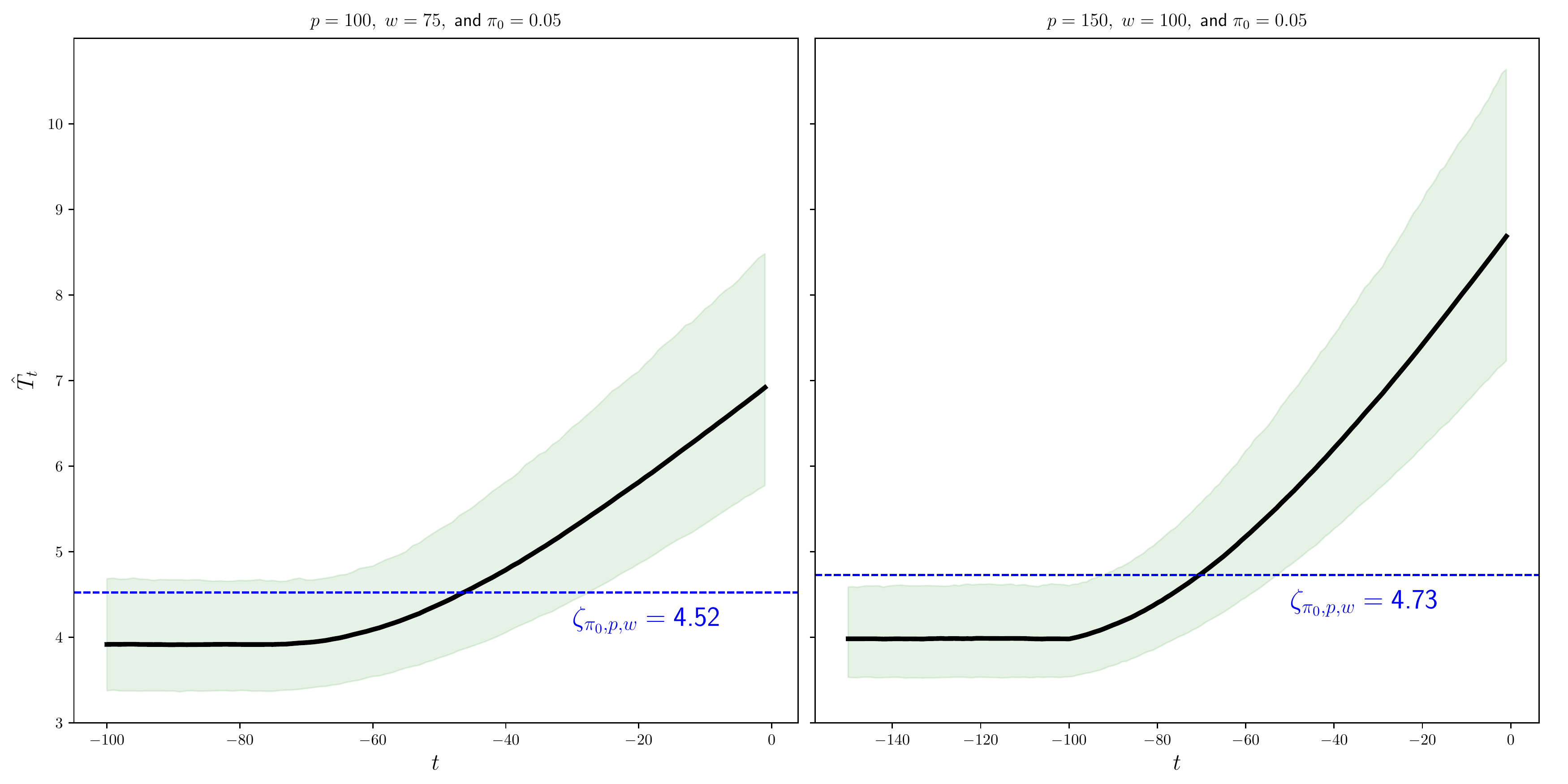}
\caption{In each panel, the solid black curve represents the average value of  $\LpNorm{\boldsymbol{E}_{t, w}}{\infty}$ over $10^4$ independent experiments and the blue dashed line shows $\zeta_{\pi_0, p, w}$.}
\label{fig:imagec}
\end{figure}
\end{center}

Since the focus is on detection, the estimated pre-change precision matrix is not updated for avoiding unnecessary complexity. For each of the two aforementioned scenarios, we independently repeat the process of generating $T$ samples $10^4$ times. For simplicity, consider the time interval $t\in\set{-t_0,1-t_0,\ldots, 0}$. The two panels in Figure \ref{fig:imagec} present the mean plug-in statistic time series $\LpNorm{\hat{\boldsymbol{E}}_{t, w}}{\infty},\;-t_0\leq t\leq 0$, as well as the confidence interval around it, as a function of $t$. Notice that $\hat{\boldsymbol{E}}_{t, w}$ is determined by the estimated pre-change precision matrix and generated samples $\boldsymbol{X}_i,\;i=t+1,\ldots, t+w$. Therefore, the detection statistic does not utilize any samples from the post-change regime, as long as $t < -w$. In contrast, it only utilizes samples after the change point, whenever $t\geq 0$. This fact is clearly shown in Figure \ref{fig:imagec}. Our proposed test statistic starts below the critical value $\zeta_{\pi_0, p, w}$ for $t=-t_0$ and gradually increases as $t$ grows. It is also apparent that the detection delay is indeed less than $w$, as the average time series crosses $\zeta_{\pi_0, p, w}$ before $t=0$. In particular, the proposed algorithm detects the change point only after observing $29$ and $27$ samples (on average) after the change point in the left and right panels, respectively.

The asymptotic results in Section \ref{Section3.2} manifest the advantages of the proposed procedure for identifying sudden changes that affect only a small number of edges in large GGMs. The next simulation study aims to corroborate our previous asymptotic understanding. We again consider a time-varying sparse graphical model comprising of $p=100$ nodes, where its dependence structure goes through a sudden change according to the model in Eq. \eqref{SimulSet1}. The pre-change GGM is generated according to mechanism (i) with $5\%$ non-zero entries per row, and is initially estimated by $N=300$ independent samples during the burn-in period. The detection delay is set to $w=150$. For brevity, we use $\boldsymbol{\Delta}$ for referring to 
$ \boldsymbol{\Omega}_{post} - \boldsymbol{\Omega}_{pre}$. Choose $s$ from $\set{1,2,3}$ and  construct $\boldsymbol{\Delta}$ as follows:
\begin{equation}\label{SimulSet2}
\boldsymbol{\Delta}_{ij} = \frac{\beta}{s} \bbM{1}_{\set{1\leq i, j \leq s}},
\end{equation}
where $\beta$ is a positive number. Obviously, $\supp\paren{\boldsymbol{\Delta}} = \set{1,\ldots,s}\times \set{1,\ldots, s}$ and $\LpNorm{\boldsymbol{\Delta}}{2} = \beta$, which is independent of $s$. Hence, the change point only affects $s$ nodes in the network. The parameter $\beta$ in Eq. \eqref{SimulSet2} is needed for controlling the intensity of the signal that induces the change point. Further, for a fixed $\beta$, increasing $s$ leads to a more distributed change point, as it affects more nodes without increasing the signal-to-noise-ratio (SNR). For each fixed set of parameters $s$ and $\beta$, $10^4$ independent replicates are used to approximate the mis-detection rate $\pi_1$ of the plug-in test statistic. Figure \ref{fig:imaged} depicts $\pi_1$ as a function of $\beta$ for different values of $s$. The summary results in Figure \ref{fig:imaged} illustrate two facts. First, the detection power increases (lower $\pi_1$) for larger SNR (larger $\beta$), that is aligned with the insights from our asymptotic analysis. Moreover, our method performs better for detecting change points that are confined to a small sub-graph of the precision matrix (smaller $s$).  
\begin{center}
\begin{figure}
\includegraphics[scale=0.45]{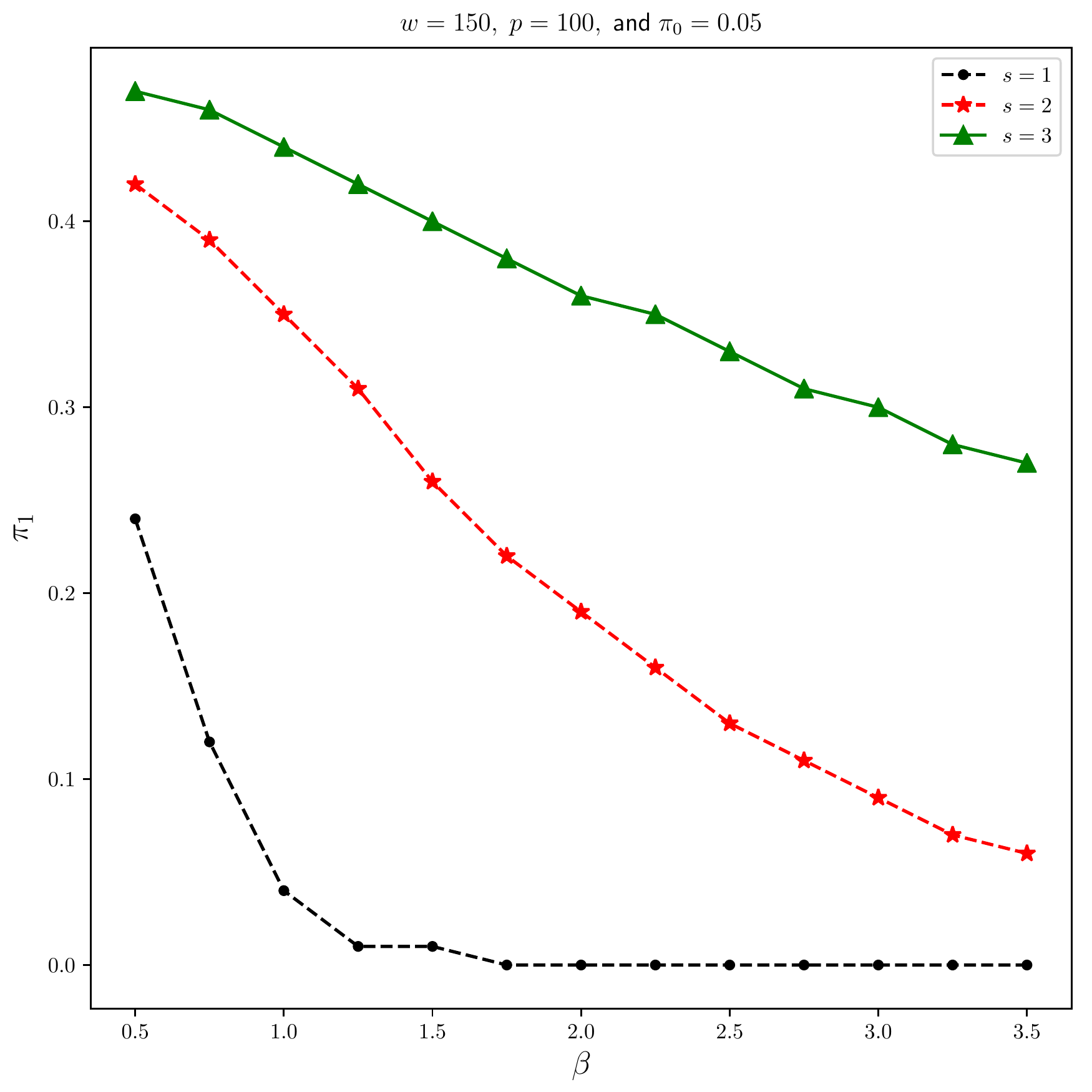}
\caption{$\pi_1$ versus $\beta$ in the change point model \eqref{SimulSet2}. The black, red, and green lines refer to $s=1, 2,$ and $3$, respectively. For any pair $\paren{\beta, s}$, $\pi_1$ is approximated by $10^4$ independent replicates.}
\label{fig:imaged}
\end{figure}
\end{center}

Next, we study the role of $w$ on the mis-detection of the proposed plug-in test. For doing so, we fix $\beta = 3$ and choose $s\in\set{1,2,3}$ in Eq. \eqref{SimulSet2}. We also increase $w$ from $60$ to $300$. Again, $10^4$ independent replicates are used to approximate $\pi_1$. Figure \ref{fig:imagef} exhibits $\pi_1$ versus $w$ for different values of $s$. It is apparent from Figure \ref{fig:imagef} that increasing delay $w$ reduces $\pi_1$, which confirms our asymptotic understanding.

\begin{center}
\begin{figure}
\includegraphics[scale=0.45]{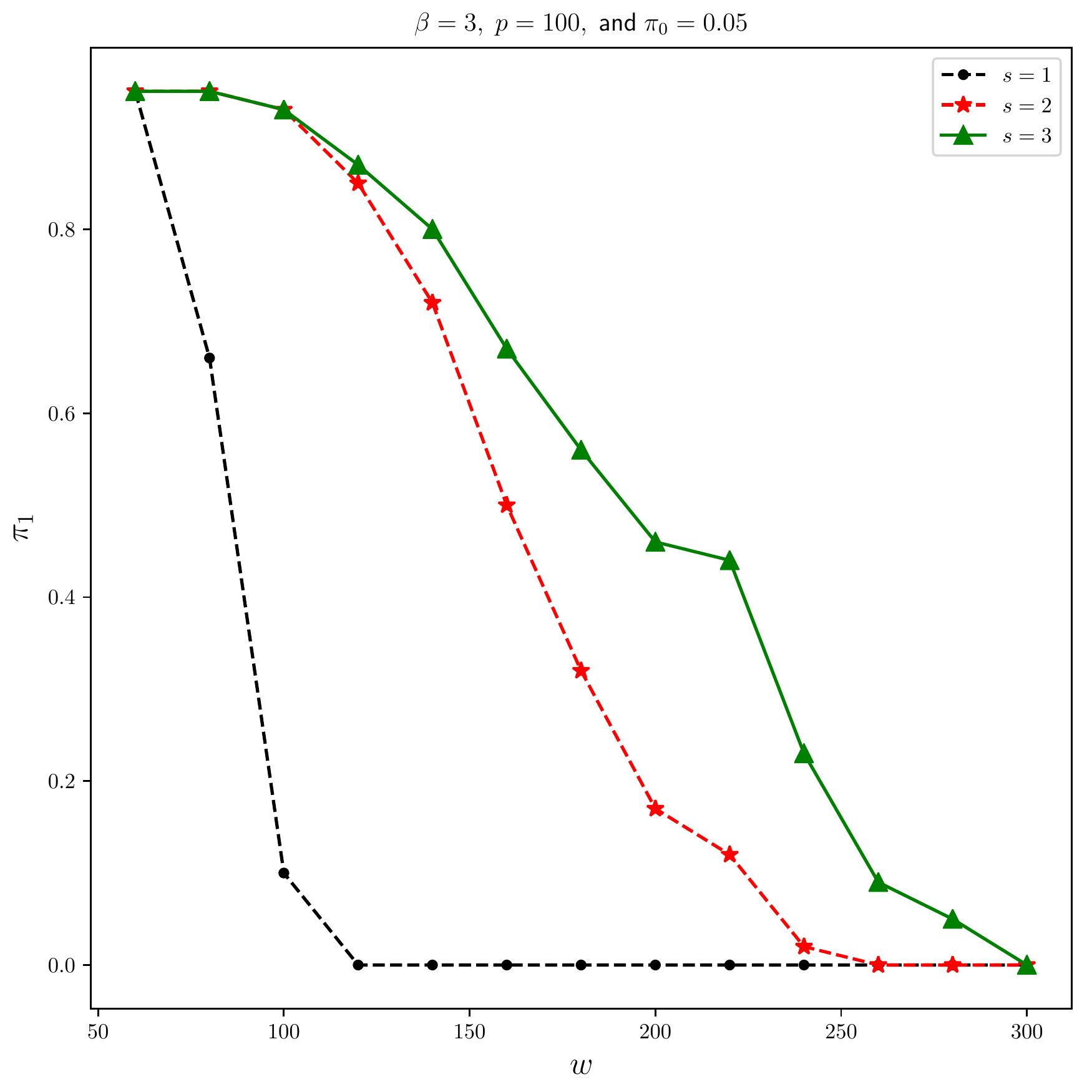}
\caption{$\pi_1$ versus $w$ in the change point model \eqref{SimulSet2}. The black, red, and green lines refer to $s=1, 2,$ and $3$, respectively. For any pair $\paren{s, w}$, $\pi_1$ is approximated by $10^4$ independent replicates.}
\label{fig:imagef}
\end{figure}
\end{center}

We conclude this section by comparing the performance of our proposed test, which we call it \emph{Local Change Point Detector (LCPD)}, to the procedure in \cite{keshavarz2018sequential}, designed for detecting changes affecting many nodes in the network. The test in \cite{keshavarz2018sequential} aggregates the signal across all nodes in the network, so it is referred to as \emph{Aggregated Change Point Detector (ACPD)}. Note that the comparison is based upon oracle settings, in order to avoid any distortions due to estimation errors for the pre-change precision matrix. A precision matrix comprising of $p=100$ nodes is generated according to mechanism (i) resulting in $236$ edges, for a total of $336$ distinct non-zero entries in the precision matrix, including the diagonal ones. The identity matrix is added to the generated precision matrix to avoid a small condition number. We also set $w=100$ and approximate $\pi_1$ based on $10^4$ replicates. Let $\boldsymbol{\Omega}_{pre}$ denote the pre-change precision matrix, and define $\boldsymbol{\Delta}$ as in the previous numerical study in this section. We assume that $\boldsymbol{\Delta}$ is given by
\begin{equation*}
\boldsymbol{\Delta} = \beta\begin{bmatrix}
\zero_{s\times \paren{p-s}} & & \boldsymbol{I}_s \\
& \zero_{\paren{p-2s}\times p}  &\\
\boldsymbol{I}_s & & \zero_{s\times \paren{p-s}}
\end{bmatrix}.
\end{equation*}
In words, $s$ new edges are added to the precision matrix after the change point. Similar to the previous numerical study, the size of the smallest sub-graph encompassing the nodes affected by the abrupt change is quantified by $s$. In this study, we assume that $s\in\set{1,5,20}$. Figure \ref{fig:imagee} presents $\pi_1$ as a function of $\beta$ for both ACPD and LCPD, with increasing $s$ from left to right. Figure \ref{fig:imagee} suggests that $\pi_1$ converges faster to zero for the LCPD and the gap between the two tests decreases as $s$ increases to $20$. This observation confirms the claim that aggregation over all nodes can be detrimental for detecting sudden changes affecting a small sub-set of nodes (and thus edges) in the precision matrix of a GGM.

\begin{center}
\begin{figure}
\includegraphics[scale=0.4]{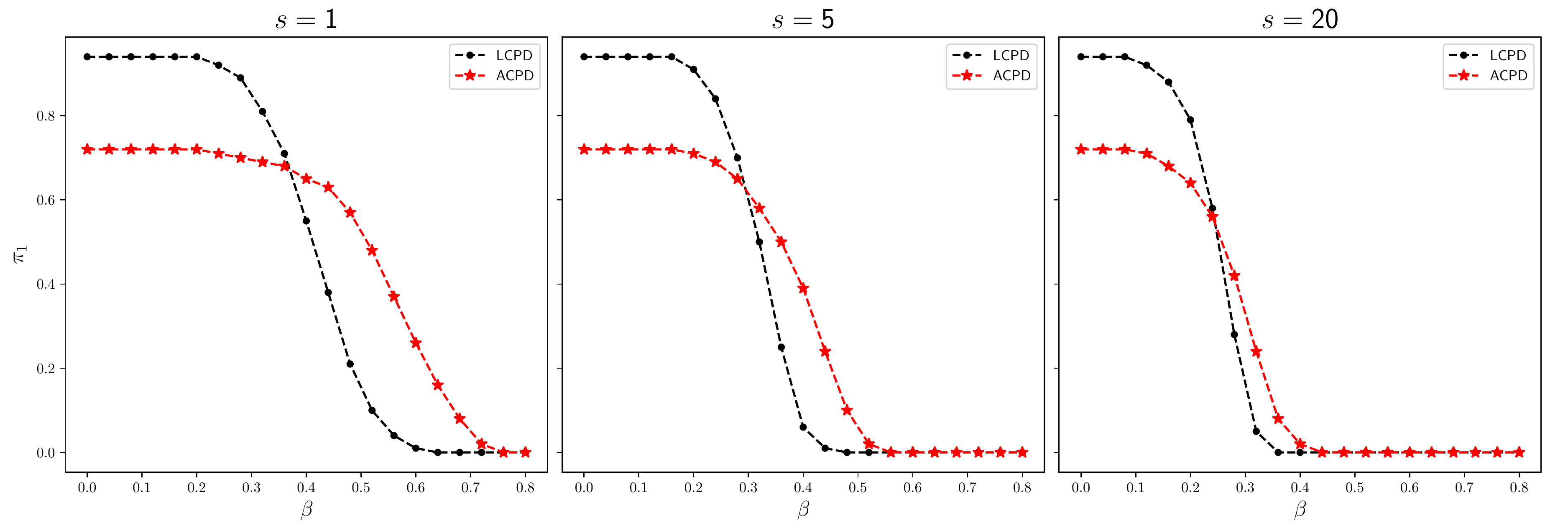}
\caption{Comparing the detection power of the proposed test (LCPD) with the algorithm in \cite{keshavarz2018sequential} (ACPD) for three different scenarios. In each panel, black and red solid lines refer to the LCPD and ACPD, respectively. For any pair $\paren{\beta, s}$, $\pi_1$ is approximated from $10^4$ simulations.}
\label{fig:imagee}
\end{figure}
\end{center}

\section{Concluding Remarks}\label{Discussion}

The paper studies the problem of sequential detection of abrupt changes in the precision matrix of sparse high-dimensional GGMs, whenever such changes impact few edges only. 
The analysis of the distribution of the test statistic $T_t$ under the null and the alternative hypotheses relies on extreme value theory for dependent random variables. An approach 
based on technical tools already used in the literature for two sample tests for covariance matrices leads to exceedingly stringent conditions. Instead, we develop novel techniques
leveraging Galambos' technique that provides the distribution of the maximum of graph-dependent random variables, that require mild regularity conditions, and renders the detection
procedure widely applicable. Note that these novel techniques are of independent interest and potentially applicable in other problems involving network data.
The numerical experiments provide strong evidence in support of the theoretical developments and confirm the good performance of the proposed change point detection procedure.

\section{Proofs}\label{Proofs}

This section contains the proofs of our main results. Recall from Eq. \eqref{Sbartw} that 
\begin{equation*}
\boldsymbol{E}_{t, w} = \sum_{r=1}^{w}\frac{ \paren{\boldsymbol{Y}_{t+r}\boldsymbol{Y}_{t+r}^\top-\TimeDepMatrix{\Omega}{t}}}{\sqrt{w}} \circ \brac{ \paren{\TimeDepMatrix{\Omega}{t}_{uu}\TimeDepMatrix{\Omega}{t}_{vv} + \paren{\TimeDepMatrix{\Omega}{t}_{uv}}^2}^{-1/2} }^p_{u,v=1},
\end{equation*} 
where $\set{\boldsymbol{Y}_{t+r}:\;r=1,\ldots,w}$ are $\cc{N}\paren{\zero_p,\TimeDepMatrix{\Omega}{t}}$ distributed, in the absence of a break between $t+1$ and $t+w$. Throughout this section, $Z$ represents a $\cc{N}\paren{0,1}$ random variable. Without loss of generality, we assume that diagonal entries of $\TimeDepMatrix{\Omega}{t}$ are equal to $1$. We interchangeably use $\boldsymbol{\Omega}$, $\boldsymbol{E}$ and $\boldsymbol{Y}_r$ instead of $\TimeDepMatrix{\Omega}{t}$, $\boldsymbol{E}_{t, w}$ and $\boldsymbol{Y}_{t+r}$ below. Finally, define the set $\cc{K}_{p}$ by 
\begin{equation*}
\cc{K}_p = \set{\paren{r,s}:\;1\leq r\leq s\leq p}.
\end{equation*}

\begin{proof}[Proof of Theorem \ref{thm0}]
The objective is to obtain $\zeta_{\pi_0, p, w}$ such that
\begin{eqnarray*}
\bb{P}_{\FA}\paren{T_t} &=& \bb{P}\paren{\max_{\paren{r,s}\in\cc{K}_p} \abs{\boldsymbol{E}_{rs}}\geq \zeta_{\pi_0, p, w}} = \bb{P}\brac{\paren{\max_{\paren{r,s}\in\cc{K}_p} \boldsymbol{E}_{rs}\geq \zeta_{\pi_0, p, w}} \bigcup \paren{\max_{\paren{r,s}\in\cc{K}_p} -\boldsymbol{E}_{rs}\geq \zeta_{\pi_0, p, w}}}\\
&\leq& \pi_0\brac{1+o\paren{1}}.
\end{eqnarray*}
Due to the union bound, we only need to show that $\bb{P}\paren{\max_{\paren{r,s}\in\cc{K}_p} \boldsymbol{E}_{rs}\geq \zeta_{\pi_0, p, w}} \leq \frac{\pi_0}{2}\brac{1+o\paren{1}}$. Let $\boldsymbol{G}$ be a symmetric centered Gaussian random matrix with the same correlation structure as $\boldsymbol{E}$, i.e. $\boldsymbol{G}\in S^{p\times p}$ with
\begin{equation*}
\cov\paren{\boldsymbol{G}_{rs}, \boldsymbol{G}_{r's'}} = \cov\paren{\boldsymbol{E}_{rs}, \boldsymbol{E}_{r's'}},\quad \paren{r,s},\paren{r',s'}\in\cc{K}_p.
\end{equation*}
It is sufficient to establish the following claims.
\begin{clawithinpf}\label{Claim1Thm0}
$\bb{P}\paren{\max_{\paren{r,s}\in\cc{K}_p} \boldsymbol{G}_{rs}\geq \zeta_{\pi_0, p, w}} \rightarrow \frac{\pi_0}{2}$.
\end{clawithinpf}
	
\begin{clawithinpf}\label{Claim2Thm0}
$\max_{\paren{r,s}\in\cc{K}_p} \boldsymbol{E}_{rs}-\max_{\paren{r,s}\in\cc{K}_p} \boldsymbol{G}_{rs} \cp{\bb{P}} 0$.
\end{clawithinpf}
We proceed by proving Claim \ref{Claim1Thm0}. Set $q=\abs{\cc{K}_p} = p\paren{p+1}/2$. Let $\boldsymbol{Q}\in\bb{R}^{q\times q}$ denote the covariance matrix of $\brac{\boldsymbol{G}_{rs}:\;\paren{r,s}\in\cc{K}_p}^\top$. We obtain a closed formulation for the entries of $\boldsymbol{Q}$ as follows. Choose two arbitrary pairs $\paren{r,s}, \paren{r',s'}\in\cc{K}_p$. Since entries of $\boldsymbol{E}$ comprise of a summation of $w$ i.i.d. random variables, an application Isserlis' Theorem yields
\begin{eqnarray*}
\cov\paren{\boldsymbol{G}_{rs}, \boldsymbol{G}_{r's'}} &=& \cov\brac{\frac{\sum_{l=1}^{w}\paren{\boldsymbol{Y}_l}_r \paren{\boldsymbol{Y}_l}_s}{\sqrt{w\paren{1+\boldsymbol{\Omega}^2_{rs}}}}, \frac{\sum_{l=1}^{w} \paren{\boldsymbol{Y}_l}_{r'} \paren{\boldsymbol{Y}_l}_{s'}}{\sqrt{w\paren{1+\boldsymbol{\Omega}^2_{r's'}}}}} \\
&=& \sum_{l=1}^{w} \cov\brac{\frac{\paren{\boldsymbol{Y}_l}_r \paren{\boldsymbol{Y}_l}_s}{\sqrt{w\paren{1+\boldsymbol{\Omega}^2_{rs}}}}, \frac{ \paren{\boldsymbol{Y}_l}_{r'} \paren{\boldsymbol{Y}_l}_{s'}}{\sqrt{w\paren{1+\boldsymbol{\Omega}^2_{r's'}}}}} = \frac{\boldsymbol{\Omega}_{rr'}\boldsymbol{\Omega}_{ss'}+\boldsymbol{\Omega}_{rs'}\boldsymbol{\Omega}_{r's}}{\sqrt{\paren{1+\boldsymbol{\Omega}^2_{rs}}\paren{1+\boldsymbol{\Omega}^2_{r's'}}}}.
\end{eqnarray*}
Based on Assumption \ref{AssuCLT}, $\OpNorm{\boldsymbol{\Omega}}{1}{1}$ remains below a fixed scalar $d_{\max}$, as $p\rightarrow\infty$. We aim to prove a similar property for $\boldsymbol{Q}$. Particularly, we show that $\OpNorm{\boldsymbol{Q}}{1}{1} \leq 2\OpNorm{\boldsymbol{\Omega}}{1}{1}^2$. Opt any $\paren{r,s}\in\cc{K}_p$. Then,
\begin{eqnarray}\label{Eq1Thm0}
\sum_{{\paren{r',s'}\in\cc{K}_p}} \abs{\cov\paren{\boldsymbol{G}_{rs}, \boldsymbol{G}_{r's'}}} &=& \sum_{{\paren{r',s'}\in\cc{K}_p}} \frac{\abs{\boldsymbol{\Omega}_{rr'}\boldsymbol{\Omega}_{ss'}+\boldsymbol{\Omega}_{rs'}\boldsymbol{\Omega}_{r's}}}{\sqrt{\paren{1+\boldsymbol{\Omega}^2_{rs}}\paren{1+\boldsymbol{\Omega}^2_{r's'}}}} \leq \sum_{{\paren{r',s'}\in\cc{K}_p}} \abs{\boldsymbol{\Omega}_{rr'}\boldsymbol{\Omega}_{ss'}+\boldsymbol{\Omega}_{rs'}\boldsymbol{\Omega}_{r's}}\nonumber\\
&\leq& \sum_{{\paren{r',s'}\in\cc{K}_p}} \abs{\boldsymbol{\Omega}_{rr'}\boldsymbol{\Omega}_{ss'}}+\abs{\boldsymbol{\Omega}_{rs'}\boldsymbol{\Omega}_{r's}} \leq \sum_{r',s'=1}^{p} \abs{\boldsymbol{\Omega}_{rr'}\boldsymbol{\Omega}_{ss'}}+\abs{\boldsymbol{\Omega}_{rs'}\boldsymbol{\Omega}_{r's}}\nonumber\\
&=& 2\sum_{r'=1}^{p} \abs{\boldsymbol{\Omega}_{rr'}}\sum_{s'=1}^{p} \abs{\boldsymbol{\Omega}_{ss'}} = 2\LpNorm{\boldsymbol{\Omega}\boldsymbol{e}_{r}}{1}\LpNorm{\boldsymbol{\Omega}\boldsymbol{e}_{s}}{1}.
\end{eqnarray}
Eq. \eqref{Eq1Thm0} establishes the claim, since 
\begin{equation*}
\OpNorm{\boldsymbol{Q}}{1}{1} = \max_{\paren{r,s}\in\cc{K}_p} \brac{\sum_{{\paren{r',s'}\in\cc{K}_p}} \abs{\cov\paren{\boldsymbol{G}_{rs}, \boldsymbol{G}_{r's'}}}} \leq 2\max_{\paren{r,s}\in\cc{K}_p} \LpNorm{\boldsymbol{\Omega}\boldsymbol{e}_{r}}{1}\LpNorm{\boldsymbol{\Omega}\boldsymbol{e}_{s}}{1} \leq 2\OpNorm{\boldsymbol{\Omega}}{1}{1}^2\leq 2d^2_{\max}.
\end{equation*}
Given that $\OpNorm{\boldsymbol{Q}}{1}{1}$ remains bounded as $p$ grows (soft sparsity), known results on the extreme value of dependent Gaussian random variables (see Lemma $6$ of \cite{cai2013two}) implies that
\begin{equation}\label{Eq2Thm0}
\bb{P}\paren{\max_{\paren{r,s}\in\cc{K}_p} \boldsymbol{G}_{rs}\geq \sqrt{2\log q-\log\log q+x}} \rightarrow 1-\exp\brac{\frac{-e^{-x/2}}{2\sqrt{\pi}}}, \quad\forall\;x\in\bb{R}.
\end{equation}
We conclude the proof by choosing $x$ so that $1-\exp\brac{\frac{-e^{-x/2}}{2\sqrt{\pi}}} = \frac{\pi_0}{2}$. Then, we can rewrite Eq. \eqref{Eq2Thm0} as
\begin{equation*}
\bb{P}\paren{\max_{\paren{r,s}\in\cc{K}_p} \boldsymbol{G}_{rs}\geq \sqrt{2\log q-\log\log q-2\log\brac{2\sqrt{\pi}\log\paren{\frac{1}{1-\pi_0/2}}}}} \rightarrow \frac{\pi_0}{2}.
\end{equation*}
Lastly, we pick $\zeta_{\pi_0, p, w}$ by
\begin{equation*}
\zeta^2_{\pi_0, p, w} = 2\log q-\log\log q-2\log\brac{2\sqrt{\pi}\log\paren{\frac{1}{1-\pi_0/2}}}.
\end{equation*}
Next, we focus on the proof of Claim \ref{Claim2Thm0}. We utilize Theorem $4.1$ in Chernozhukov et al. \cite{chernozhukov2014gaussian}, which is about a Gaussian approximation of maxima of zero-mean empirical processes. For completeness, we state this result before proceeding further. 
\begin{thm}\label{ThmCher} (Theorem $4.1$ \cite{chernozhukov2014gaussian}) Let $\set{\boldsymbol{U}_i}^w_{i=1}$ be i.i.d. zero-mean random vectors in $\bb{R}^q$ with finite absolute third moments, i.e. $\max_{j=1,\ldots, q}\bb{E}\abs{\boldsymbol{U}_{1j}}^3 < \infty$. Let $\set{\boldsymbol{Z}_i}^w_{i=1}$ be a set of i.i.d. centered Gaussian random vectors in $\bb{R}^q$ with $\cov\boldsymbol{Z}_i = \cov\boldsymbol{U}_i = \boldsymbol{\Lambda},\;\forall\; i=1,\ldots,w$. Consider the following random objects
\begin{equation*}
\Xi_w = \max_{1\leq j\leq q} \sum_{i=1}^{w} \boldsymbol{U}_{ij},\quad\mbox{and}\quad \tilde{\Xi}_w = \max_{1\leq j\leq q} \sum_{i=1}^{w} \boldsymbol{Z}_{ij}.
\end{equation*}
Then, 
\begin{equation}\label{Thm41Cher}
\bb{P}\paren{\abs{\Xi_w-\tilde{\Xi}_w} > \delta} \lesssim \frac{\log w}{w} + \frac{B_1\log\paren{q+w} }{\delta^2} + \frac{B_2\log^2\paren{q+w}}{\delta},
\end{equation}
where $B_1$ and $B_2$ are given by
\begin{equation}\label{B1B2}
B_1\coloneqq \bb{E}\paren{\max_{1\leq j\leq q} \sum_{i=1}^{w} \abs{\boldsymbol{U}_{ij}}^3}, \quad \mbox{and}\quad B_2 = \bb{E}\paren{\max_{1\leq j, k\leq q} \abs{\sum_{i=1}^{w} \boldsymbol{U}_{ij}\boldsymbol{U}_{ik}-w\boldsymbol{\Lambda}_{jk}}}.
\end{equation}
\end{thm}
We proceed by defining $q=\abs{\cc{K}_p}$ dimensional i.i.d. random vectors $\boldsymbol{U}_i,\;j=1,\ldots, w$ as
\begin{equation}\label{UVec}
\boldsymbol{U}_i = \vect\brac{\frac{ \paren{\boldsymbol{Y}_{t+i}}_r\paren{\boldsymbol{Y}_{t+i}}_s-\boldsymbol{\Omega}_{rs}}{\sqrt{w\paren{\boldsymbol{\Omega}_{rr}\boldsymbol{\Omega}_{ss} + \boldsymbol{\Omega}^2_{rs}}}}:\;\paren{r,s}\in\cc{K}_p}.
\end{equation}
$\boldsymbol{U}_r$ is a vectorized version (concatenating the rows) of the upper triangle matrix constructed by indices in $\cc{K}_p$. Notice that the entries of $\set{\boldsymbol{U}_r}^w_{r=1}$ have zero-mean and unit-variance, and have finite third moment. Furthermore, let $\boldsymbol{Z}_i,\;j=1,\ldots, w$ be a set of i.i.d. centered Gaussian vectors with the same covariance matrix as $\boldsymbol{U}_1$. The formulation of $\boldsymbol{U}_i$ and $\boldsymbol{Z}_i$ vectors obviously implies that
\begin{equation*}
\max_{1\leq j\leq q} \sum_{i=1}^{w} \boldsymbol{U}_{ij} = \max_{\paren{r,s}\in\cc{K}_p} \boldsymbol{E}_{rs}\quad\mbox{and}\quad\max_{1\leq j\leq q} \sum_{i=1}^{w} \boldsymbol{Z}_{ij} = \max_{\paren{r,s}\in\cc{K}_p} \boldsymbol{G}_{rs},
\end{equation*}
which set the stage for proving Claim \ref{Claim2Thm0} using Theorem \ref{ThmCher}. Next, we obtain sharp upper bounds on $B_1$ and $B_2$ in terms of $p$ and $w$. Observe that
\begin{equation*}
B_1 \leq w^{-3/2}\sum_{i=1}^{w} \bb{E}\paren{\max_{\paren{r,s}\in\cc{K}_p}\abs{\frac{ \paren{\boldsymbol{Y}_{t+i}}_r\paren{\boldsymbol{Y}_{t+i}}_s-\boldsymbol{\Omega}_{rs}}{\sqrt{\paren{\boldsymbol{\Omega}_{rr}\boldsymbol{\Omega}_{ss} + \boldsymbol{\Omega}^2_{rs}}}}}^3} = \frac{1}{\sqrt{w}} \bb{E}\paren{\max_{\paren{r,s}\in\cc{K}_p}\abs{\frac{ \paren{\boldsymbol{Y}_{t+1}}_r\paren{\boldsymbol{Y}_{t+1}}_s-\boldsymbol{\Omega}_{rs}}{\sqrt{\paren{\boldsymbol{\Omega}_{rr}\boldsymbol{\Omega}_{ss} + \boldsymbol{\Omega}^2_{rs}}}}}^3}.
\end{equation*}
Since all diagonal entries of $\boldsymbol{\Omega}$ are one, then the triangle inequality leads to
\begin{equation}\label{B1UppBnd}
B_1 \lesssim \frac{1}{\sqrt{w}}\paren{1\vee \bb{E}\paren{\max_{\paren{r,s}\in\cc{K}_p} \abs{\paren{\boldsymbol{Y}_{t+1}}_r\paren{\boldsymbol{Y}_{t+1}}_s}^3}} \leq \frac{\bb{E}\LpNorm{\boldsymbol{Y}_{t+1}}{\infty}^6 }{\sqrt{w}}
\end{equation}
Thus, according to Lemma \ref{UppBndmomentsSubGauss} $B_1\lesssim \frac{\log^3 p}{\sqrt{w}}$.  Now, applying Theorem \ref{ThmCher} gives us the following inequality for any $\delta > 0$.
\begin{eqnarray}\label{eq1Thm1}
\bb{P}\paren{\abs{\max_{\paren{r,s}\in\cc{K}_p} \boldsymbol{E}_{rs}-\max_{\paren{r,s}\in\cc{K}_p} \boldsymbol{G}_{rs}}> \delta} &=& \bb{P}\paren{\abs{\max_{1\leq j\leq q} \sum_{i=1}^{w} \boldsymbol{U}_{ij}-\max_{1\leq j\leq q} \sum_{i=1}^{w} \boldsymbol{Z}_{ij}}> \delta}\nonumber\\
&\lesssim& \frac{\log w}{w}\vee \frac{B_1\log p}{\delta^2} \vee \frac{B_2\log^2 p}{\delta} \asymp \frac{\log w}{w}\vee \frac{\log^4 p}{\delta^2\sqrt{w}} \vee \frac{B_2\log^2 p}{\delta}
\end{eqnarray}
Since $\frac{\log^8 p}{w}\rightarrow 0$, then Eq. \eqref{eq1Thm1} implies that  $\max_{\paren{r,s}\in\cc{K}_p} \boldsymbol{E}_{rs}-\max_{\paren{r,s}\in\cc{K}_p} \boldsymbol{G}_{rs}\cp{\bb{P}} 0$, if $B_2\log^2p\rightarrow 0$. We conclude the proof by  controlling $B_2$ from above. We use Lemma $1$ in \cite{chernozhukov2015comparison}, which we state next for completeness.
	
\begin{lem}\label{LemCher}
$\set{\boldsymbol{U}_i}^w_{i=1}$ be a set of i.i.d. centered random vectors in $\bb{R}^p$ with covariance matrix $\boldsymbol{\Lambda}$ and the fourth moment $\bb{E}\boldsymbol{U}^4_{1j} = M_j$. There exists a bounded constant $C$ such that
\begin{equation*}
B_2 = \bb{E}\paren{\max_{1\leq j, k\leq q} \abs{\sum_{i=1}^{w} \boldsymbol{U}_{ij}\boldsymbol{U}_{ik}-w\boldsymbol{\Lambda}_{jk}}} \leq C\sqrt{\log p}\brac{\sqrt{w\max_{1\leq j\leq p}M_j}+\sqrt{\log p \;\bb{E}\paren{\max_{1\leq i\leq w}\max_{1\leq j\leq p}\boldsymbol{U}^4_{ij}}}}
\end{equation*}
\end{lem}
The formulation of $\boldsymbol{U}_i$ vectors in Eq. \eqref{UVec} implies that $\bb{E}\boldsymbol{U}^4_{ij} \asymp w^{-2}$ (notice the existence of $\sqrt{w}$ in the denominator of $\boldsymbol{U}_i$). So, the upper bound on $B_2$ in Lemma \ref{LemCher} implies that
\begin{equation*}
B_2 \lesssim \sqrt{\frac{w}{w^2}\log p} + \log p\sqrt{\bb{E}\paren{\max_{1\leq i\leq w}\max_{1\leq j\leq p}\boldsymbol{U}^4_{ij}}} \asymp \sqrt{\frac{\log p}{w}} \vee \log p\sqrt{\bb{E}\paren{\max_{1\leq i\leq w}\max_{1\leq j\leq p}\boldsymbol{U}^4_{ij}}}.
\end{equation*}
An analogous technique as in Eq. \eqref{B1UppBnd}, as well as applying Lemma \ref{UppBndmomentsSubGauss}, implies that
\begin{equation*}
\bb{E}\paren{\max_{1\leq i\leq w}\max_{1\leq j\leq p}\boldsymbol{U}^4_{ij}} \lesssim \bb{E}\paren{\frac{\LpNorm{\boldsymbol{Y}_{t+1}}{\infty}^2}{\sqrt{w}}}^4 = \frac{\bb{E}\LpNorm{\boldsymbol{Y}_{t+1}}{\infty}^8}{w^2} \asymp\frac{\log^4 p}{w^2}.
\end{equation*}
Thus, $B_2$ can be controlled from above by
\begin{equation*}
B_2 \lesssim \sqrt{\frac{\log p}{w}} \vee \log p \sqrt{\frac{\log^4 p}{w^2}} = \sqrt{\frac{\log p}{w}} \vee \frac{\log^3 p}{w}.
\end{equation*}
Hence, 
\begin{equation*}
B_2\log^2 p \lesssim \sqrt{\frac{\log p}{w}} \log^2 p \vee \frac{\log^5 p}{w} = \sqrt{\frac{\log^5 p}{w}} \vee \frac{\log^5 p}{w} \rightarrow 0,
\end{equation*}
which concludes the proof of Claim \ref{Claim2Thm0}.
\end{proof}

\begin{proof}[Proof of Theorem \ref{Thm1}]
We remind the definition of $\zeta_{\pi_0, p, w}$ in Eq. \eqref{zeta}.
\begin{equation*}
\bb{P}\paren{ \abs{\vartheta_w} \geq \zeta_{\pi_0, p, w}} = \frac{-2}{p\paren{p+1}}\log\paren{1-\pi_0}.
\end{equation*}
Here $\vartheta_w$ denotes the standardized inner product of two independent standard Gaussian random vectors in $\bb{R}^w$ (see Definition \ref{InnerProdGauss}). Notice that
\begin{equation*}
\bb{P}_{\FA}\paren{T_t} = \bb{P}\paren{\max_{\paren{r,s}\in\cc{K}_p} \abs{\boldsymbol{E}_{rs}}\geq \zeta_{\pi_0, p, w}}.
\end{equation*}
The proof is based upon Theorem $1$ in \cite{galambos1988variants}, which we state next for completeness. We also refer the reader to \cite{galambos1988variants, galambos1972distribution} for further technicalities.
	
\begin{thm}[\emph{Galambos Theorem} \cite{galambos1988variants}]\label{GalambosThm}
Let $\set{V_i}^q_{i=1}$ be a set of (possibly dependent) random variables. Consider the graph $\bb{G} = \paren{\bb{V},\cc{E}_q}$ with $\bb{V} = \set{1,\ldots, q}$, no self-loop, and $\paren{r,s}\notin\cc{E}_q$ if an only if $V_r \independent V_t$. For any fixed $\pi_0\in\paren{0,1}$, select $\xi_{q,\pi_0}$ such that
\begin{equation}\label{cond1Galambos}
\lim\limits_{q\rightarrow\infty}\sum_{j=1}^{q} \bb{P}\paren{V_j \geq \xi_{q,\pi_0}} = \log\paren{\frac{1}{1-\pi_0}}.
\end{equation}
Assume that the following conditions are satisfied as $q\rightarrow\infty$.
\begin{enumerate}
\item $\abs{\cc{E}_q} = o\paren{q^2}$.
\item There is a bounded constant $K$ such that $\limsup_{q\rightarrow\infty} \Bigbrac{\max_{1\leq j\leq q}\;q\bb{P}\paren{V_j \geq \xi_{q,\pi_0}}} \leq K.$ 
\item $\limsup_{q\rightarrow\infty} \Bigbrac{\sum_{r, s\in\cc{E}_q} \bb{P}\paren{V_r \wedge V_s \geq \xi_{q,\pi_0}}} = 0$.
\end{enumerate}
Then, 
\begin{equation*}
\lim\limits_{q\rightarrow\infty}\bb{P}\paren{\max_{1\leq j\leq q} V_j \geq \xi_{q,\pi_0}} = \pi_0
\end{equation*}
\end{thm} 
	
The rest of the proof is devoted to verifying the conditions in Theorem \ref{GalambosThm}. For any $u\in\set{1,\ldots, p}$, $\cc{H}_u$ denotes nodes connected to $u$ in $\cc{G}_t$ (pre-change GGM), i.e., $\cc{H}_u = \set{v\ne u: \boldsymbol{\Omega}_{uv} \ne 0}$. We know from Assumption \ref{AssuCLT} that $\abs{\cc{H}_u} \leq d_{\max}$. Let $q=\abs{\cc{K}_p}$. Construct the graph $\bb{G} = \paren{\cc{K}_p, \cc{E}_q}$ by connecting an edge between each pair of dependent random variables in $\set{\boldsymbol{E}_{rs}: \paren{r,s}\in\cc{K}_p}$. We first claim that the maximum degree of $\bb{G}$ is no more than $2pd_{\max}$. For proving this claim, we use two facts. First, $\abs{\cc{H}_r \cup \cc{H}_s} \leq 2d_{\max}$ for any $\paren{r,s}\in\cc{K}_p$. Further,
\begin{equation*}
\boldsymbol{E}_{rs} \independent \boldsymbol{E}_{uv},\;\forall\paren{u,v}\in\cc{K}_p\;\;\suchthat\;\; u, v \notin \cc{H}_r \cup \cc{H}_s.
\end{equation*}
Namely, $\paren{u,v}$ is connected to $\paren{r,s}$ only if at least one of $u$ or $v$ belong to $\cc{H}_r \cup \cc{H}_s$. Hence, the first condition in Theorem \ref{GalambosThm} obviously holds, since
\begin{equation*}
\abs{\cc{E}_q} \leq 2pd_{\max}\abs{\cc{K}_p} \asymp p^3d_{\max} = o\paren{q^2}.
\end{equation*}

Before verifying the other conditions in Theorem \ref{GalambosThm}, we require to study the asymptotic behaviour of critical value $\zeta_{\pi_0, p, w}$. Recall that we defined $\zeta_{\pi_0,p,w}$ by 
\begin{equation*}
\bb{P}\paren{\vartheta_w\geq \zeta_{\pi_0,p,w}} = \frac{-\log\paren{1-\pi_0}}{p\paren{p+1}}.
\end{equation*}
We also discussed in Remark \ref{ChernoffBound} that $\zeta^2_{\pi_0,p,w} = \cc{O}\paren{\log p}$. Our objective is to prove the following property for $\zeta_{\pi_0, p, w}$, which is more informative than the discussion in Remark \ref{ChernoffBound}.
\begin{equation}\label{Eq7Thm1}
\mbox{If } \log^3 p = o\paren{w}, \mbox{ then} \lim\limits_{w, p\rightarrow\infty} \frac{\zeta^2_{\pi_0,p,w}}{4\log p} = 1,
\end{equation}
As we assume that $\log^3 p = o\paren{w}$, then obviously $w^{-1}\zeta^6_{\pi_0, p, w}\rightarrow 0$. Thus, Corollary \ref{GaussInnerProdUppBnd} implies that 
\begin{equation*}
\bb{P}\paren{\vartheta_w\geq \zeta_{\pi_0,p,w}} = \frac{-\log\paren{1-\pi_0}}{p\paren{p+1}} \sim \frac{\exp\paren{-\frac{\zeta^2_{\pi_0,p,w}}{2}}}{\zeta_{\pi_0,p,w}\sqrt{2\pi}}.
\end{equation*}
So, $\zeta^2_{\pi_0,p,w} \sim 4\log p-2\log\log p+C_{\pi_0}$ for a bounded scalar $C_{\pi_0}$, which validates asymptotic identity \eqref{Eq7Thm1}. We are now ready for verifying the two other conditions in Theorem \ref{GalambosThm}.
	
\begin{clawithinpf}\label{Claim1Thm1}
There is scalar $K_{\pi_0} < \infty$ such that $\max_{\paren{r,s}\in\cc{K}_p} \bb{P}\paren{\abs{\boldsymbol{E}_{rs}}\geq \zeta_{\pi_0, p, w}} \leq K_{\pi_0}p^{-2}$.
\end{clawithinpf}
Note that Claim \ref{Claim1Thm1} is equivalent to the second condition in Theorem \ref{GalambosThm}. For proving Claim \ref{Claim1Thm1}, choose an arbitrary $\paren{r,s}\in\cc{K}_p$. We give a simpler formulation for $\boldsymbol{E}_{rs}$. Let 
\begin{equation*}
\boldsymbol{U} = \brac{\paren{\boldsymbol{Y}_1}_r,\ldots, \paren{\boldsymbol{Y}_w}_r}^\top,\quad\mbox{and}\quad \boldsymbol{U'} = \brac{\paren{\boldsymbol{Y}_1}_s,\ldots, \paren{\boldsymbol{Y}_w}_s}^\top.
\end{equation*}
$\boldsymbol{U}$ and $\boldsymbol{U'}$ are standard Gaussian random vectors with $\cov\paren{\boldsymbol{U}_i, \boldsymbol{U'}_j} = \boldsymbol{\Omega}_{rs}\bbM{1}_{\brac{i=j}}$. Notice that
\begin{equation*}
\boldsymbol{E}_{rs} = \frac{\sum_{l=1}^{w} \paren{\boldsymbol{U}_l\boldsymbol{U'}_l-\boldsymbol{\Omega}_{rs}} }{\sqrt{w\paren{1+\boldsymbol{\Omega}^2_{rs}}}}.
\end{equation*}
In words, $\boldsymbol{E}_{rs}$ is a standardized inner product of two correlated Gaussian vectors. We have studied the non-asymptotic properties of such an object in Appendix \ref{AppendixB}. Theorem \ref{MainThmApp} implies that if $\zeta^6_{\pi_0, p, w}$ grows to infinity at a slower rate than $w$, which we know that it holds, then 
\begin{equation}\label{Eq0Thm1}
\lim\limits_{w, p\rightarrow\infty}\frac{\bb{P}\paren{\abs{\boldsymbol{E}_{rs}}\geq \zeta_{\pi_0, p, w}}}{2\bb{P}\paren{Z\geq \zeta_{\pi_0, p, w}}} \rightarrow 1,\quad\mbox{and}\quad \lim\limits_{w, p\rightarrow\infty}\frac{\bb{P}\paren{\abs{\vartheta_w}\geq \zeta_{\pi_0, p, w}}}{2\bb{P}\paren{Z\geq \zeta_{\pi_0, p, w}}} \rightarrow 1.
\end{equation}
Combining the two limiting identities in Eq. \eqref{Eq0Thm1} implies that 
\begin{eqnarray*}
\max_{\paren{r,s}\in\cc{K}_p} \bb{P}\paren{\abs{\boldsymbol{E}_{rs}}\geq \zeta_{\pi_0, p, w}} &\leq& 2\bb{P}\paren{\abs{\vartheta_w}\geq \zeta_{\pi_0, p, w}} = \frac{4}{p\paren{p+1}}\log\paren{\frac{1}{1-\pi_0}} \\
&\leq& \frac{4}{p^2}\log\paren{\frac{1}{1-\pi_0}}.
\end{eqnarray*}
as $w, p\rightarrow\infty$. Thus, Claim \ref{Claim1Thm1} holds with $K_{\pi_0} = -4\log\paren{1-\pi_0}$. Next, we verify condition \eqref{cond1Galambos} in Galambos Theorem. Namely, we want to show that
\begin{equation}\label{Eq1Thm1}
\lim\limits_{p\rightarrow\infty}\sum_{{\paren{r,s}\in\cc{K}_p}}\bb{P}\paren{\abs{\boldsymbol{E}_{rs}}\geq \zeta_{\pi_0, p, w}} = \log\paren{\frac{1}{1-\pi_0}}.
\end{equation}
Decompose $\cc{K}_p$ into two non-overlapping parts $\cc{K}'_p$ and $\cc{K}''_p$ defined by
\begin{equation*}
\cc{K}'_p = \set{\paren{u, v}\in\cc{K}_p:\; \boldsymbol{\Omega}_{uv} \ne 0}, \quad \cc{K}''_p = \set{\paren{u, v}\in\cc{K}_p:\; \boldsymbol{\Omega}_{uv} = 0}
\end{equation*}
Based on the bounded degree Assumption \ref{AssuCLT}, $\abs{\cc{K}'_p} \leq pd_{\max}$. It is also obvious that $\boldsymbol{E}_{rs} \eqd \vartheta_w$ for any $\paren{r,s}\in\cc{K}''_p$. Therefore,
\begin{equation}
\bb{P}\paren{\abs{\boldsymbol{E}_{rs}}\geq \zeta_{\pi_0, p, w}} = \bb{P}\paren{ \abs{\vartheta_w} \geq \zeta_{\pi_0, p, w}} = \frac{-2}{p\paren{p+1}}\log\paren{1-\pi_0},\quad \forall\; \paren{r,s}\in\cc{K}''_p.
\end{equation}
Thus,
\begin{equation}\label{Eq2Thm1}
\lim\limits_{p\rightarrow\infty}\sum_{{\paren{r,s}\in\cc{K}''_p}}\bb{P}\paren{\abs{\boldsymbol{E}_{rs}}\geq \zeta_{\pi_0, p, w}} = \lim\limits_{p\rightarrow\infty} \abs{\cc{K}''_p}\bb{P}\paren{ \abs{\vartheta_w} \geq \zeta_{\pi_0, p, w}}
= -\log\paren{1-\pi_0}. 
\end{equation}
The last identity is obtained from the fact that $\abs{\cc{K}''_p} \asymp \abs{\cc{K}_p} = p\paren{p+1}/2$. Further, Claim \ref{Claim1Thm1} ensures uniform boundedness of $p^2\bb{P}\paren{\abs{\boldsymbol{E}_{rs}}\geq \zeta_{\pi_0, p, w}}$ over all  $\paren{r,s}\in\cc{K}'_p$. Thus
\begin{eqnarray}\label{Eq3Thm1}
\limsup\limits_{p\rightarrow\infty}\sum_{{\paren{r,s}\in\cc{K}'_p}}\bb{P}\paren{\abs{\boldsymbol{E}_{rs}}\geq \zeta_{\pi_0, p, w}} &\leq& \lim\limits_{p\rightarrow\infty} \abs{\cc{K}'_p}\max_{\paren{r,s}\in\cc{K}_p} \bb{P}\paren{\abs{\boldsymbol{E}_{rs}}\geq \zeta_{\pi_0, p, w}} \nonumber\\
&\lesssim& \lim\limits_{p\rightarrow\infty} \frac{\abs{\cc{K}'_p}}{p^2}\asymp\lim\limits_{p\rightarrow\infty} \frac{d_{\max}}{p}=0.
\end{eqnarray}
Combining Eq. \eqref{Eq2Thm1} and \eqref{Eq3Thm1} completes the proof of the identity \eqref{Eq1Thm1}. We end the proof by verifying the third condition in Theorem \ref{GalambosThm}. Particularly, our objective is to show that 
\begin{equation}\label{Eq4Thm1}
\lim\limits_{p,w\rightarrow\infty}\brac{\sum_{\paren{r,s}\leftrightarrow\paren{u,v}} \bb{P}\Bigparen{\abs{\boldsymbol{E}_{rs}} \wedge \abs{\boldsymbol{E}_{uv}} \geq \zeta_{\pi_0,p,w} } } = 0. 
\end{equation}
In Eq. \eqref{Eq4Thm1}, $\paren{r,s}\leftrightarrow\paren{u,v}$ refers to the existence of an edge between $\paren{r,s}$ and $\paren{u,v}$ in $\bb{G} = \paren{\cc{K}_p, \cc{E}_q}$. Namely, 
the covariance between $\boldsymbol{E}_{rs}$ and $\boldsymbol{E}_{uv}$ is non-zero and its absolute value is strictly less than one, if $\paren{r,s}\leftrightarrow\paren{u,v}$. In the asymptotic setting of $p, w\rightarrow\infty$, we get
\begin{eqnarray}\label{Eq5Thm1}
\bb{P}\Bigparen{\abs{\boldsymbol{E}_{rs}} \wedge \abs{\boldsymbol{E}_{uv}} \geq \zeta_{\pi_0,p,w} } &=& \bb{P}\Bigparen{\abs{\boldsymbol{E}_{rs}} \geq \zeta_{\pi_0,p,w} }\bb{P}\Bigparen{\abs{\boldsymbol{E}_{rs}} \geq \zeta_{\pi_0,p,w} \Bigl\lvert \abs{\boldsymbol{E}_{uv}} \geq \zeta_{\pi_0,p,w}}\nonumber\\
&=&o\set{\bb{P}\Bigparen{\abs{\boldsymbol{E}_{rs}} \geq \zeta_{\pi_0,p,w} }} = o\paren{\frac{1}{p^2}}
\end{eqnarray}
Define the function $d\paren{\cdot, \cdot}:\cc{K}_p\times \cc{K}_p \mapsto \set{1,2,3,4}$ by
\begin{align*}
d\Bigbrac{\paren{r,s}, \paren{u,v}} \coloneqq\; &\mbox{The size of the maximal set } \cc{H}\subseteq\set{\boldsymbol{X}_r, \boldsymbol{X}_s, \boldsymbol{X}_u, \boldsymbol{X}_v},\\ &\mbox{ whose elements are all independent.}
\end{align*}
If $\paren{r,s}\leftrightarrow\paren{u,v}$, then $d\Bigbrac{\paren{r,s}, \paren{u,v}} \ne 4$. Thus, $\cc{E}_q$ is partitioned in the following way. 
\begin{eqnarray*}
\cc{E}_q &=& \cc{E}'_q \cup \cc{E}''_q \\
&=& \set{\paren{r,s}\leftrightarrow\paren{u,v}:\; d\Bigbrac{\paren{r,s}, \paren{u,v}}\leq 2} \bigcup \set{\paren{r,s}\leftrightarrow\paren{u,v}:\; d\Bigbrac{\paren{r,s}, \paren{u,v}}= 3}.
\end{eqnarray*}
$d\Bigbrac{\paren{r,s}, \paren{u,v}} \leq 2$ means that at least two edges connect the nodes in $\set{\boldsymbol{X}_r, \boldsymbol{X}_s, \boldsymbol{X}_u, \boldsymbol{X}_v}$. Hence, according to Assumption \ref{AssuCLT}, $\abs{\cc{E}'_q} \lesssim p^2d^2_{\max}$. Using Eq. \eqref{Eq5Thm1} leads to
\begin{equation*}
\lim\limits_{p,w\rightarrow\infty}\brac{\sum_{\cc{E}'_q} \bb{P}\Bigparen{\abs{\boldsymbol{E}_{rs}} \wedge \abs{\boldsymbol{E}_{uv}} \geq \zeta_{\pi_0,p,w} } } \lesssim o\paren{\frac{p^2d^2_{\max}}{p^2}} = 0. 
\end{equation*}
Therefore, we only need to focus on $\cc{E}''_q$. Observe that $\abs{\cc{E}''_q} \lesssim p^3d_{\max}$. Thus, for proving the condition in Eq. \eqref{Eq4Thm1}, it suffices to show that
\begin{clawithinpf}\label{Claim2Thm1}
There exists some $\kappa\in\paren{0,1}$ such that
\begin{equation*}
\bb{P}\Bigparen{\abs{\boldsymbol{E}_{rs}} \wedge \abs{\boldsymbol{E}_{uv}} \geq \zeta_{\pi_0,p,w} } = o\paren{p^{-\paren{3+\kappa}}},\quad\forall\;\Bigparen{\paren{r,s}, \paren{u,v}}\in \cc{E}''_q.
\end{equation*}
\end{clawithinpf} 
Pick any $\Bigparen{\paren{r,s}, \paren{u,v}}\in \cc{E}''_q$. Recall that $\boldsymbol{E}_{rs}$ and $\boldsymbol{E}_{uv}$ are mean-zero and unit-variance random variables, defined by
\begin{equation*}
\boldsymbol{E}_{rs} = \frac{\sum_{l=1}^{w} \Bigbrac{\paren{\boldsymbol{Y}_l}_r\paren{\boldsymbol{Y}_l}_s-\boldsymbol{\Omega}_{rs}} }{\sqrt{w\paren{1+\boldsymbol{\Omega}^2_{rs}}}},\quad\mbox{and}\quad \boldsymbol{E}_{uv} = \frac{\sum_{l=1}^{w} \Bigbrac{\paren{\boldsymbol{Y}_l}_u\paren{\boldsymbol{Y}_l}_v-\boldsymbol{\Omega}_{uv}} }{\sqrt{w\paren{1+\boldsymbol{\Omega}^2_{uv}}}}.
\end{equation*}
We first show that $\boldsymbol{E}_{rs}$ and $\boldsymbol{E}_{uv}$ are uncorrelated. Since $d\Bigbrac{\paren{r,s}, \paren{u,v}} = 3$, without loss of generality, we suppose that $\set{\paren{\boldsymbol{Y}_l}_r, \paren{\boldsymbol{Y}_l}_s, \paren{\boldsymbol{Y}_l}_u}^w_{l=1}$ are independent, i.e., $\boldsymbol{\Omega}_{rs}=\boldsymbol{\Omega}_{ru}=\boldsymbol{\Omega}_{su}=0$. Applying Isserlis' Theorem implies that
\begin{equation*}
\cov\paren{\boldsymbol{E}_{rs}, \boldsymbol{E}_{uv}} = \frac{\sum_{l=1}^{w}\cov\paren{\paren{\boldsymbol{Y}_l}_r\paren{\boldsymbol{Y}_l}_s, \paren{\boldsymbol{Y}_l}_u\paren{\boldsymbol{Y}_l}_v}}{w\sqrt{\paren{1+\boldsymbol{\Omega}^2_{rs}}\paren{1+\boldsymbol{\Omega}^2_{uv}}}} = \frac{\boldsymbol{\Omega}_{ru}\boldsymbol{\Omega}_{sv}+\boldsymbol{\Omega}_{rv}\boldsymbol{\Omega}_{su}}{\sqrt{\paren{1+\boldsymbol{\Omega}^2_{rs}}\paren{1+\boldsymbol{\Omega}^2_{uv}}}},
\end{equation*}
which verifies the desired result. We are now ready to prove Claim \ref{Claim2Thm1}. Observe that,
\begin{eqnarray}\label{Eq6Thm1}
\bb{P}\Bigparen{\abs{\boldsymbol{E}_{rs}} \wedge \abs{\boldsymbol{E}_{uv}} \geq \zeta_{\pi_0,p,w} } &=& \bb{P}\Bigparen{\boldsymbol{E}_{rs} \wedge \boldsymbol{E}_{uv} \geq \zeta_{\pi_0,p,w}}+\bb{P}\Bigparen{\boldsymbol{E}_{rs} \wedge -\boldsymbol{E}_{uv} \geq \zeta_{\pi_0,p,w}}\nonumber\\
&=&\bb{P}\Bigparen{-\boldsymbol{E}_{rs} \wedge \boldsymbol{E}_{uv} \geq \zeta_{\pi_0,p,w}}+\bb{P}\Bigparen{-\boldsymbol{E}_{rs} \wedge -\boldsymbol{E}_{uv} \geq \zeta_{\pi_0,p,w}}.
\end{eqnarray}
For avoiding repetition, we only focus on the first term in the right hand side of Eq. \eqref{Eq6Thm1}. The other terms can be handled in an analogous way. The asymptotic identity \eqref{Eq7Thm1} ensures the existence of some $\kappa\in\paren{0,1}$ so that $\zeta^2_{\pi_0,p,w} \geq \paren{3+\kappa}\log p$ (when both $p,w\rightarrow\infty$). Due to the absence of correlation between $\boldsymbol{E}_{rs}$ and $\boldsymbol{E}_{uv}$, $\paren{\boldsymbol{E}_{rs}+ \boldsymbol{E}_{uv}}/\sqrt{2}$ is a zero-mean and unit-variance quadratic form of Gaussian random variables. Hence, an application of Lemma \ref{ChernoffBound} yields,
\begin{equation*}
\bb{P}\Bigparen{\boldsymbol{E}_{rs} \wedge \boldsymbol{E}_{uv} \geq \zeta_{\pi_0,p,w}} \leq \bb{P}\Bigparen{\frac{\boldsymbol{E}_{rs}+\boldsymbol{E}_{uv}}{\sqrt{2}} \geq \sqrt{2}\zeta_{\pi_0,p,w}} \lesssim \exp\paren{-\zeta^2_{\pi_0,p,w}}\leq p^{-\paren{3+\kappa}},
\end{equation*}
which ends the proof of Claim \ref{Claim2Thm1}. 
\end{proof}

\begin{proof}[Theorem \ref{Thm2}]
$T_t$ correctly identifies a change-point at $\paren{t+w}$, if $\LpNorm{\boldsymbol{E}_{t,w}}{\infty} > \zeta_{\pi_0, p, w}$. The goal is to introduce a sufficient condition on $\TimeDepMatrix{\Delta}{t}$ for this criterion to hold with high probability. Notice that,
\begin{equation*}
\bb{E}\paren{\boldsymbol{E}_{t,w} \mid \bb{H}_{1,t}} = \sqrt{w}\TimeDepMatrix{\Delta}{t}.
\end{equation*}
The triangle inequality implies that under the alternative hypothesis $\bb{H}_{1,t}$, 
\begin{equation*}
\LpNorm{\boldsymbol{E}_{t,w}}{\infty} - \zeta_{\pi_0, p, w} \geq \sqrt{w}\LpNorm{\TimeDepMatrix{\Delta}{t}}{\infty} - \zeta_{\pi_0, p, w} - \LpNorm{\boldsymbol{E}_{t,w}-\bb{E}\paren{\boldsymbol{E}_{t,w} \mid \bb{H}_{1,t}}}{\infty}.
\end{equation*}
\begin{clawithinpf}\label{Claim1Thm2} 
For any $\xi > 0$, there exists a bounded scalar $C_{\xi}$ such that
\begin{equation*}
\bb{P}\paren{\LpNorm{\boldsymbol{E}_{t,w}-\bb{E}\paren{\boldsymbol{E}_{t,w} \mid \bb{H}_{1,t}}}{\infty} \geq C_{\xi}\frac{d^2_{\max}}{\alpha_{\min}}\sqrt{\log p}\; \Bigl\lvert \;\bb{H}_{1,t}} \leq p^{-\xi}
\end{equation*}
\end{clawithinpf}
Given Claim \ref{Claim1Thm2}, the following inequality holds with probability at least $1-p^{-\xi}$.
\begin{equation}
\LpNorm{\boldsymbol{E}_{t,w}}{\infty} - \zeta_{\pi_0, p, w} \geq \sqrt{w}\LpNorm{\TimeDepMatrix{\Delta}{t}}{\infty} - \zeta_{\pi_0, p, w} - C_{\xi}\frac{d^2_{\max}}{\alpha_{\min}}\sqrt{\log p}.
\end{equation}
Thus, $T_t = 1$ with probability at least $1-p^{-\xi}$, if
\begin{equation}\label{Eq1Thm2}
\sqrt{w}\LpNorm{\TimeDepMatrix{\Delta}{t}}{\infty} - \zeta_{\pi_0, p, w} - C_{\xi}\frac{d^2_{\max}}{\alpha_{\min}}\sqrt{\log p} > 0\;\; \Longleftrightarrow\;\;  \LpNorm{\TimeDepMatrix{\Delta}{t}}{\infty} > \frac{\zeta_{\pi_0, p, w}}{\sqrt{w}} + C_{\xi}\frac{d^2_{\max}}{\alpha_{\min}}\sqrt{\frac{\log p}{w}}.
\end{equation}
The condition on $\LpNorm{\TimeDepMatrix{\Delta}{t}}{\infty}$ in Eq. \eqref{Eq1Thm2} is same as Eq. \eqref{SuffDetectCond} in the statement of Theorem \ref{Thm2}. So, we only need to prove Claim \ref{Claim1Thm2}. Before proceeding further, recall that  $\boldsymbol{X}_{t+r},\;r=1,\ldots,w$ are i.i.d. zero-mean Gaussian vectors with covariance matrix $\TimeDepMatrix{\Sigma}{t+1}$ (as opposed to $\boldsymbol{X}_t$ whose precision matrix is given by $\TimeDepMatrix{\Omega}{t} = \paren{\TimeDepMatrix{\Sigma}{t}}^{-1}$). Observe that,
\begin{eqnarray*}
\boldsymbol{E}_{t,w}-\bb{E}\paren{\boldsymbol{E}_{t,w}} &=& \sum_{r=1}^{w}\frac{\paren{ \TimeDepMatrix{\Omega}{t}\boldsymbol{X}_{t+r}\boldsymbol{X}_{t+r}^\top\TimeDepMatrix{\Omega}{t}-\TimeDepMatrix{\Omega}{t}\TimeDepMatrix{\Sigma}{t+1}\TimeDepMatrix{\Omega}{t}}}{\sqrt{w}} \circ \brac{ \paren{\TimeDepMatrix{\Omega}{t}_{uu}\TimeDepMatrix{\Omega}{t}_{vv} + \paren{\TimeDepMatrix{\Omega}{t}_{uv}}^2}^{-1/2} }^p_{u,v=1}\\
&=& \sum_{r=1}^{w}\frac{ \paren{\TimeDepMatrix{\Omega}{t}\boldsymbol{X}_{t+r}\boldsymbol{X}_{t+r}^\top\TimeDepMatrix{\Omega}{t}-\TimeDepMatrix{\Omega}{t}\TimeDepMatrix{\Sigma}{t+1}\TimeDepMatrix{\Omega}{t}}}{\sqrt{w}} \circ \brac{ \paren{1 + \paren{\TimeDepMatrix{\Omega}{t}_{uv}}^2}^{-1/2} }^p_{u,v=1}.
\end{eqnarray*}
Thus, 
\begin{eqnarray*}
\LpNorm{\boldsymbol{E}_{t,w}-\bb{E}\paren{\boldsymbol{E}_{t,w}}}{\infty}&\leq& \LpNorm{\TimeDepMatrix{\Omega}{t}\sum_{r=1}^{w}\paren{\frac{\boldsymbol{X}_{t+r}\boldsymbol{X}_{t+r}^\top-\TimeDepMatrix{\Sigma}{t+1}}{\sqrt{w}}}\TimeDepMatrix{\Omega}{t}}{\infty}\LpNorm{\brac{\paren{1 + \paren{\TimeDepMatrix{\Omega}{t}_{uv}}^2}^{-1/2} }^p_{u,v=1}}{\infty}\\
&\leq& \OpNorm{\TimeDepMatrix{\Omega}{t}}{\infty}{\infty}^2\LpNorm{\sum_{r=1}^{w}\paren{\frac{\boldsymbol{X}_{t+r}\boldsymbol{X}_{t+r}^\top-\TimeDepMatrix{\Sigma}{t+1}}{\sqrt{w}}}}{\infty}\\
&\leq& d^2_{\max}\LpNorm{\sum_{r=1}^{w}\paren{\frac{\boldsymbol{X}_{t+r}\boldsymbol{X}_{t+r}^\top-\TimeDepMatrix{\Sigma}{t+1}}{\sqrt{w}}}}{\infty}.
\end{eqnarray*}
The last inequality is implied from Assumption \ref{AssuCLT} (row-sparsity of $\TimeDepMatrix{\Omega}{t}$) and the fact that all diagonal entries of $\TimeDepMatrix{\Omega}{t}$ are equal to one. Thus, Claim \ref{Claim1Thm2} holds, if we can prove that 
\begin{equation}\label{Eq2Thm2}
\bb{P}\paren{\LpNorm{\frac{1}{w}\sum_{r=1}^{w}\boldsymbol{X}_{t+r}\boldsymbol{X}_{t+r}^\top-\TimeDepMatrix{\Sigma}{t+1}}{\infty} \geq \frac{C_{\xi}}{\alpha_{\min}}\sqrt{\frac{\log p}{w}}}. \leq p^{-\xi},
\end{equation}
Observe that
\begin{equation}\label{Eq3Thm2}
\max_{1\leq i \leq p} \abs{\TimeDepMatrix{\Sigma}{t+1}_{ii}} \leq \OpNorm{\TimeDepMatrix{\Sigma}{t+1}}{2}{2} = \frac{1}{\lambda_{\min}\paren{\TimeDepMatrix{\Omega}{t+1}}} \leq \frac{1}{\alpha_{\min}}.
\end{equation}
Thus, the inequality \eqref{Eq2Thm2} is directly followed from Lemma \ref{SampleCovarSupNormUppBnd}.
\end{proof}

Before beginning the proof of Theorem \ref{Thm3}, recall that 
\begin{align*}
&\boldsymbol{E}_{t, w} = \sum_{r=1}^{w}\frac{ \paren{\TimeDepMatrix{\Omega}{t}\boldsymbol{X}_{t+r}\boldsymbol{X}_{t+r}^\top\TimeDepMatrix{\Omega}{t}-\TimeDepMatrix{\Omega}{t}}}{\sqrt{w}} \circ \brac{ \paren{\TimeDepMatrix{\Omega}{t}_{uu}\TimeDepMatrix{\Omega}{t}_{vv} + \paren{\TimeDepMatrix{\Omega}{t}_{uv}}^2}^{-1/2} }^p_{u,v=1},\\
&\boldsymbol{\hat{E}}_{t, w} = \sum_{r=1}^{w}\frac{ \paren{\TimeDepMatrix{\hat{\Omega}}{t}\boldsymbol{X}_{t+r}\boldsymbol{X}_{t+r}^\top\TimeDepMatrix{\hat{\Omega}}{t}-\TimeDepMatrix{\hat{\Omega}}{t}}}{\sqrt{w}} \circ \brac{ \paren{\TimeDepMatrix{\hat{\Omega}}{t}_{uu}\TimeDepMatrix{\hat{\Omega}}{t}_{vv} + \paren{\TimeDepMatrix{\hat{\Omega}}{t}_{uv}}^2}^{-1/2} }^p_{u,v=1}.
\end{align*} 
Here $\TimeDepMatrix{\hat{\Omega}}{t}$ denotes the positive semi-definite CLIME precision matrix estimate \cite{cai2011constrained} of $\TimeDepMatrix{\Omega}{t}$ which is obtained from $N$ samples collected prior to $t$ ($\boldsymbol{X}_{t-i},\;i=1,\ldots,N$). We also define $T_t$ and $\hat{T}_t$ by 
\begin{equation*}
T_t = \bbM{1}\paren{\LpNorm{\boldsymbol{E}_{t, w}}{\infty} \geq \zeta_{\pi_0, p, w}},\quad \hat{T}_t = \bbM{1}\paren{\LpNorm{\boldsymbol{\hat{E}}_{t, w}}{\infty} \geq \zeta_{\pi_0, p, w}}.
\end{equation*}
In order to have a compact formulation, set
\begin{equation*}
\TimeDepMatrix{\Psi}{t} \coloneqq \brac{ \paren{\TimeDepMatrix{\Omega}{t}_{uu}\TimeDepMatrix{\Omega}{t}_{vv} + \paren{\TimeDepMatrix{\Omega}{t}_{uv}}^2}^{-1/2} }^p_{u,v=1},\quad\mbox{and}\quad \TimeDepMatrix{\hat{\Psi}}{t}\coloneqq \brac{ \paren{\TimeDepMatrix{\hat{\Omega}}{t}_{uu}\TimeDepMatrix{\hat{\Omega}}{t}_{vv} + \paren{\TimeDepMatrix{\hat{\Omega}}{t}_{uv}}^2}^{-1/2} }^p_{u,v=1}.
\end{equation*}

\begin{proof}[Proof of Theorem \ref{Thm3}]
The goal is to show that $\bb{P}_{\FA}\paren{\hat{T}_t} = \pi_0+o\paren{1}$. So we need a sufficient condition on $N$ for which,
\begin{equation*}
\abs{\LpNorm{\hat{\boldsymbol{E}}_{t,w}}{\infty}-\LpNorm{\boldsymbol{E}_{t,w}}{\infty}} \leq \LpNorm{\hat{\boldsymbol{E}}_{t,w}-\boldsymbol{E}_{t,w}}{\infty} = o_{\bb{P}}\paren{1}.
\end{equation*}
The triangle inequality implies that
\begin{eqnarray*}
\LpNorm{\hat{\boldsymbol{E}}_{t,w}-\boldsymbol{E}_{t,w}}{\infty} &\leq& \LpNorm{\TimeDepMatrix{\hat{\Psi}}{t}-\TimeDepMatrix{\Psi}{t}}{\infty}\LpNorm{\sum_{r=1}^{w}\frac{ \paren{\TimeDepMatrix{\Omega}{t}\boldsymbol{X}_{t+r}\boldsymbol{X}_{t+r}^\top\TimeDepMatrix{\Omega}{t}-\TimeDepMatrix{\Omega}{t}}}{\sqrt{w}}}{\infty}\\
&+&\LpNorm{\TimeDepMatrix{\hat{\Psi}}{t}}{\infty}\LpNorm{\sum_{r=1}^{w}\frac{ \paren{\TimeDepMatrix{\hat{\Omega}}{t}\boldsymbol{X}_{t+r}\boldsymbol{X}_{t+r}^\top\TimeDepMatrix{\hat{\Omega}}{t}-\TimeDepMatrix{\Omega}{t}\boldsymbol{X}_{t+r}\boldsymbol{X}_{t+r}^\top\TimeDepMatrix{\Omega}{t}-\TimeDepMatrix{\hat{\Omega}}{t}+\TimeDepMatrix{\Omega}{t}}}{\sqrt{w}}}{\infty}.
\end{eqnarray*}
Next, we present a slightly weaker (but simpler) upper bound on $\LpNorm{\hat{\boldsymbol{E}}_{t,w}-\boldsymbol{E}_{t,w}}{\infty}$.
\begin{eqnarray}\label{Eq2Thm3}
\LpNorm{\hat{\boldsymbol{E}}_{t,w}-\boldsymbol{E}_{t,w}}{\infty} &\leq& \LpNorm{\TimeDepMatrix{\hat{\Psi}}{t}-\TimeDepMatrix{\Psi}{t}}{\infty}\OpNorm{\TimeDepMatrix{\Omega}{t}}{\infty}{\infty}^2\LpNorm{\sum_{r=1}^{w}\frac{ \paren{\boldsymbol{X}_{t+r}\boldsymbol{X}_{t+r}^\top-\TimeDepMatrix{\Sigma}{t}}}{\sqrt{w}}}{\infty}\nonumber\\
&+&\LpNorm{\TimeDepMatrix{\hat{\Psi}}{t}}{\infty}\LpNorm{\sum_{r=1}^{w}\frac{ \paren{\TimeDepMatrix{\hat{\Omega}}{t}\boldsymbol{X}_{t+r}\boldsymbol{X}_{t+r}^\top\TimeDepMatrix{\hat{\Omega}}{t}-\TimeDepMatrix{\Omega}{t}\boldsymbol{X}_{t+r}\boldsymbol{X}_{t+r}^\top\TimeDepMatrix{\Omega}{t}-\TimeDepMatrix{\hat{\Omega}}{t}+\TimeDepMatrix{\Omega}{t}}}{\sqrt{w}}}{\infty}.
\end{eqnarray}
Let $\spadesuit$ and $\clubsuit$ respectively denote the two terms in the right hand side of Eq. \eqref{Eq2Thm3}. The asymptotic properties of the matrices $\TimeDepMatrix{\Pi}{t} = \TimeDepMatrix{\hat{\Omega}}{t}-\TimeDepMatrix{\Omega}{t}$ and $\TimeDepMatrix{\hat{\Omega}}{t}$ are needed for controlling $\spadesuit$ and $\clubsuit$ from above. Based on the asymptotic results in \cite{cai2011constrained}, if $\TimeDepMatrix{\Omega}{t}$ satisfies Assumption \ref{AssuCLT}, then as $p\rightarrow\infty$
\begin{enumerate}
\item There is $C < \infty$ (depends on the quantities appeared in Assumption \ref{AssuCLT}) such that
\begin{equation*}
\bb{P}\paren{\LpNorm{\TimeDepMatrix{\Pi}{t}}{\infty}\geq C\sqrt{\frac{\log p}{N}}} \leq \frac{1}{p}
\end{equation*}
\item $\TimeDepMatrix{\hat{\Omega}}{t}$ has bounded condition number with probability $1-\cc{O}\paren{p^{-1}}$.
\end{enumerate}
Using these facts, one can easily show that 
\begin{equation*}
\bb{P}\paren{\LpNorm{\TimeDepMatrix{\hat{\Psi}}{t}-\TimeDepMatrix{\Psi}{t}}{\infty}\geq C'\sqrt{\frac{\log p}{N}}} \leq \frac{1}{p}.
\end{equation*}
for a bounded scalar $C'$. Hence, with probability $1-\cc{O}\paren{p^{-1}}$,
\begin{equation*}
\spadesuit\lesssim \sqrt{\frac{\log p}{N}}\OpNorm{\TimeDepMatrix{\Omega}{t}}{\infty}{\infty}^2\LpNorm{\sum_{r=1}^{w}\frac{ \paren{\boldsymbol{X}_{t+r}\boldsymbol{X}_{t+r}^\top-\TimeDepMatrix{\Sigma}{t}}}{\sqrt{w}}}{\infty}\asymp \sqrt{\frac{\log p}{N}}\LpNorm{\sum_{r=1}^{w}\frac{ \paren{\boldsymbol{X}_{t+r}\boldsymbol{X}_{t+r}^\top-\TimeDepMatrix{\Sigma}{t}}}{\sqrt{w}}}{\infty}.
\end{equation*}
We showed in Eq. \eqref{Eq3Thm2} that any diagonal entry of $\TimeDepMatrix{\Sigma}{t}$ is smaller than $\alpha^{-1}_{\min}$ (remember the role of $\alpha_{\min}$ in Assumption \ref{AssuCLT}). Thus an application of Lemma \ref{SampleCovarSupNormUppBnd} yields
\begin{equation}\label{Eq8Thm3}
\bb{P}\paren{\LpNorm{\sum_{r=1}^{w}\frac{ \paren{\boldsymbol{X}_{t+r}\boldsymbol{X}_{t+r}^\top-\TimeDepMatrix{\Sigma}{t}}}{\sqrt{w}}}{\infty} \geq \frac{C''}{\alpha_{\min}}\sqrt{\log p}}\leq \frac{1}{p}.
\end{equation}
for some $C''$. Thus, the bound on $\spadesuit$ is as simple as $\spadesuit = \cc{O}_{\bb{P}}\paren{N^{-1/2}\log p}$, which leads to
\begin{equation*}
\LpNorm{\hat{\boldsymbol{E}}_{t,w}-\boldsymbol{E}_{t,w}}{\infty} = \cc{O}_{\bb{P}}\paren{N^{-1/2}\log p + \clubsuit}.
\end{equation*}
Since $N$ grows faster than $w\log p$ (and hence $\log^2 p$), it suffices to show that,
\begin{equation}\label{Eq9Thm3}
\clubsuit = \cc{O}_{\bb{P}}\paren{\sqrt{N^{-1}w\log p}} =  o_{\bb{P}}\paren{1}.
\end{equation}
We can ignore $\LpNorm{\TimeDepMatrix{\hat{\Psi}}{t}}{\infty}$ in $\clubsuit$, as it remains bounded with high probability. Moreover, $\clubsuit$ depends on samples prior and after $t$ (both pre- and post-change regimes). Namely, there are two sources of randomness in the formulation of $\clubsuit$. The triangle inequality helps us to alleviate this matter by introducing the following upper bound on $\clubsuit$.
\begin{eqnarray}\label{Eq4Thm3}
\clubsuit &\leq& \LpNorm{\sum_{r=1}^{w}\frac{ \paren{\TimeDepMatrix{\hat{\Omega}}{t}\boldsymbol{X}_{t+r}\boldsymbol{X}_{t+r}^\top\TimeDepMatrix{\hat{\Omega}}{t}-\TimeDepMatrix{\Omega}{t}\boldsymbol{X}_{t+r}\boldsymbol{X}_{t+r}^\top\TimeDepMatrix{\Omega}{t}-\TimeDepMatrix{\hat{\Omega}}{t}\TimeDepMatrix{\Sigma}{t}\TimeDepMatrix{\hat{\Omega}}{t}+\TimeDepMatrix{\Omega}{t}}}{\sqrt{w}}}{\infty}\nonumber\\
&+&\sqrt{w} \LpNorm{\TimeDepMatrix{\hat{\Omega}}{t}-\TimeDepMatrix{\hat{\Omega}}{t}\TimeDepMatrix{\Sigma}{t}\TimeDepMatrix{\hat{\Omega}}{t}}{\infty} 
\end{eqnarray}
Let $\clubsuit_1$ and $\clubsuit_2$ represents the two terms in the right hand side of Eq. \eqref{Eq4Thm3}. Next, we control $\clubsuit_1$ and $\clubsuit_2$ from above. Observe that
\begin{eqnarray}\label{Eq5Thm3}
\frac{\clubsuit_2}{\sqrt{w}} &=&  \LpNorm{\TimeDepMatrix{\hat{\Omega}}{t}-\TimeDepMatrix{\hat{\Omega}}{t}\TimeDepMatrix{\Sigma}{t}\TimeDepMatrix{\hat{\Omega}}{t}}{\infty} \leq \OpNorm{\TimeDepMatrix{\hat{\Omega}}{t}}{\infty}{\infty}\LpNorm{\TimeDepMatrix{\hat{\Omega}}{t}\TimeDepMatrix{\Sigma}{t}-I_p}{\infty} \nonumber\\
&\leq& \OpNorm{\TimeDepMatrix{\hat{\Omega}}{t}}{\infty}{\infty}\OpNorm{\TimeDepMatrix{\hat{\Omega}}{t}\TimeDepMatrix{\Sigma}{t}-I_p}{2}{2}
\leq\OpNorm{\TimeDepMatrix{\hat{\Omega}}{t}}{\infty}{\infty}\OpNorm{\TimeDepMatrix{\hat{\Omega}}{t}\TimeDepMatrix{\Sigma}{t}-I_p}{2}{2} \nonumber\\
&\leq& \frac{1}{\alpha_{\min}}\OpNorm{\TimeDepMatrix{\hat{\Omega}}{t}}{\infty}{\infty}\OpNorm{\TimeDepMatrix{\Pi}{t}}{2}{2}.
\end{eqnarray}
The two aforementioned facts from \cite{cai2011constrained}, guarantee that
\begin{equation}\label{Eq7Thm3}
\OpNorm{\TimeDepMatrix{\hat{\Omega}}{t}}{\infty}{\infty} = \cc{O}_{\bb{P}}\paren{1} \,\quad\mbox{and},\quad \OpNorm{\TimeDepMatrix{\Pi}{t}}{2}{2} = \cc{O}_{\bb{P}}\paren{\sqrt{N^{-1}\log p}}.
\end{equation}
Replacing these inequalities into Eq. \eqref{Eq5Thm3} implies that $\clubsuit_2=\cc{O}_{\bb{P}}\paren{\sqrt{N^{-1}w\log p}}$. Further, by using the triangle inequality, and the basic properties of $\ell_\infty\mapsto\ell_\infty$, one can show that
\begin{eqnarray}\label{Eq6Thm3}
\clubsuit_1 &\leq& \paren{\OpNorm{\TimeDepMatrix{\Omega}{t}}{\infty}{\infty}+\OpNorm{\TimeDepMatrix{\hat{\Omega}}{t}}{\infty}{\infty}}\OpNorm{\TimeDepMatrix{\Pi}{t}}{\infty}{\infty}\LpNorm{\sum_{r=1}^{w}\frac{ \paren{\boldsymbol{X}_{t+r}\boldsymbol{X}_{t+r}^\top-\TimeDepMatrix{\Sigma}{t}}}{\sqrt{w}}}{\infty}\nonumber\\
&\leq&\paren{2\OpNorm{\TimeDepMatrix{\Omega}{t}}{\infty}{\infty}+\OpNorm{\TimeDepMatrix{\Pi}{t}}{\infty}{\infty}}\OpNorm{\TimeDepMatrix{\Pi}{t}}{\infty}{\infty}\LpNorm{\sum_{r=1}^{w}\frac{ \paren{\boldsymbol{X}_{t+r}\boldsymbol{X}_{t+r}^\top-\TimeDepMatrix{\Sigma}{t}}}{\sqrt{w}}}{\infty}.
\end{eqnarray}
The details are omitted due to space constraints and the fact that they are rather straightforward algebraic derivations. Again, the asymptotic properties in Eq. \eqref{Eq7Thm3} ensure the existence of a bounded scalar $\tilde{C}$, for which
\begin{equation*}
\bb{P}\paren{\clubsuit_1 \geq \tilde{C}\sqrt{\frac{\log p}{N}}\LpNorm{\sum_{r=1}^{w}\frac{ \paren{\boldsymbol{X}_{t+r}\boldsymbol{X}_{t+r}^\top-\TimeDepMatrix{\Sigma}{t}}}{\sqrt{w}}}{\infty}} \lesssim \frac{1}{p}.
\end{equation*}
Remember from Eq. \eqref{Eq8Thm3} that $\LpNorm{\sum_{r=1}^{w}\frac{ \paren{\boldsymbol{X}_{t+r}\boldsymbol{X}_{t+r}^\top-\TimeDepMatrix{\Sigma}{t}}}{\sqrt{w}}}{\infty} = \cc{O}_{\bb{P}}\paren{\sqrt{\log p}}$. Furthermore, the two terms in the upper bound on $\clubsuit_1$ in Eq. \eqref{Eq6Thm3} are independent, i.e.,
\begin{equation*}
\paren{2\OpNorm{\TimeDepMatrix{\Omega}{t}}{\infty}{\infty}+\OpNorm{\TimeDepMatrix{\Pi}{t}}{\infty}{\infty}}\OpNorm{\TimeDepMatrix{\Pi}{t}}{\infty}{\infty} \independent \LpNorm{\sum_{r=1}^{w}\frac{ \paren{\boldsymbol{X}_{t+r}\boldsymbol{X}_{t+r}^\top-\TimeDepMatrix{\Sigma}{t}}}{\sqrt{w}}}{\infty}.
\end{equation*}
Thus, $\clubsuit_1 = \cc{O}_{\bb{P}}\paren{\sqrt{\frac{\log p}{N}}\sqrt{\log p}} = \cc{O}_{\bb{P}}\paren{\sqrt{N^{-1}\log^2 p}}$. In summary, 
\begin{equation*}
\clubsuit = \cc{O}_{\bb{P}}\paren{\sqrt{N^{-1}w\log p}+\sqrt{N^{-1}\log^2 p}} = \cc{O}_{\bb{P}}\paren{\sqrt{N^{-1}\paren{w\vee\log p}\log p}}.
\end{equation*}
We conclude the proof by recalling that $w^{-1}\log^3 p\rightarrow 0$. Thus, $\clubsuit = \cc{O}_{\bb{P}}\paren{\sqrt{N^{-1}w\log p}}$, which is same as the desired identity in Eq. \eqref{Eq9Thm3}.
\end{proof}

\begin{appendix}

\section{Auxiliary technical results}\label{AppendixA}

This section includes some auxiliary technical lemmas which are used for proving the main results in Section \ref{Proofs}. For the sake of a clear exposition, we provide a succinct summary of each auxiliary result.
\begin{itemize}
\item Lemma \ref{UppBndmomentsSubGauss} gives an upper bound on the expected value of $\max_{1\leq j\leq n} \abs{X_j}^m$, where $\set{X_j}^n_{j=1}$ are sub-Gaussian random variables. Such result is beneficial for proving Theorem \ref{thm0}.
\item Lemma \ref{ChernoffBound}, which presents a cleaner version of the Hanson-Wright inequality \cite{rudelson2013hanson} for multivariate Gaussian vectors, is needed for establishing Theorem \ref{Thm1}.
\item In Lemma \ref{SampleCovarSupNormUppBnd}, we control $\ell_\infty$ norm of the error in estimating sample covariance matrix of i.i.d. $p$-variate standard Gaussian vectors. This result, which is needed for proving Theorems \ref{Thm2} and \ref{Thm3}, has been appeared in \cite{cai2011constrained}. Due to space limitations, we drop the proof of Lemma \ref{SampleCovarSupNormUppBnd} and refer the interested reader to \cite{cai2011constrained} (p. $605$) for full technical details.
\end{itemize}

\begin{lem}\label{UppBndmomentsSubGauss}
Let $\set{X_i}^n_{i=1}$ be a set of sub-Gaussian random variables such that 
\begin{equation*}
\max_{1\leq j\leq n}\bb{P}\paren{\abs{X_j} \geq t} \leq 2\exp\paren{-\frac{t^2}{2\sigma^2}},\quad\forall\;t\geq 0,
\end{equation*}
for some bounded scalar $\sigma > 0$. For any $m\in\bb{N}$, there exists $C_{m}\in\paren{0,\infty}$ such that
\begin{equation*}
\bb{E}\max_{1\leq j\leq n} \abs{X_j}^m \leq C_{m}\sigma^m\log^{m/2} n.
\end{equation*}
\end{lem}

\begin{proof}
Without loss of generality, we assume that $\sigma=1$ and $n\geq 4$. Consider the case that $m$ is an even number. Set $\beta_n = \sqrt{2\log n}$. Applying Fubini's Theorem, we get 
\begin{eqnarray}\label{Eq1App1}
\bb{E}\max_{1\leq j\leq n} \abs{X_j}^m &=& \int_{0}^{\infty} mt^{m-1}\bb{P}\paren{\max_{1\leq j\leq n}\abs{X_j} \geq t} dt \nonumber\\
&=& \int_{0}^{\beta_n} mt^{m-1}\bb{P}\paren{\max_{1\leq j\leq n}\abs{X_j} \geq t} dt+ \int_{\beta_n}^{\infty} mt^{m-1}\bb{P}\paren{\max_{1\leq j\leq n}\abs{X_j} \geq t} dt.
\end{eqnarray}
Next, we control the two terms in the second line of Eq. \eqref{Eq1App1} from above. Observe that,
\begin{equation}\label{Eq2App1}
\int_{0}^{\beta_n} mt^{m-1}\bb{P}\paren{\max_{1\leq j\leq n}\abs{X_j} \geq t} dt \leq \int_{0}^{\beta_n} mt^{m-1} dt = \beta^m_n = \paren{2\log n}^{m/2}.
\end{equation}
Due to the union bound, we have
\begin{equation*}
\bb{P}\paren{\max_{1\leq j\leq n}\abs{X_j} \geq t} \leq \sum_{j=1}^{n} \bb{P}\paren{\abs{X_j} \geq t}\leq 2ne^{-\frac{t^2}{2}}.
\end{equation*}
Using simple integration by substitution techniques, we get
\begin{eqnarray}\label{Eq3App1}
\int_{\beta_n}^{\infty} mt^{m-1}\bb{P}\paren{\max_{1\leq j\leq n}\abs{X_j} \geq t} dt &\leq& 2nm \int_{\beta_n}^{\infty} t^{m-1}e^{-\frac{t^2}{2}} dt = nm2^{m/2}\int_{\log n}^{\infty} t^{m/2-1}e^{-t} dt\nonumber\\
&=& mn2^{m/2} \paren{m/2-1}! e^{-x}\sum_{l=0}^{m/2-1}\frac{x^l}{l!}\Big\lvert^{\infty}_{\log n}\nonumber\\
&=& m2^{m/2} \sum_{l=0}^{m/2-1}\frac{\paren{m/2-1}!\log^l n}{l!}
\end{eqnarray}
Notice that when $n\geq 4$ ($\log n > 1$), there is a bounded scalar $c_m$ such that
\begin{equation*}
m\sum_{l=0}^{m/2-1}\frac{\paren{m/2-1}!\log^l n}{l!} \leq c_m\log^{m/2-1} n.
\end{equation*} 
Finally, replacing Eq. \eqref{Eq2App1} and \eqref{Eq3App1} into Eq. \eqref{Eq1App1} yields
\begin{equation*}
\bb{E}\max_{1\leq j\leq n} \abs{X_j}^m \leq 2^{m/2}\paren{1+ \frac{c_m}{\log n}} \log^{m/2} n \leq 2^{m/2}\paren{1+\frac{c_m}{\log 4}} = C_m \log^{m/2} n,
\end{equation*}
which is the desired upper bound on $\bb{E}\max_{1\leq j\leq n} \abs{X_j}^m$. The proof for an odd $m$ is an immediate consequence of the fact that (which is implied by Holder inequality)
\begin{equation*}
\bb{E}\max_{1\leq j\leq n} \abs{X_j}^m \leq \sqrt{\bb{E}\max_{1\leq j\leq n} \abs{X_j}^{2m}}\leq \sqrt{C_{2m}\log^{m} n} = \sqrt{C_{2m}}\log^{m/2} n.
\end{equation*}
\end{proof}

\begin{lem}\label{ChernoffBound}
Let $\set{\boldsymbol{Z}_j}^n_{j=1}$ be i.i.d. $d$-dimensional standard Gaussian column vectors. Let $\boldsymbol{A}\in\bb{R}^{d\times d}$ be a symmetric matrix, and let $\set{t_m:m\in\bb{N}}$ be a divergent sequence with
\begin{equation*}
\lim\limits_{m\rightarrow\infty} \frac{t^6_m}{m} = 0.
\end{equation*}
Then, the following inequality holds, when $n\rightarrow\infty$.
\begin{equation}
p_n\coloneqq \bb{P}\brac{\sum_{j=1}^{n}\paren{\frac{\boldsymbol{Z}^\top_j\boldsymbol{A}\boldsymbol{Z}_j-\tr \boldsymbol{A}}{\sqrt{2n}\LpNorm{\boldsymbol{A}}{2}}} \geq t_n} \leq e^{-\frac{t^2_n}{2}} \paren{1+o\paren{1}}.
\end{equation}
\end{lem}

\begin{proof}
We control the moment generating function of $\sum_{j=1}^{n}\boldsymbol{Z}^\top_j\boldsymbol{A}\boldsymbol{Z}_j$ from above. Select 
\begin{equation}\label{Eq1LemmaA2}
\beta_n \coloneqq \frac{t_n}{\LpNorm{\boldsymbol{A}}{2}\sqrt{2n}},\quad \Longleftrightarrow\quad \beta_n\sqrt{\frac{2}{n}}\LpNorm{\boldsymbol{A}}{2} = \frac{t_n}{n}.
\end{equation}	
Notice that $\beta_n$ tends to zero, as $n\rightarrow\infty$. Since $\boldsymbol{Z}_j,\;j=1,\ldots, n$ are i.i.d., then
\begin{eqnarray*}
n^{-1}\log\bb{E}\exp\paren{\sum_{j=1}^{n}\beta_n\boldsymbol{Z}_j\boldsymbol{A}\boldsymbol{Z}_j} &=& \log\bb{E}\exp^{\beta_n\boldsymbol{Z}_1\boldsymbol{A}\boldsymbol{Z}_1} = \log\int_{\bb{R}^d} \paren{2\pi}^{-d/2}\exp\paren{\frac{\boldsymbol{u}^\top\paren{2\beta_n\boldsymbol{A}-\boldsymbol{I}_d}\boldsymbol{u}}{2}} d\boldsymbol{u}\\
&\RelNum{\paren{a}}{=}& \frac{-1}{2n}\log\det\Bigparen{-2\beta_n\boldsymbol{A}+\boldsymbol{I}_d}.
\end{eqnarray*}
We know that the identity $\paren{a}$ is valid, as $2\beta_n\OpNorm{\boldsymbol{A}}{2}{2} < 1$ for large enough $n$. With above identity, the \emph{Chernoff bound} can be written in the following from.
\begin{eqnarray}\label{Eq2LemmaA2}
\frac{\log p_n}{n} &\leq& -\beta_n\tr\paren{\boldsymbol{A}}-\beta_nt_n\sqrt{\frac{2}{n}}\LpNorm{\boldsymbol{A}}{2}+\log\bb{E}\exp^{\beta_n\boldsymbol{Z}_1\boldsymbol{A}\boldsymbol{Z}_1}\nonumber\\
&=&-\beta_n\tr\paren{\boldsymbol{A}}-\beta_nt_n\sqrt{\frac{2}{n}}\LpNorm{\boldsymbol{A}}{2}-\frac{1}{2}\log\det\Bigparen{-2\beta_n\boldsymbol{A}+\boldsymbol{I}_d}\nonumber\\
&=&\frac{-1}{2}\sum_{i=1}^{d}\Bigparen{2\beta_n\lambda_i\paren{\boldsymbol{A}}+\log\paren{1-2\beta_n\lambda_i\paren{\boldsymbol{A}}}}-\beta_nt_n\sqrt{\frac{2}{n}}\LpNorm{\boldsymbol{A}}{2}.
\end{eqnarray}
For brevity, define $f:\parbra{-\infty, 1}\mapsto\bb{R}$ by $f\paren{x} = x+\log\paren{1-x}+x^2/2$. Selection of $\beta_n$ in Eq. \eqref{Eq1LemmaA2} implies that
\begin{equation*}
\beta_nt_n\sqrt{\frac{2}{n}}\LpNorm{\boldsymbol{A}}{2} = \frac{t^2_n}{2n} + \beta^2_n\LpNorm{\boldsymbol{A}}{2}^2.
\end{equation*} 
Hence, Eq. \eqref{Eq2LemmaA2} can be rewritten as
\begin{eqnarray*}
\frac{\log p_n}{n} &\leq& \frac{-1}{2}\sum_{i=1}^{d}\Bigparen{2\beta_n\lambda_i\paren{\boldsymbol{A}}+\log\paren{1-2\beta_n\lambda_i\paren{\boldsymbol{A}}}}-\beta^2_n\LpNorm{\boldsymbol{A}}{2}^2-\frac{t^2_n}{2n}\\
&=&\frac{-1}{2}\sum_{i=1}^{d}\Bigparen{2\beta_n\lambda_i\paren{\boldsymbol{A}}+\log\paren{1-2\beta_n\lambda_i\paren{\boldsymbol{A}}}-\frac{\paren{2\beta_n\lambda_i\paren{\boldsymbol{A}}}^2}{2}}-\frac{t^2_n}{2n} \\
&=& \frac{-1}{2}\sum_{i=1}^{d}f\Bigparen{2\beta_n\lambda_i\paren{\boldsymbol{A}}}-\frac{t^2_n}{2n} = \frac{1}{n}\brac{-\frac{n}{2}\sum_{i=1}^{d}f\Bigparen{2\beta_n\lambda_i\paren{\boldsymbol{A}}}-\frac{t^2_n}{2}}.
\end{eqnarray*}
So, it suffices to prove that $n\sum_{i=1}^{d}f\Bigparen{2\beta_n\lambda_i\paren{\boldsymbol{A}}} = o\paren{1}$, as $n\rightarrow\infty$. Since $2\beta_n\OpNorm{\boldsymbol{A}}{2}{2}$ lies in a small neighborhood of $0$, for large $n$, then the Taylor expansion of $f$ near zero yields
\begin{equation*}
f\Bigparen{2\beta_n\lambda_i\paren{\boldsymbol{A}}} = \sum_{l=3}^{\infty} \frac{\Bigparen{-2\beta_n\lambda_i\paren{\boldsymbol{A}}}^l}{l} = \cc{O}\paren{\frac{t^3_n}{n^{3/2}}},\quad \forall\;i=1,\ldots, d.
\end{equation*}
Hence,
\begin{equation*}
\abs{n\sum_{i=1}^{d}f\Bigparen{2\beta_n\lambda_i\paren{\boldsymbol{A}}}} \leq nd\max_{1\leq i \leq d}\abs{f\Bigparen{2\beta_n\lambda_i\paren{\boldsymbol{A}}}} = \cc{O}\paren{nd\frac{t^3_n}{n^{3/2}}} = \cc{O}\paren{\sqrt{\frac{t^6_n}{n}}}.
\end{equation*}
We conclude the proof by recalling that $n^{-1}t^6_n\rightarrow 0$, when $n\rightarrow\infty$.
\end{proof}

\begin{lem}\label{SampleCovarSupNormUppBnd}
Let $\boldsymbol{Z}_i,\;i=1,\ldots, w$ be i.i.d. zero-mean Gaussian random vectors in $\bb{R}^p$ with covariance matrix $\boldsymbol{\Sigma}$. For any $\xi >0$, there is a bounded scalar $C_\xi$ (depending only on $\xi$) such that
\begin{equation*}
\bb{P}\paren{\LpNorm{\frac{1}{w}\sum_{j=0}^{w} \boldsymbol{Z}_j\boldsymbol{Z}^\top_j - \boldsymbol{\Sigma} }{\infty} \geq C_\xi\max_{1\leq j \leq p} \boldsymbol{\Sigma}_{jj}\sqrt{\frac{\log p}{w}} } \leq p^{-\xi}.
\end{equation*}
\end{lem}

\section{The non-asymptotic analysis of the inner product of dependent Gaussian random vectors}\label{AppendixB}

Let $\boldsymbol{X}, \boldsymbol{Y}\in\bb{R}^n$ be standard Gaussian column vectors with $\cov\paren{X_i, Y_i} = r,\;\forall\;i$. Set
\begin{equation}\label{Vn}
V_n\paren{r} \coloneqq \frac{1}{\sqrt{n\paren{1+r^2}}}\paren{\boldsymbol{X}^\top\boldsymbol{Y}-nr}.
\end{equation}
Also, define the \emph{standardized incomplete gamma function} by
\begin{equation*}
Q\paren{m, z} \coloneqq \frac{1}{\Gamma\paren{m}}\int_{z}^{\infty} x^{m-1} e^{-x} dx,
\end{equation*}
where $\Gamma\paren{\cdot}$ denotes the gamma function. 

\begin{thm}\label{MainThmApp}
Let $Z$ be a standard normal random variable. Let $\set{t_n}_{n\in\bb{N}}$ be a positive sequence such that $t_n\rightarrow\infty$ and
\begin{equation}\label{Cond}
\lim\limits_{n\rightarrow\infty} \frac{t^6_n}{n} = 0.
\end{equation}
Then,
\begin{equation*}
\lim\limits_{n\rightarrow\infty}\frac{\bb{P}\paren{V_n\paren{r} \geq t_n}}{\bb{P}\paren{Z\geq t_n}} = 1.
\end{equation*}
\end{thm}

We first present some technical results for enhancing the readability of the lengthy proof of Theorem \ref{MainThmApp}. 

\begin{prop}\label{GaussInnerEqRep}
Let $Z$ be a standard normal random variable independent of $\boldsymbol{X}$. Then,
\begin{equation*}
V_n\paren{0} \eqd \frac{\LpNorm{\boldsymbol{X}}{2}}{\sqrt{n}}Z.
\end{equation*}
\end{prop}

\begin{proof}
Trivially $V_n\paren{0} = \frac{\LpNorm{\boldsymbol{X}}{2}}{\sqrt{n}} \InnerProd{\boldsymbol{Y}}{\frac{\boldsymbol{X}}{\LpNorm{\boldsymbol{X}}{2}}}$. We first show that the two terms in the alternative representation of $V_n\paren{0}$ are independent. Choose $\sigma>0$ and set $\boldsymbol{X}_\sigma \coloneqq \sigma\boldsymbol{X}$. Note that $\LpNorm{\boldsymbol{X}_\sigma}{2}^2/n$ is the complete sufficient statistic for estimating $\sigma^2$. Furthermore,
\begin{equation*}
\InnerProd{\boldsymbol{Y}}{\frac{\boldsymbol{X}}{\LpNorm{\boldsymbol{X}}{2}}} = \InnerProd{\boldsymbol{Y}}{\frac{\boldsymbol{X}_\sigma}{\LpNorm{\boldsymbol{X}_\sigma}{2}}}
\end{equation*}
is ancillary to $\sigma$. So, our claim is an immediate consequence of \emph{Basu's theorem}. Thus, it suffices to show that
\begin{equation*}
\InnerProd{\boldsymbol{Y}}{\frac{\boldsymbol{X}}{\LpNorm{\boldsymbol{X}}{2}}} \sim \cc{N}\paren{0, 1}.
\end{equation*} 
Let $\boldsymbol{R}$ be the rotation matrix in $\bb{R}^n$ with $\frac{\boldsymbol{R}\boldsymbol{X}}{\LpNorm{\boldsymbol{X}}{2}} = \brac{1,0,\ldots,0}^\top$. The orthogonality of $\boldsymbol{R}$ yields
\begin{equation*}
\InnerProd{\boldsymbol{Y}}{\frac{\boldsymbol{X}}{\LpNorm{\boldsymbol{X}}{2}}} = \InnerProd{\boldsymbol{R}^\top \boldsymbol{X}}{\brac{1,0,\ldots,0}^\top} = \paren{\boldsymbol{R}^\top \boldsymbol{Y}}_1.
\end{equation*}
The rotational invariance feature of Gaussian vectors implies that $\paren{\boldsymbol{R}^\top \boldsymbol{Y}}_1\eqd Z$, concluding the proof.
\end{proof}

\begin{lem}\label{Lemma1}
Let $Z_1$ and $Z_2$ be two random variables. Then, for any $h\in\paren{0,\infty}$
\begin{equation*}
\bb{P}\paren{\abs{Z_1Z_2} \geq h^2} \leq \bb{P}\paren{\abs{Z_1 \geq h}} + \bb{P}\paren{\abs{Z_2 \geq h}}.
\end{equation*}
\end{lem}

\begin{proof}
The following series of straightforward inequalities shows the desired result. 
\begin{eqnarray*}
\bb{P}\paren{\abs{Z_1Z_2} \geq h^2} &=& \bb{P}\brac{ \paren{\abs{Z_1Z_2} \geq h^2} \bigcap \paren{\abs{Z_1} \geq h}} + \bb{P}\brac{ \paren{\abs{Z_1Z_2} \geq h^2} \bigcap \paren{\abs{Z_1} \leq h}}\\
&\leq& \bb{P}\paren{\abs{Z_1} \geq h} + \bb{P}\brac{ \paren{\abs{Z_2} \geq \frac{h^2}{\abs{Z_1}}} \bigcap \paren{\abs{Z_1} \leq h}} \\
&\leq& \bb{P}\paren{\abs{Z_1} \geq h} + \bb{P}\brac{ \paren{\abs{Z_2} \geq h} \bigcap \paren{\abs{Z_1} \leq h}} \leq \bb{P}\paren{\abs{Z_1} \geq h} + \bb{P}\paren{\abs{Z_2} \geq h}.
\end{eqnarray*}
\end{proof}

We are now ready to establish Theorem \ref{MainThmApp}.
\begin{proof}[Proof of Theorem \ref{MainThmApp}]
Throughout the proof, $Z$ denotes a standard Gaussian random variable. Without loss of generality, we assume that $r$ is non-negative. For improving readability, we first focus on the special case of $r=1$. In this scenario, $\boldsymbol{X}=\boldsymbol{Y}$ and thus
\begin{equation*}
V_n\paren{1} = \frac{1}{\sqrt{2n}}\paren{\LpNorm{\boldsymbol{X}}{2}^2-n}.
\end{equation*}
We obtain an equivalent formulation for $\bb{P}\paren{V_n\paren{1} \geq t_n}$ by using the fact that $\LpNorm{\boldsymbol{X}}{2}^2$ is a $\chi^2$ random variable with $n$ degrees of freedom.
\begin{eqnarray}\label{QDist}
\bb{P}\paren{V_n\paren{1} \geq t_n} &=& \bb{P}\paren{\LpNorm{\boldsymbol{X}}{2}^2 \geq n+t_n\sqrt{2n}} = \int_{n+t_n\sqrt{2n}}^{\infty} \frac{u^{\frac{n}{2}-1}e^{-\frac{u}{2} }}{2^{\frac{n}{2}}\Gamma\paren{\frac{n}{2}} }du = \int_{\frac{n}{2}+t_n\sqrt{\frac{n}{2}}}^{\infty} \frac{u^{\frac{n}{2}-1}e^{-\frac{u}{2} }}{\Gamma\paren{\frac{n}{2}} }du\nonumber\\
&=& \Gamma^{-1}\paren{\frac{n}{2}} \int_{\frac{n}{2}+t_n\sqrt{\frac{n}{2}}}^{\infty} u^{\frac{n}{2}-1}e^{-\frac{u}{2}}du = Q\paren{\frac{n}{2}, \frac{n}{2}+t_n\sqrt{\frac{n}{2}}}.
\end{eqnarray}
The asymptotic properties of $Q\paren{\cdot,\cdot}$ has been studied in \cite{nemes2019asymptotic}. According to Theorem $1.1$ in \cite{nemes2019asymptotic}, as $n, t_n\rightarrow\infty$, we have
\begin{equation}\label{AsympIdentity}
Q\paren{\frac{n}{2}, \frac{n}{2}+t_n\sqrt{\frac{n}{2}}} \sim \bb{P}\paren{Z\geq t_n} + \sqrt{\frac{2}{n\pi}} \exp\paren{-\frac{t^2_n}{2}} \sum_{j=0}^{\infty} \paren{\frac{2}{n}}^{j/2} B_j\paren{t_n},
\end{equation}
in which $B_j\paren{x}$ is a polynomial of degree $\paren{3j+2}$ for any non-negative $x$. We proceed by evaluating the asymptotic behaviour of the second term in the right hand side of Eq. \eqref{AsympIdentity}. Observe that there exists a large enough scalar $C$ such that
\begin{eqnarray*}
\sqrt{\frac{2}{n\pi}} \exp\paren{-\frac{t^2_n}{2}} \sum_{j=0}^{\infty} \paren{\frac{2}{n}}^{j/2} B_j\paren{t_n} &\asymp& \frac{1}{\sqrt{n}} \exp\paren{-\frac{t^2_n}{2}} \sum_{j=0}^{\infty} \paren{\frac{Ct^6_n}{n}}^{j/2} t^2_n \\
&=& \frac{\exp\paren{-\frac{t^2_n}{2}}}{t_n}\sum_{j=1}^{\infty} \paren{\frac{Ct^6_n}{n}}^{j/2},
\end{eqnarray*}
as $n, t_n \rightarrow\infty$. It is also known that \cite{borjesson1979simple} 
\begin{equation*}
\bb{P}\paren{Z\geq t_n} \asymp \frac{\exp\paren{-\frac{t^2_n}{2}}}{t_n},\quad \mbox{when}\;t_n\rightarrow\infty.
\end{equation*}
So $\sqrt{\frac{2}{n\pi}} \exp\paren{-\frac{t^2_n}{2}} \sum_{j=0}^{\infty} \paren{\frac{2}{n}}^{j/2} B_j\paren{t_n} \asymp \bb{P}\paren{Z\geq t_n} \sum_{j=1}^{\infty} \paren{\frac{Ct^6_n}{n}}^{j/2}$. Therefore, Eq. \eqref{AsympIdentity} can be rewritten in the following from.
\begin{equation}\label{Eq7}
Q\paren{\frac{n}{2}, \frac{n}{2}+t_n\sqrt{\frac{n}{2}}} \sim \bb{P}\paren{Z\geq t_n}\paren{1+\sum_{j=1}^{\infty} \paren{\frac{Ct^6_n}{n}}^{j/2}}.
\end{equation}
The condition \eqref{Cond} implies that $\sum_{j=1}^{\infty} \paren{Cn^{-1}t^6_n}^{j/2}\rightarrow 0$. Thus, 
\begin{equation*}
\lim\limits_{n\rightarrow\infty} \frac{\bb{P}\paren{V_n\paren{1} \geq t_n} }{\bb{P}\paren{Z\geq t_n}} = \lim\limits_{n\rightarrow\infty} \frac{Q\paren{\frac{n}{2}, \frac{n}{2}+t_n\sqrt{\frac{n}{2}}} }{\bb{P}\paren{Z\geq t_n}} = 1 + \lim\limits_{n\rightarrow\infty} \sum_{j=1}^{\infty} \paren{Cn^{-1}t^6_n}^{j/2} = 1.
\end{equation*}

Next, we extend the proof to any $r\in\paren{-1, 1}$. Let $\boldsymbol{X'}$ be a standard Gaussian random vector independent of $\boldsymbol{X}$. Then, $\boldsymbol{Y}$ can be decomposed in the following way.
\begin{equation*}
\boldsymbol{Y} \eqd r\boldsymbol{X} + \sqrt{1-r^2}\boldsymbol{X'}.
\end{equation*}
For brevity, set $\rho \coloneqq \sqrt{\frac{2r^2}{1+r^2}}$ and  $\Psi_n \coloneqq \frac{\paren{\LpNorm{\boldsymbol{X}}{2}^2-n}}{\sqrt{2n}}$. The new representation of $\boldsymbol{Y}$ implies that
\begin{equation*}
V_n\paren{r} \eqd \sqrt{\rho} \frac{1}{\sqrt{2n}}\paren{\LpNorm{\boldsymbol{X}}{2}^2-n} + \sqrt{1-\rho} \frac{\InnerProd{\boldsymbol{X}}{\boldsymbol{X'}}}{\sqrt{n}} = \sqrt{\rho}\Psi_n +\sqrt{1-\rho} \frac{\InnerProd{\boldsymbol{X}}{\boldsymbol{X'}}}{\sqrt{n}}.
\end{equation*}
Applying the result in Proposition \ref{GaussInnerEqRep} implies that
\begin{eqnarray}\label{Eq1}
\bb{P}\paren{V_n\paren{r} \geq t_n} &=& \bb{P}\paren{\sqrt{\rho}\Psi_n +\sqrt{1-\rho} \frac{\InnerProd{\boldsymbol{X}}{\boldsymbol{X'}}}{\sqrt{n}} \geq t_n} = \bb{P}\paren{\sqrt{\rho}\Psi_n +\sqrt{1-\rho}Z\frac{\LpNorm{\boldsymbol{X}}{2}}{\sqrt{n}} \geq t_n}\nonumber\\
&=& \bb{P}\paren{\sqrt{\rho}\Psi_n +\sqrt{1-\rho}Z\sqrt{1+\sqrt{\frac{2}{n}}\Psi_n} \geq t_n}
\end{eqnarray}
We break the expression in the second line of Eq. \eqref{Eq1} into two terms. Set
\begin{align}
&p_0 \coloneqq \bb{P}\brac{\paren{\sqrt{\rho}\Psi_n +\sqrt{1-\rho}Z\sqrt{1+\sqrt{\frac{2}{n}}\Psi_n} \geq t_n} \bigcap \paren{\abs{\Psi_n Z} \leq 4t^2_n}},\nonumber\\
&p_1 \coloneqq \bb{P}\brac{\paren{\sqrt{\rho}\Psi_n +\sqrt{1-\rho}Z\sqrt{1+\sqrt{\frac{2}{n}}\Psi_n} \geq t_n} \bigcap \paren{\abs{\Psi_n Z} \geq 4t^2_n}}.
\end{align}
Thus, Eq. \eqref{Eq1} can be rewritten as $\bb{P}\paren{V_n\paren{r} \geq t_n} = p_0+p_1$. 
\begin{clawithinpf}\label{Claim1}
$\bb{P}^{-1}\paren{Z\geq t_n}p_1 \rightarrow 0$, as $n\rightarrow\infty$.
\end{clawithinpf}
\begin{proof}[Proof of Claim \ref{Claim1}]
An application of Lemma \ref{Lemma1} introduces an upper bound on $p_1$.
\begin{eqnarray}\label{p1UppBnd}
p_1 &\leq& \bb{P}\paren{\abs{\Psi_n Z} \geq 4t^2_n} \leq \bb{P}\paren{\abs{\Psi_n} \geq 2t_n} + \bb{P}\paren{\abs{Z} \geq 2t_n} \nonumber\\
&=& 2\brac{\bb{P}\paren{\Psi_n \geq 2t_n} + \bb{P}\paren{Z \geq 2t_n}}.
\end{eqnarray}
Since $\Psi_n$ has the same distribution as $V_n\paren{1}$, then (as we showed in the first part of the proof)
\begin{equation}\label{Eq2}
\lim\limits_{n,t_n\rightarrow\infty} \frac{\bb{P}\paren{\Psi_n \geq 2t_n}}{\bb{P}\paren{Z \geq 2t_n}} = 1.
\end{equation}
Combining Eq. \eqref{p1UppBnd} and \eqref{Eq2} implies that $p_1\lesssim \bb{P}\paren{Z\geq 2t_n}$. Hence,
\begin{equation*}
\lim\limits_{n,t_n\rightarrow\infty}\frac{p_1}{\bb{P}\paren{Z\geq t_n}} \lesssim \lim\limits_{n,t_n\rightarrow\infty}\frac{\bb{P}\paren{Z\geq 2t_n}}{\bb{P}\paren{Z\geq t_n}} \lesssim \lim\limits_{n,t_n\rightarrow\infty}\frac{\exp\paren{-2t^2_n}}{\exp\paren{-t^2_n/2}} = 0.
\end{equation*}
\end{proof}
In summary, we get the following asymptotic identity so far
\begin{equation*}
\frac{\bb{P}\paren{V_n\paren{r} \geq t_n}}{\bb{P}\paren{Z\geq t_n}} =  \frac{p_0+p_1}{\bb{P}\paren{Z\geq t_n}}\sim \frac{p_0}{\bb{P}\paren{Z\geq t_n}},\quad\mbox{when}\; n,t_n\rightarrow\infty.
\end{equation*}
Next, we introduce upper and lower bounds on $\frac{p_0}{\bb{P}\paren{Z\geq t_n}}$. If $\abs{\Psi_n Z} \leq 4t^2_n$, then
\begin{eqnarray*}
\sqrt{\rho}\Psi_n +\sqrt{1-\rho}Z\sqrt{1+\sqrt{\frac{2}{n}}\Psi_n} &=& \sqrt{\rho}\Psi_n +\sqrt{1-\rho}Z +\sqrt{1-\rho}Z\brac{\sqrt{1+\sqrt{\frac{2}{n}}\Psi_n}-1}\\
&\leq& \sqrt{\rho}\Psi_n +\sqrt{1-\rho}Z + \sqrt{\frac{1-\rho}{2n}}\abs{\Psi_n Z} \\
&\leq& \sqrt{\rho}\Psi_n +\sqrt{1-\rho}Z + \sqrt{\frac{8}{n}}t^2_n.
\end{eqnarray*}
Hence,
\begin{eqnarray*}
p_0 &=& \bb{P}\brac{\paren{\sqrt{\rho}\Psi_n +\sqrt{1-\rho}Z\sqrt{1+\sqrt{\frac{2}{n}}\Psi_n} \geq t_n} \bigcap \paren{\abs{\Psi_n Z} \leq 4t^2_n}} \\
&\leq& \bb{P}\paren{\sqrt{\rho}\Psi_n +\sqrt{1-\rho}Z \geq t_n-\sqrt{\frac{8}{n}}t^2_n}.
\end{eqnarray*}
Similarly, one can show that
\begin{equation*}
p_0 \geq \bb{P}\paren{\sqrt{\rho}\Psi_n +\sqrt{1-\rho}Z \geq t_n+\sqrt{\frac{8}{n}}t^2_n} - \bb{P}\paren{\abs{\Psi_n Z} \geq 4t^2_n}.
\end{equation*}
Thus, if we combine the upper and lower bounds on $p_0$ with the result in Claim \ref{Claim1}, we obtain
\begin{equation*}
\frac{ \bb{P}\paren{\sqrt{\rho}\Psi_n +\sqrt{1-\rho}Z \geq t_n+\sqrt{\frac{8}{n}}t^2_n}}{\bb{P}\paren{Z\geq t_n}} \leq \frac{p_0}{\bb{P}\paren{Z\geq t_n}} \leq \frac{ \bb{P}\paren{\sqrt{\rho}\Psi_n +\sqrt{1-\rho}Z \geq t_n-\sqrt{\frac{8}{n}}t^2_n}}{\bb{P}\paren{Z\geq t_n}}.
\end{equation*}
Set 
\begin{equation*}
t^+_n \coloneqq t_n+\frac{\sqrt{8}t^2_n}{\sqrt{n}}, \quad \mbox{and} \quad t^-_n \coloneqq t_n-\frac{\sqrt{8}t^2_n}{\sqrt{n}}
\end{equation*}
Since $\frac{t^3_n}{\sqrt{n}}\rightarrow 0$, then $\bb{P}\paren{Z\geq t^+_n} \sim \bb{P}\paren{Z\geq t_n} \sim \bb{P}\paren{Z\geq t^-_n}$. Thus,
\begin{eqnarray}\label{Eq3}
\lim\limits_{n,t_n\rightarrow\infty}\frac{\bb{P}\paren{\sqrt{\rho}\Psi_n +\sqrt{1-\rho}Z\geq t^+_n}}{\bb{P}\paren{Z\geq t^+_n}}&\leq& \lim\limits_{n,t_n\rightarrow\infty}\frac{\bb{P}\paren{V_n\paren{r} \geq t_n}}{\bb{P}\paren{Z\geq t_n}} \nonumber\\
&\leq& \lim\limits_{n,t_n\rightarrow\infty}\frac{\bb{P}\paren{\sqrt{\rho}\Psi_n \sqrt{1-\rho}Z\geq t^-_n}}{\bb{P}\paren{Z\geq t^-_n}}.
\end{eqnarray}
Lastly, we show that the upper and lower bounds on $\frac{\bb{P}\paren{V_n\paren{r} \geq t_n}}{\bb{P}\paren{Z\geq t_n}}$ in Eq. \eqref{Eq3} converge to one, if $n,t_n\rightarrow\infty$ in a way that the condition \eqref{Cond} is satisfied. We only prove that
\begin{equation}\label{Eq5}
\frac{\bb{P}\paren{\sqrt{\rho}\Psi_n +\sqrt{1-\rho}Z\geq t^-_n}}{\bb{P}\paren{Z\geq t^-_n}} \rightarrow 1,\quad\mbox{when}\;t_n, n\rightarrow\infty.
\end{equation}
The other asymptotic identity can be proved in an analogous way. Without loss of generality we can assume that $\rho > 0$. Let $\phi\paren{\cdot}$ and $\bar{\Phi}\paren{\cdot}$ successively denote the standard Gaussian probability density and complementary cumulative distribution functions. Furthermore, define $\Xi_n\coloneqq \sqrt{\rho}\Psi_n +\sqrt{1-\rho}Z$. Recall the distribution of $\Psi_n$ from Eq. \eqref{QDist}. Observe that
\begin{equation*}
\bb{P}\paren{\Xi_n\geq t^-_n} = \bb{E}\brac{\bb{P}\paren{\Xi_n\geq t^-_n \arrowvert Z} } = \int_{\bb{R}} \phi\paren{u}Q\paren{\frac{n}{2}, \frac{n}{2}+\sqrt{\frac{n}{2}}\paren{\frac{t^-_n}{\sqrt{\rho}}-u\sqrt{\frac{1-\rho}{\rho}}}} du.
\end{equation*}
We now obtain a similar formulation for $\bb{P}\paren{Z\geq t^-_n}$. Notice that, there are two independent standard Gaussian random random variables such that $Z\eqd \sqrt{\rho} Z_1+\sqrt{1-\rho}Z_2$. Therefore, 
\begin{equation}\label{Eq4prime}
\bb{P}\paren{Z\geq t^-_n} = \bb{E}\brac{\bb{P}\paren{\sqrt{\rho} Z_1+\sqrt{1-\rho}Z_2\geq t^-_n \arrowvert Z_1} } = \int_{\bb{R}} \phi\paren{u} \bar{\Phi}\paren{\frac{t^-_n}{\sqrt{\rho}}-u\sqrt{\frac{1-\rho}{\rho}}} du.
\end{equation}
For brevity, define $\Delta\paren{x} \coloneqq Q\paren{\frac{n}{2}, \frac{n}{2}+\sqrt{\frac{n}{2}}x}-\bar{\Phi}\paren{x},\; \forall\; x\in\bb{R}$. Combining the last two identities, yields
\begin{equation}\label{Eq4}
\frac{\bb{P}\paren{\Xi_n\geq t^-_n}}{\bb{P}\paren{Z\geq t^-_n}}-1 = \frac{\int_{\bb{R}} \phi\paren{u}\Delta\paren{\frac{t^-_n}{\sqrt{\rho}}-u\sqrt{\frac{1-\rho}{\rho}}} du}{\bb{P}\paren{Z\geq t^-_n}}.
\end{equation}
Recall that $\rho=\sqrt{\frac{2r^2}{1+r^2}} > 0$ for some $r\in\paren{-1,1}$. So, it is possible to choose $\epsilon\in\paren{\rho, 1}$. Set
\begin{equation*}
\cc{B}_{\epsilon,\rho} \coloneqq \set{u:\; \abs{u}\leq t^-_n\sqrt{\frac{1-\epsilon}{1-\rho}}},\quad \cc{B}^c_{\epsilon,\rho} \coloneqq \set{u:\; \abs{u} > t^-_n\sqrt{\frac{1-\epsilon}{1-\rho}}}.
\end{equation*}
Now we can decompose the integral in Eq. \eqref{Eq4} into two parts.
\begin{equation}\label{Eq6}
\frac{\bb{P}\paren{\Xi_n\geq t^-_n}}{\bb{P}\paren{Z\geq t^-_n}}-1 = \frac{\int_{\cc{B}_{\epsilon,\rho}} \phi\paren{u}\Delta\paren{\frac{t^-_n}{\sqrt{\rho}}-u\sqrt{\frac{1-\rho}{\rho}}} du}{\bb{P}\paren{Z\geq t^-_n}} + \frac{\int_{\cc{B}^c_{\epsilon,\rho}} \phi\paren{u}\Delta\paren{\frac{t^-_n}{\sqrt{\rho}}-u\sqrt{\frac{1-\rho}{\rho}}} du}{\bb{P}\paren{Z\geq t^-_n}}.
\end{equation}
Let $p_0$ and $p_1$ stand for the two expressions on the right hand side of Eq. \eqref{Eq6}. For proving identity \eqref{Eq5}, it suffices to show that both $p_0$ and $p_1$ tend to zero. One line of straightforward algebra implies that
\begin{eqnarray*}
\abs{p_1} &=& \frac{\abs{\int_{\cc{B}^c_{\epsilon,\rho}} \phi\paren{u}\Delta\paren{\frac{t^-_n}{\sqrt{\rho}}-u\sqrt{\frac{1-\rho}{\rho}}} du}}{\bb{P}\paren{Z\geq t^-_n}} \leq \frac{\int_{\cc{B}^c_{\epsilon,\rho}} \phi\paren{u}\abs{\Delta\paren{\frac{t^-_n}{\sqrt{\rho}}-u\sqrt{\frac{1-\rho}{\rho}}} du}}{\bb{P}\paren{Z\geq t^-_n}} \nonumber\\
&\leq& \frac{\int_{\cc{B}^c_{\epsilon,\rho}} \phi\paren{u}du}{\bb{P}\paren{Z\geq t^-_n}} =  \frac{\bb{P}\paren{Z\geq t^-_n\sqrt{\frac{1-\epsilon}{1-\rho}}}}{\bb{P}\paren{Z\geq t^-_n}}\RelNum{\paren{a}}{\rightarrow} 0.
\end{eqnarray*}
Notice that, the asymptotic identity $\paren{a}$ holds, as $t^-_n\rightarrow\infty$ and $\sqrt{\frac{1-\epsilon}{1-\rho}}$ is strictly greater than $1$. We finally show that $\abs{p_0}\rightarrow 0$. Applying the triangle inequality implies that for any $u\in \cc{B}_{\epsilon,\rho}$,
\begin{equation*}
\abs{t^-_n} \paren{\frac{1+\sqrt{1-\epsilon}}{\sqrt{\rho}}}\geq \frac{\abs{t^-_n}}{\sqrt{\rho}}+\abs{u}\sqrt{\frac{1-\rho}{\rho}}\geq \abs{\frac{t^-_n}{\sqrt{\rho}}-u\sqrt{\frac{1-\rho}{\rho}}} \geq \frac{\abs{t^-_n}}{\sqrt{\rho}}-\abs{u}\sqrt{\frac{1-\rho}{\rho}}\geq \abs{t^-_n} \paren{\frac{1-\sqrt{1-\epsilon}}{\sqrt{\rho}}}.
\end{equation*} 
In words, $\paren{\frac{t^-_n}{\sqrt{\rho}}-u\sqrt{\frac{1-\rho}{\rho}}}$ grows with the same rate as $t_n$, $\forall\;u\in\cc{B}_{\epsilon,\rho}$. So, based on Eq. \eqref{Eq7},
\begin{equation}\label{Eq8}
\abs{\frac{\Delta\paren{\frac{t^-_n}{\sqrt{\rho}}-u\sqrt{\frac{1-\rho}{\rho}}}}{\bb{P}\paren{Z\geq \frac{t^-_n}{\sqrt{\rho}}-u\sqrt{\frac{1-\rho}{\rho}}}}} \asymp \sum_{j=1}^{\infty} \paren{\frac{t^6_n}{n}}^{j/2} \asymp\sqrt{\frac{t^6_n}{n}},\quad\forall\;u\in\cc{B}_{\epsilon,\rho}.
\end{equation}
Therefore,
\begin{eqnarray*}
\abs{p_0} &\leq& \int_{\cc{B}_{\epsilon,\rho}} \phi\paren{u}\abs{\frac{\Delta\paren{\frac{t^-_n}{\sqrt{\rho}}-u\sqrt{\frac{1-\rho}{\rho}}}}{\bb{P}\paren{Z\geq t^-_n}}} du	\lesssim \sqrt{\frac{t^6_n}{n}} \int_{\cc{B}_{\epsilon,\rho}} \phi\paren{u} \frac{\bb{P}\paren{Z\geq \frac{t^-_n}{\sqrt{\rho}}-u\sqrt{\frac{1-\rho}{\rho}}}}{\bb{P}\paren{Z\geq t^-_n}} du \\
&\lesssim& \sqrt{\frac{t^6_n}{n}} \int_{\bb{R}} \phi\paren{u} \frac{\bb{P}\paren{Z\geq \frac{t^-_n}{\sqrt{\rho}}-u\sqrt{\frac{1-\rho}{\rho}}}}{\bb{P}\paren{Z\geq t^-_n}} du.
\end{eqnarray*}
We proved in Eq. \eqref{Eq4prime} that $\bb{P}\paren{Z\geq t^-_n} = \int_{\bb{R}} \phi\paren{u}\bb{P}\paren{Z\geq \frac{t^-_n}{\sqrt{\rho}}-u\sqrt{\frac{1-\rho}{\rho}}} du$. Thus,
\begin{equation*}
\abs{p_0} 
\lesssim \sqrt{\frac{t^6_n}{n}} \int_{\bb{R}} \phi\paren{u} \frac{\bb{P}\paren{Z\geq \frac{t^-_n}{\sqrt{\rho}}-u\sqrt{\frac{1-\rho}{\rho}}}}{\bb{P}\paren{Z\geq t^-_n}} du = \sqrt{\frac{t^6_n}{n}} \rightarrow 0.
\end{equation*}
In summary we showed that both $p_0$ and $p_1$ tend to $0$ as $n\rightarrow\infty$, which concludes the proof of our claim in Eq. \eqref{Eq5}.
\end{proof}

\begin{cor}\label{GaussInnerProdUppBnd}
Under the same notation and conditions as in Theorem \ref{MainThmApp}, we have
\begin{equation*}
\bb{P}\paren{\abs{V_n\paren{r}} \geq t_n} \sim \bar{\Phi}\paren{t_n}\paren{1+o\paren{1}} \sim  \frac{1}{t_n\sqrt{2\pi}}\exp\paren{-\frac{t^2_n}{2}},\quad \mbox{as}\;n\rightarrow\infty.
\end{equation*}
\end{cor}

\end{appendix}

\section*{Acknowledgements}
The authors would like to thank Professor Yves Atchad\'{e} for his constructive comments that improved the
quality of this paper.

The second author is partially supported by NSF grants DMS-1545277, DMS-1632730 and NIH grant 1R01-GM1140201A1.

\bibliographystyle{abbrv}
\bibliography{Ref_Final}

\end{document}